\documentclass{article}
\usepackage {amsmath, amsfonts, amssymb, mathrsfs, amsthm, sectsty,hyphenat,enumitem,calc,url,color,array,tabu, stmaryrd}
\usepackage{palatino}
\usepackage{comment}
\usepackage[all]{xy}
\usepackage[nottoc]{tocbibind}
\usepackage[toc,page]{appendix}
\usepackage{tocloft}

\def\mf#1{\mathfrak{#1}}
\def\mc#1{\mathcal{#1}}
\def\mb#1{\mathbb{#1}}
\def\tx#1{\textrm{#1}}

\def\R{\mathbb{R}}
\def\P{\mathbb{P}}
\def\C{\mathbb{C}}
\def\Q{\mathbb{Q}}

\def\Z{\mathbb{Z}}

\def\lmod{\backslash}

\def\ul#1{\underline{#1}}

\def\Ad{\tx{Ad}}
\def\hat{\widehat}

\def\from{\leftarrow}

\def\lrw{\longrightarrow}
\def\llw{\longleftarrow}

\def\sm{\smallsetminus}

\def\<{\langle}
\def\>{\rangle}

\newenvironment{mytitle}
{\begin{center}\large\sc}
{\end{center}}

\newtheorem{thm}{Theorem}[subsection]
\newtheorem{lem}[thm]{Lemma}
\newtheorem{pro}[thm]{Proposition}
\newtheorem{cor}[thm]{Corollary}

\theoremstyle{definition}
\newtheorem{dfn}[thm]{Definition}
\newtheorem{rem}[thm]{Remark}

\sectionfont{\center\sc\normalsize}
\subsectionfont{\bf\normalsize}

\numberwithin{equation}{section}

\newcommand{\n}{\mathfrak n}
\newcommand{\s}{\mathfrak s}
\newcommand{\g}{\mathfrak g}

\renewcommand{\O}{\mathcal O}
\renewcommand{\P}{\mathcal P}
\newcommand{\K}{\mathcal K}

\newcommand{\w}{\mathfrak w}
\newcommand{\Cent}{\mathrm{Cent}}
\newcommand{\WF}{\mathrm{WF}}

\newcommand{\AC}{\mathrm{AC}}
\newcommand{\SL}{\mathrm{SL}}
\newcommand{\ch}[1]{\negthinspace\negthinspace\negthinspace\phantom{a}^\vee\negthinspace #1}

\begin{document}

\begin{mytitle} Discrete series $L$\-packets for real reductive groups \end{mytitle}
\begin{center} Jeffrey Adams and Tasho Kaletha \end{center}

{\let\thefootnote\relax\footnotetext{T.K. was supported in part by NSF grant DMS-2301507.}}

\begin{abstract}
We give a modern exposition of the construction, parameterization, and character relations for discrete series $L$\-packets of real reductive groups, which are fundamental results due to Langlands and Shelstad. This exposition incorporates recent developments not present in the original sources, such as normalized geometric transfer factors and the canonical double covers of tori and endoscopic groups, allowing for simpler statements and proofs. We also prove some new results, such as a simple criterion for detecting generic representations for a prescribed Whittaker datum, and an explicit formula for the factor $\Delta_I$ in terms of covers of tori.
\end{abstract}

\tableofcontents

\section{Introduction}

The construction of the discrete series of representations  for a real reductive group was a fundamental achievement of Harish-Chandra and a key step towards understanding the tempered representations and obtaining an explicit Plancherel formula. Motivated by the emerging theory of the trace formula, Langlands organized the tempered representations into packets \cite{Lan89} and Shelstad proved that these packets satisfy character relations with respect to endoscopic groups \cite{She82}. These fundamental results lie at the core of the stabilization of the spectral side of the Arthur-Selberg trace formula and are the cornerstone of many applications of the Langlands program to arithmetic questions. 

In this paper we give an exposition of these important results that is largely self-contained and incorporates a number of developments and insights that have occurred since the results were originally obtained. We focus on the case of an essentially discrete series representation (which we henceforth abbreviate to eds, see Definition \ref{dfn:eds}). Our motivation to do so is two-fold. First, the available expositions of the internal structure and character identities treat the general case of tempered representations, which is more complicated due to the reducibility of parabolic induction and the theory of the Knapp-Stein $R$-group. The case of eds representations on the other hand is more direct and is a good setting in which the basic ideas become visible. At the same time, it is sufficient for many important applications.

Second, a number of developments over the ensuing years have enabled a simplified construction of the $L$\-packets and their internal structure, a more natural statement of the character identities, and a more direct approach to their proof. These developments include the introduction of pure and rigid inner forms and the resulting normalizations of transfer factors \cite{Vog93}, \cite{KalECI}, \cite{KalRI}, the introduction of double covers of tori and the recognition of the role they play in the classification of eds representations and the construction of their packets \cite{AV92}, \cite{AV16}, \cite{KalDC}, and the introduction of double covers of endoscopic groups, which alleviate the need for ad-hoc constructions and technical work-arounds in the considerations of endoscopy \cite{KalHDC}. 

We utilize these developments to give a clean and direct account of the construction, internal structure, and character identities of discrete series $L$\-packets. This exposition is rather different from the classical literature and follows more closely the developments in the $p$-adic case such as \cite{KalRSP}, whose combination with double covers was hinted at in \cite{KalDC}. In particular, the statement of the character identities that we ultimately present and prove in this paper does not yet exist in the literature in this form. We do however additionally present an alternative statement, for which an indirect and more complicated proof was given in \cite{KalRI} based on the classical works of Langlands and Shelstad.

To improve readability of the paper, we have kept the main arguments short and to the point: the construction and internal structure of $L$-packets is treated in \S\ref{sec:cons} in about 7 pages, and the endoscopic transfer in \S\ref{sec:endo} in about 10 pages. At the same time, we have included a somewhat lengthy discussion of relevant background material in \S\ref{sec:recoll}, as well as a discussion of genericity in \S\ref{sec:gen} containing results that do not yet appear in the literature.

To say more precisely what we do and how it differs from classical results, we first review these results in the setting of discrete parameters. We refer to \S\ref{sec:recoll} for an overview of the concepts that will appear. Let $G$ be a connected reductive $\R$-group, $\hat G$ its complex dual group, and $^LG=\hat G \rtimes \Gamma$ its $L$\-group, where $\Gamma$ is the Galois group of $\C/\R$. A discrete Langlands parameter is an $L$\-homomorphism $W_\R \to {^LG}$ whose image is not contained in a proper parabolic subgroup of $^LG$, or equivalently for which the centralizer group $S_\varphi=\tx{Cent}(\varphi,\hat G)$ contains $Z(\hat G)^\Gamma$ with finite index; the finite group $S_\varphi/Z(\hat G)^\Gamma$ turns out to be an abelian $2$-group. Two such $\varphi$ are considered equivalent if they are $\hat G$-conjugate. To an equivalence class of $\varphi$ Langlands constructs in \cite[\S3]{Lan89} a finite set $\Pi_\varphi(G)$ of eds representations of $G(\R)$. This construction is a bit complicated due to the fact that Langlands extracts by hand the explicit data contained in $\varphi$ and uses it to specify the representations in $\Pi_\varphi(G)$. The finite sets $\Pi_\varphi(G)$ form a disjoint partition of the set of equivalence classes of irreducible eds representations of $G(\R)$.

In \cite{She82}, Shelstad established an injective map $\pi \mapsto \<\pi,-\>$ from $\Pi_\varphi(G)$ to the set of characters of the finite abelian group $S_\varphi/Z(\hat G)^\Gamma$. This map depends on some auxiliary choices. Shelstad then proceeded to prove the following result. Let $(H,s,\mc{H},\eta)$ be an endoscopic datum with $\eta(s) \in S_\varphi$. Assume that there exists, and fix, an $L$\-isomorphism $^LH \to \mc{H}$ and let $^L\eta : {^LH} \to {^LG}$ denote its composition with $\eta : \mc{H} \to {^LG}$. We note that such an $L$\-isomorphism need not exist in general. By construction $\varphi$ factors through $^L\eta$ as $\varphi = {^L\eta}\circ\varphi'$ for a necessarily discrete parameter $\varphi' : W_\R \to {^LH}$. Let $\Pi_{\varphi'}(H)$ be the associated $L$\-packet on $H(\R)$. Shelstad proved the existence of a function $\Delta : H(\R)_\tx{sr} \times G(\R)_\tx{sr} \to \C$, which is implicitly defined and unique up multiplication by the constant $-1$, and shows that there exists a constant $c(\Delta)$ depending on $\Delta$ such that for all $\delta \in G(\R)_\tx{sr}$ the following character identity holds
\begin{equation} \label{eq:s1}
c(\Delta)\sum_{\pi \in \Pi_\varphi(G)}\<\pi,s\>\Theta_\pi(\delta) = \sum_\gamma \Delta(\gamma,\delta)\sum_{\sigma \in \Pi_{\varphi'}(H)} \Theta_{\sigma}(\gamma),
\end{equation}
where $\gamma$ runs over the set of (representatives for) the stable conjugacy classes of strongly regular semi-simple elements of $H(\R)$, and $\Theta_\pi$ and $\Theta_\sigma$ are the Harish-Chandra character functions of the representations $\pi$ and $\sigma$, respectively. The proof of this identity is rather complicated due to the implicit nature of the functions $\pi \mapsto \<\pi,-\>$ and $\Delta$, and the constant $c(\Delta)$.

In \cite{LS87}, Langlands and Shelstad succeeded in making the functions $\Delta$ explicit. More precisely, they explicitly defined such functions for all local fields of characteristic zero, in such a way that they satisfy a global product formula related to the stabilization of the Arthur-Selberg trace formula. These functions still remained ambiguous up to multiplication by a non-zero complex root of unity (of order $2$ over $\R$). In a subsequent paper \cite[Theorem 2.6.A]{LS90} they proved that these explicitly defined functions must coincide up to scalar with those implicitly defined by Shelstad earlier. This result was indirect and relied on Shelstad's identity \eqref{eq:s1}. This prompted Arthur to pose the question \cite{Art08} of rederiving Shelstad's identity \eqref{eq:s1} directly in terms of the factors of \cite{LS87}, and without using the arguments of \cite{LS90} and the implicit work of \cite{She82}. This was undertaken by Shelstad in her papers \cite{SheTE1} and \cite{SheTE2}. There she explicitly  defined the so called ``spectral transfer factors'' $\Delta(\sigma,\pi)$ and ``compatibility factors'' $\Delta(\sigma,\pi,\gamma,\delta)$, where $\sigma$ and $\pi$ are eds representations of $H(\R)$ and $G(\R)$ respectively, $\gamma \in H(\R)_\tx{sr}$, $\delta \in G(\R)_\tx{sr}$. The spectral transfer factor is, just like the geometric transfer factor $\Delta(\gamma,\delta)$, well-defined only up to a non-zero complex scalar multiple. The compatibility factor $\Delta(\sigma,\pi,\gamma,\delta)$ is canonical. Its purpose is to link the arbitrary choices for normalizations of geometric and spectral factors, by saying that a chosen normalization of $\Delta(\gamma,\delta)$ is compatible with a chosen normalization of $\Delta(\sigma,\pi)$ if the identity
\[ \Delta(\sigma,\pi)  = \Delta(\sigma,\pi,\gamma,\delta) \cdot \Delta(\gamma,\delta) \]
holds for all possible $\sigma,\pi,\gamma,\delta$, see \cite[\S4]{SheTE2}. Just like the variable $\gamma$ matters only up to stable conjugacy, the variable $\sigma$ matters only up to $L$\-packets, so we may write $\Delta(\varphi',\pi)$ in place of $\Delta(\sigma,\pi)$ whenever $\sigma \in \Pi_{\varphi'}(H)$. With this notation, Shelstad proved the following modern version of her classical result
\begin{equation} \label{eq:s2}
	\sum_{\pi \in \Pi_\varphi(G)}\Delta(\varphi',\pi)\Theta_\pi(\delta) = \sum_\gamma \Delta(\gamma,\delta)\sum_{\sigma \in \Pi_{\varphi'}(H)} \Theta_{\sigma}(\gamma).
\end{equation}
The advantage of this result is that the terms $\Delta(\gamma,\delta)$ and $\Delta(\varphi',\pi)$ are now explicitly constructed. The relationship between the spectral factor $\Delta(\varphi',\pi)$ and the character $\<\pi,s\>$ in the previous formulation is
\[ \frac{\Delta(\varphi',\pi_1)}{\Delta(\varphi',\pi_2)} = \frac{\<\pi_1,s\>}{\<\pi_2,s\>}.\]
Note that both fractions are well-defined, since the ambiguity of all objects cancels in the fractions. This relative identity is weaker than what is desired: a version of \eqref{eq:s1} without implicit constants and arbitrary choices, or equivalently a version of \eqref{eq:s2} but with the spectral transfer factor $\Delta(\varphi',\pi)$ replaced by a pairing $\<\pi,s\>$ between $\Pi_\varphi(G)$ and a suitable version of $S_\varphi$.

The next step towards that goal was given by the ``generic packet conjecture'' formulated by Shahidi in \cite{Sha90}, which in its strong form states that any tempered $L$\-packet containes a unique member that is generic with respect to a fixed Whittaker datum $\mf{w}$. The validity of this conjecture for real groups can be extracted from the work of Kostant \cite{Kos78} and Vogan \cite{Vog78}. This prompted Kottwitz and Shelstad to single out \cite[\S5.3]{KS99} a normalization of the geometric factor $\Delta(\gamma,\delta)$ depending on $\mf{w}$ when $G$ is a quasi-split group. Shelstad was able to prove \cite[Theorem 11.5]{SheTE3} that the compatibly normalized spectral transfer factor $\Delta(\varphi,\pi_\mf{w})$ is equal to $1$ when $\pi_\mf{w}$ is that unique generic constituent. If we normalize the pairing $\<\pi_\mf{w},s\>$ to equal $1$ for that same $\pi_\mf{w}$, then for general $\pi \in \Pi_\varphi(G)$ one has
\[ \Delta(\varphi',\pi) = \<\pi,s\> \]
and this finally leads to the clean formulation
\begin{equation} \label{eq:s3}
	\sum_{\pi \in \Pi_\varphi(G)}\<\pi,s\>\Theta_\pi(\delta) = \sum_\gamma \Delta(\gamma,\delta)\sum_{\sigma \in \Pi_{\varphi'}(H)} \Theta_{\sigma}(\gamma)
\end{equation}
of the character identities, where $\pi \mapsto \<\pi,-\>$ is an injection of $\Pi_\varphi(G)$ into $(S_\varphi/Z(\hat G)^\Gamma)^*$, uniquely determined to make the above identity true, and $\Delta(\gamma,\delta)$ is the Whittaker normalization of the geometric transfer factor.

The main problem with \eqref{eq:s3} was that it was only available for quasi-split groups (slightly more generally, quasi-split $K$-groups), because other groups lack the concept of a Whittaker datum. For such more general groups, Statement \eqref{eq:s2} was still the only option. 

The ideas that led to the resolution of this problem originated in the work of Adams-Barbasch-Vogan \cite{ABV92}. There the authors showed that the theory becomes more balanced if one considers all groups in a given inner class together. However, one has to rigidify the concept of an inner form appropriately. An easy, but incomplete, way to do this for general fields was introduced in \cite{ABV92} over $\R$ and in \cite{Vog93} over general local fields, under the name ``pure inner form''. It was shown in \cite[\S2.2]{KalECI} that pure inner forms can be used to normalize transfer factors beyond the quasi-split case. To go beyond pure inner forms, the concept of strong rational form was introduced in \cite{ABV92} over $\R$, but the associated methods didn't generalize to other local fields, and didn't appear to have a connection to transfer factors. The ideas needed to resolve this problem were based on the notion of Galois gerbe from \cite{LR87}, and began taking shape in Kottwitz's study \cite{Kot97} of isocrystals with additional structure. They led to the construction of rigid inner forms and to the normalizations of transfer factors for all groups in \cite{KalRI} for characteristic zero local fields, and in \cite{Dillery20} for positive characteristic local fields. It was shown in \cite[\S5.2]{KalRI} and that over the base field $\R$ the concepts of rigid inner forms and strong rational forms coincide. Moreover, it was shown in \cite[\S5.6]{KalRI} that a suitable formulation of \eqref{eq:s3} holds for all real groups. 

More precisely, there is a bijection $\pi \mapsto \<\pi,-\>$ between the union of the packets $\Pi_\varphi(G)$ as $G$ varies over all rigid inner forms of a given group, and $\pi_0(S_\varphi^+)^*$, where $S_\varphi^+$ is the preimage of $S_\varphi$ in the universal cover of $\hat G$. This bijection is normalized so that $\<\pi_\mf{w},-\>=1$ when $\pi_\mf{w}$ is the unique generic member in the $L$\-packet for $\varphi$ on the unique quasi-split inner form. Then \eqref{eq:s3} hold for all $G$, provided $\Delta(\gamma,\delta)$ has been normalized using both the Whittaker datum $\mf{w}$ and the rigid inner form datum, as discussed in \cite[\S5.3]{KalRI}. This claim was proved in \cite[\S5.6]{KalRI} for all tempered $\varphi$, not just the discrete ones, using the results of Shelstad from \cite{SheTE2} and \cite{SheTE3}. Thus, while the statement was clean, the proof was again roundabout and involved the construction and study of objects, such as $\Delta(\varphi',\pi)$ and $\Delta(\sigma,\pi,\gamma,\delta)$, that were not needed for the statement. Another issue with the proof was that it relied on \cite[Theorem 11.5]{SheTE3}, whose proof was not direct, but rather involved a series of reductions steps to the case of $\tx{SL}_2(\R)$, which was then left to the reader as a ``well-known computation'', but, at least to the authors of this article, that computation was not known in any form other than the result for general groups that we will formulate and prove in this note, see Proposition \ref{p:whittaker} and Lemma \ref{lem:gen}.

The final issue in the classical set-up that we want to discuss is that of the assumed existence of an $L$\-embedding $^LH \to {^LG}$. In general such an $L$\-embedding does not exist. Instead, one is given an endoscopic datum $(H,s,\mc{H},\eta)$, where $\eta$ is an $L$\-embedding $\mc{H} \to {^LG}$, but there is generally no $L$\-isomorphism $^LH \to \mc{H}$. In the classical set-up one makes an arbitrary choice of a $z$-extension $H_1 \to H$ and an $L$\-embedding $\mc{H} \to {^LH_1}$ and then works not with representations $\sigma$ of $H(\R)$, but rather with representations $\sigma_1$ of $H_1(\R)$ that transform under a certain non-trivial character of the kernel $H_1(\R) \to H(\R)$. This brings yet another amount of ambiguity and complication into the theory.

Work of Adams and Vogan \cite{AV92} introduced the idea that groups of the form $\mc{H}$ should be understood as $L$\-groups of topological covers of $H(\R)$. In \cite{KalHDC} it was shown that, in the setting of an endoscopic datum $(H,s,\mc{H},\eta)$, there is a natural double cover $H(\R)_\pm \to H(\R)$ whose $L$\-group is canonically identified with $\mc{H}$. In other words, there is a canonical $L$\-embedding $^LH_\pm \to {^LG}$. Therefore, instead of assuming the existence of an  (arbitrary) $L$\-embedding $^LH \to {^LG}$ and working with the $L$\-packet $\Pi_{\varphi'}(H)$, or choosing a $z$-extension $H_1 \to H$ and an (again arbitrary) $L$\-embedding $\mc{H} \to {^LH_1}$ and working with the $L$\-packet $\Pi_{\varphi_1}(H_1)$, one ought to consider an $L$\-packet $\Pi_{\varphi'}(H_\pm)$ consisting of genuine representations of $H(\R)_\pm$, where $\varphi' : W_\R \to {^LH_\pm}$ is the factorization of $\varphi$ through the canonical $L$\-embedding $^LH_\pm \to {^LG}$. Here the word ``genuine'' means that the non-trivial element in the kernel of $H(\R)_\pm \to H(\R)$ acts on the representation by the scalar $-1$. This allows the identity \eqref{eq:s3} to be stated without the assumption of an existence of an $L$\-embedding $^LH \to {^LG}$, and without an arbitrary choice of a $z$-extension. At this point, the identity \eqref{eq:s3} becomes fully canonical and does not depend on any choices. It is also worth pointing out that the  explicit construction given in \cite{LS87} of the transfer factor $\Delta$ is rather complicated and involves multiple factors, each of which depends on auxiliary data. In contrast, the construction of $\Delta$ in the setting of the canonical cover $H(\R)_\pm$, which is given in \cite[\S4.3]{KalHDC}, simplifies noticeably and becomes the product of two natural invariants, one of which has a close relationship to the simple form of the transfer factor for Lie algebras derived by Kottwitz in \cite{Kot99}. Moreover, the entire factor attains the pleasant property of being a locally constant function on the set of pairs of regular semi-simple elements, whose values are complex roots of unity of bounded order. This simplifies the analytic arguments.

After this overview of the classical statements and constructions and the historical development of ideas, let us state the form in which we formulate and prove the main result, and briefly describe the approach taken in this note. 

We spend about the first half of the note to give in \S\ref{sec:recoll} a review of much of the background that is needed for the main argument, in order to make the note as self-contained as possible. Then, in \S\ref{sec:gen}, we discuss how to detect that an eds representation is generic for a particular fixed Whittaker datum. This goes beyond the classical work of Kostant \cite{Kos78} and Vogan \cite{Vog78}, which treats the question of when a representation is generic with respect to \emph{some} Whittaker datum. The main result of that section, Proposition \ref{p:whittaker}, describes a simple answer to this question that is entirely analogous to the answer for supercuspidal representations of $p$-adic groups (originally proved by DeBacker-Reeder in a special case \cite{DR10} and shown to hold much more generally in \cite[Lemma 6.2.2]{KalRSP}). As far as we know this result has not yet been recorded in the literature. It allows us to give a direct and explicit proof of Shelstad's result \cite[Theorem 11.5]{SheTE3}.

In \S\ref{sec:cons} we give the construction of the $L$\-packet for a discrete series parameter. Here we systematically use the device of admissible embeddings of tori into reductive groups, and the existence of a canonical $L$\-embedding from the $L$\-group of the natural double cover of a maximal torus into the $L$\-group of $G$. This provides a canonical factorization of the $L$\-parameter into the $L$\-group of that double cover, hence a genuine character of that double cover. From the genuine character we can explicitly produce a function on the set of regular elements of that torus, and basic results of Harish-Chandra provide the desired eds representation. The internal structure of the resulting $L$\-packet is a simple consequence of Tate-Nakayama duality. This approach is very conceptual and avoids any explicit calculations using cocharacters and Galois cohomology. It is also parallel to the constructions employed in the $p$-adic case, such as in \cite{KalRSP}.

The main result of \S\ref{sec:cons} is the following. Let $G_0$ be a quasi-split connected reductive $\R$-group with dual group $G$ and $L$-group $^LG$. For each discrete $L$-parameter $\varphi : W_\R \to {^LG}$ we explicitly construct the compound $L$-packet $\Pi_\varphi$, consisting of eds representations of pure (resp. rigid) inner twists of $G_0$. Given a Whittaker datum $\mf{w}$ for $G_0$, we construct a bijection $\iota_\mf{w} : \Pi_\varphi \to \pi_0(S_\varphi)^*$ in the pure case, and $\iota_\mf{w} : \Pi_\varphi \to \pi_0(S_\varphi^+)^*$ in the rigid case, fitting into the commutative diagrams
\[ \xymatrix{
	\Pi_\varphi\ar[r]^{\iota_\mf{w}}\ar[d]&\pi_0(S_\varphi)^*\ar[d]\\
	H^1(\R,G_0)\ar[r]&\pi_0(Z(\hat G)^\Gamma)^*
}\quad\tx{resp.}\quad
\xymatrix{
	\Pi_\varphi\ar[r]^{\iota_\mf{w}}\ar[d]&\pi_0(S_\varphi^+)^*\ar[d]\\
	H^1_\tx{bas}(\mc{E}_\R,G_0)\ar[r]&\pi_0(Z(\hat{\bar G})^+)^*
}.
\]
While rigid inner forms are more general than pure inner forms, and hence subsume the latter, we have taken in this note the expositional approach of treating the two cases side-by-side, and sometimes even treating only the case of pure inner forms. The reason is that this simplifies the language, and the argument in the rigid case is exactly the same as in the pure case.

If we are given a pure or rigid inner twist $(\xi,z) : G_0 \to G$  of $G_0$, then the fiber of the left vertical map in the appropriate diagram over the isomorphism class of $(\xi,z)$ is the $L$-packet $\Pi_\varphi(G)$ consisting of eds representations of the group $G(\R)$, and the restriction of $\iota_\mf{w}$ to that subset provides the internal structure of that particular $L$-packet. Note that every connected reductive $\R$-group $G$ arises this way, provided we use rigid inner forms.

In \S\ref{sec:endo} we prove the endoscopic character identities. Our proof is direct and does not involve considerations of spectral transfer factors or compatibility factors, and does not appeal to prior expositions. Instead, we work directly with the factor $\Delta(\gamma,\delta)$, which in the setting of the canonical double cover $H(\R)_\pm$ is the product of two natural invariants. We show how one of these invariants simply measures the difference between the factorizations of $\varphi$ and $\varphi'$ through the $L$\-groups of the natural double covers of the elliptic tori, while the other invariant detects which member of $\Pi_\varphi(G)$ is generic. In this way, we show that the factor $\Delta(\gamma,\delta)$ plays both a geometric and a spectral role. For the convenience of the reader, we give two versions of the character identities, one using the language of covers, and one using the classical language. For each of those, we give a statement using distributions (Theorem \ref{thm:main1}), and using the language of functions (Theorem \ref{thm:main2}). We show that these two statements are equivalent (Lemma \ref{lem:equi}) using the Weyl integration formula and its stable analog; this equivalence is stated and proved for all local fields of characteristic zero. We then begin the proof of Theorem \ref{thm:main2} and first reduce it to the case of simply connected semi-simple groups and elliptic elements (Lemma \ref{lem:redell}), relying on Harish-Chandra's results about uniqueness of invariant eigendistributions. The treatment of this essential case rests on an analysis of the structure of the transfer factor and Harish-Chandra's character formula for discrete series representations. As part of the argument, we derive a new formula for the term $\Delta_I$ of the transfer factor in terms of covers of tori, which is purely Lie-theoretic and does not appeal to Galois cohomology, see Proposition \ref{pro:magic}.

We now state the character identities of Theorem \ref{thm:main2}. Keeping the already established notation of the quasi-split group $G_0$, the discrete parameter $\varphi : W_\R \to {^LG}$, the Whittaker datum $\mf{w}$ for $G_0$, and the rigid inner twist $(\xi,z) : G_0 \to G$, we now take in addition a refined endoscopic datum $(H,s,\mc{H},\eta)$ for $G$, equivalently for $G_0$, such that $\eta(s) \in S_\varphi$ in the pure and $\eta(s) \in S_\varphi^+$ in the rigid case. This implies that $\varphi$ factors through the inclusion $\eta : \mc{H} \to {^LG}$. There is a canonical topological double cover $H(\R)_\pm$ of $H(\R)$ whose $L$-group $^LH_\pm$, as defined in \cite[\S2.6]{KalHDC}, comed equipped with a canonical isomorphism $^LH_\pm \to \mc{H}$ according to \cite[\S3.1]{KalHDC}. Writing $^L\eta$ for the composition of $\eta : \mc{H} \to {^LG}$ with this canonical isomorphism we obtain a discrete $L$-parameter $\varphi' : W_\R \to {^LH_\pm}$ satisfying $\varphi = {^L\eta}\circ\varphi'$. The construction of \S\ref{sub:packetcover} provides an $L$-packet of genuine eds representations $\Pi_{\varphi'}(H)$ on $H(\R)_\pm$ and Theorem \ref{thm:main2}(1) asserts the identity
\begin{equation} \label{eq:s4}
	e(G)\sum_{\pi \in \Pi_\varphi(G)}\<\pi,s\>\Theta_\pi(\delta) = \sum_\gamma \Delta(\dot\gamma,\delta)\sum_{\sigma \in \Pi_{\varphi'}(H_\pm)} \Theta_{\sigma}(\dot\gamma),
\end{equation}
where $e(G)$ is the Kottwitz sign of $G$, defined in \cite{Kot83}, $\<\pi,s\>=\iota(\pi)(s)$, $\Delta$ is the transfer factor constructed in \cite[\S4.3]{KalHDC}, and $\gamma$ runs over a set of representatives for the stable classes of strongly regular semi-simple elements of $H(\R)$, while $\dot \gamma \in H(\R)_\pm$ is an arbitrary lift of $\gamma$. Note that both $\Delta(\dot\gamma,\delta)$ and $\Theta_{\sigma}(\dot\gamma)$ are genuine functions of $\dot\gamma$, so their product only depends on $\gamma$.

For readers who prefer to avoid the language of covers, we also give the following alternative formulation. Choose a $z$-extension $H_1 \to H$ and an $L$-embedding ${^L\eta_1} : \mc{H} \to {^LH_1}$; such a choice is always available \cite[\S2.2]{KS99}. Since $\varphi$ factors through $\eta$ we can let $\varphi_1$ be the composition of $\varphi$ with $^L\eta_1$. This is a discrete parameter for $H_1$ and we have the eds $L$-packet $\Pi_{\varphi_1}(H_1)$ for the group $H_1(\R)$. Theorem \ref{thm:main2}(2) asserts the identity
\begin{equation} \label{eq:s5}
	e(G)\sum_{\pi \in \Pi_\varphi(G)}\<\pi,s\>\Theta_\pi(\delta) = \sum_\gamma \Delta(\gamma_1,\delta)\sum_{\sigma_1 \in \Pi_{\varphi_1}(H_1)} \Theta_{\sigma_1}(\gamma_1),
\end{equation}
Here again $\gamma$ runs over a set of representatives for the stable classes of strongly regular semi-simple elements of $H(\R)$, while $\gamma_1 \in H_1(\R)$ is an arbitrary lift of $\gamma$ under the surjective map $H_1(\R) \to H(\R)$. We note that there is a character $\lambda$ of $K(\R) = \tx{ker}(H_1(\R) \to H(\R))$, depending on the choices of $H_1$ and $^L\eta_1$, such that all representations in $\Pi_{\varphi_1}(H_1)$ transform under $K(\R)$ via $\lambda$, while $\Delta(\gamma_1,\delta)$ transforms under $K(\R)$ via $\lambda^{-1}$ in the first variable. Therefore the right hand side again depends only on $\gamma$ and not on the lift $\gamma_1$.

We end this note with an apology about its length. Our goal was to give an exposition of the construction of eds $L$-packets and their character identities that is as simple and short as possible. This is achieved in sections 4 resp. 5 in the span of 7 reps. 11 pages. But in order to make the note readable to mathematicians in neighboring fields, we decided to include a lengthy introduction and review sections, as well as a careful treatment of genericity, which became ultimately responsible for a much longer text. 

We thank Eric Opdam for some helpful advice.

\section{Recollections} \label{sec:recoll}

\subsection{The $L$\-group}

We review the $L$\-group of a connected reductive group following \cite[\S2]{Vog93}.

Let $F$ be a field. Assume first that $F$ is separably closed. Let $G$ be a connected reductive $F$-group. Given a Borel pair $(T,B)$ of $G$ one has the based root datum $\tx{brd}(T,B,G)=(X^*(T),\Delta,X_*(T),\Delta^\vee)$, where $\Delta \subset X^*(T)$ is the set of $B$-simple roots for the adjoint action of $T$ on $\tx{Lie}(G)$, and $\Delta^\vee \subset X_*(T)$ are the corresponding coroots. For a second Borel pair $(T',B')$, there is a unique element of $T'(F)\lmod G(F)/T(F)$ that conjugates $(T,B)$ to $(T',B')$. This element provides an isomorphism $\tx{brd}(T,B,G) \to \tx{brd}(T',B',G)$. This procedure leads to a system of based root data and isomorphisms, indexed by the set of Borel pairs of $G$. The limit of that system is the based root datum $\tx{brd}(G)$ of $G$.

One can formalize the notion of a based root datum: we refer the reader to \cite[\S7.4]{Spr98} for the formal notion of a root datum, to which one has to add a set of simple roots to obtain the formal notion of a based root datum. Based root data can be placed into a category, in which all morphisms are isomorphisms, for the evident notion of isomorphism of based root data. The classification of connected reductive $F$-groups \cite[Theorem 9.6.2, Theorem 10.1.1]{Spr98} can be stated as saying that $G \mapsto \tx{brd}(G)$ is a full essentially surjective functor from the category of connected reductive $F$-groups and isomorphisms to the category of based root data and isomorphisms. Moreover, two morphisms lie in the same fiber of this functor if and only if they differ by an inner automorphism.

Consider now a general field $F$, let $F^s$ a separable closure, $\Gamma=\tx{Gal}(F^s/F)$ the Galois group. Given a connected reductive $F$-group $G$, there is a natural action of $\Gamma$ on the set of Borel pairs of $G_{F^s}$, and this leads to a natural action of $\Gamma$ on $\tx{brd}(G_{F^s})$. We denote by $\tx{brd}(G)$ the based root datum $\tx{brd}(G_{F^s})$ equipped with this $\Gamma$-action. Given two connected reductive $F$-groups $G_1,G_2$, an isomorphism $\xi : G_{1,F^s} \to G_{2,F^s}$ is called an \emph{inner twist}, if $\xi^{-1}\circ\sigma\circ\xi\circ\sigma^{-1}$ is an inner automorphism of $G_{1,F^s}$ for all $\sigma \in \Gamma$. The two groups $G_1,G_2$ are then called inner forms of each other. The functor $G \mapsto \tx{brd}(G)$ from the category of connected reductive $F$-groups to the category of based root data over $F$ and isomorphisms is again essentially surjective. It maps inner twists to isomorphisms, and two inner twists map to the same isomorphism if they differ by an inner automorphism. The fiber over a given based root datum over $F$ consists of all reductive groups that are inner forms of each other.

Given a based root datum $(X,\Delta,Y,\Delta^\vee)$ over $F$, its dual $(Y,\Delta^\vee,X,\Delta)$ is also a based rood datum over $F$. If $G$ is a connected reductive $F$-group with based root datum $(X,\Delta,Y,\Delta^\vee)$, its dual $\hat G$ is the unique split connected reductive group defined over a chosen base field (we will work with $\C$) with based root datum $(Y,\Delta^\vee,X,\Delta)$. Thus, given a Borel pair $(\hat T,\hat B)$ of $\hat G$ and a Borel pair $(T,B)$ of $G_{F^s}$, one is given an identification $X_*(\hat T)=X^*(T)$ that identifies the Weyl chambers associated to $\hat B$ and $B$.

To form the $L$\-group, one chooses a pinning $(\hat T,\hat B,\{Y_\alpha\})$ of $\hat G$. The group of automorphisms of $\hat G$ that preserve this pinning is in natural isomorphism with the group of automorphisms of $\tx{brd}(\hat G)$, hence with that of $\tx{brd}(G)$. The $\Gamma$-action on $\tx{brd}(G)$ then lifts to an action on $\hat G$ by algebraic automorphisms, and $^LG=\hat G \rtimes \Gamma$.

When $G$ is quasi-split, $(T,B)$ is an $F$-Borel pair, and $(\hat T,\hat B)$ is a $\Gamma$-stable Borel pair of $\hat G$, then the identification $X_*(T)=X^*(\hat T)$ is $\Gamma$-equivariant.

\subsection{Inner forms}
\label{s:inner}

Let $G$ be a connected reductive $F$-group. An \emph{inner twist} of $G$ is a pair $(G_1,\xi)$ where $G_1$ is a connected reductive $F$-group and $\xi : G_{F^s} \to G_{1,F^s}$ is an isomorphism, such that for each $\sigma \in \Gamma$ the automorphism $\xi^{-1}\sigma(\xi) = \xi^{-1}\circ\sigma\circ\xi\circ\sigma^{-1}$ of $G_{F^s}$ is inner. An isomorphism of inner twists $(G_1,\xi_1) \to (G_2,\xi_2)$ is a homomorphism $f : G_1 \to G_2$ of $F$-groups such that $\xi_2^{-1}\circ f \circ \xi_1$ is an inner automorphism of $G_{F^s}$. From an inner twist $(G_1,\xi)$ we obtain the function $\Gamma \to G/Z(G)$ given by $\sigma \mapsto \xi^{-1}\sigma(x)$. It is an element of $Z^1(F,G/Z(G))$, whose cohomology class depends only on the isomorphism class of $(G_1,\xi)$. In this way the set of isomorphism classes of inner twists is in bijection with $H^1(F,G/Z(G))$.

Following Vogan \cite{Vog93}, a \emph{pure inner twist} of $G$ is a triple $(G_1,\xi,z)$, where $G_1$ is a connected reductive $F$-group, $\xi : G_{F^s} \to G_{1,F^s}$ is an isomorphism and $z \in Z^1(\Gamma,G)$, subject to $\xi^{-1}\sigma(\xi)=\tx{Ad}(\bar z_\sigma)$, where $\bar z \in Z^1(F,G/Z(G))$ is the image of $z$ under the natural projection $G \to G/Z(G)$. An isomorphism of pure inner twists $(G_1,\xi_1,z_1) \to (G_2,\xi_2,z_2)$ is a pair $(f,g)$ consisting of an isomorphism $f : G_1 \to G_2$ of $F$-groups and $g \in G_0(F^s)$ such that $\xi_2^{-1}\circ f \circ \xi_1=\tx{Ad}(g)$ and $z_2(\sigma)=gz_1(\sigma)\sigma(g)^{-1}$ for all $\sigma \in \Gamma$. The map $(G_1,\xi,z) \mapsto z$ induces a bijection from the set of isomorphism classes of pure inner twists to $H^1(F,G)$.

There is a lighter notation in which inner twists and pure inner twists can be recorded. It is grounded on the fact that, if $(G_1,\xi)$ is an inner twist of $G$, with $\bar z \in Z^1(F,G/Z(G))$ given by $\bar z(\sigma)=\xi^{-1}\sigma(\xi)$, and we let $G_{\bar z}$ be the algebraic $F$-group obtained from $G$ by twisting the $F$-structure by $\bar z$, then $\xi$ becomes an isomorphism of $F$-groups $G_{\bar z} \to G_1$, in fact an isomorphism of inner twists $(G_{\bar z},\tx{id}) \to (G_1,\xi)$. Therefore the map $\bar z \mapsto (G_{\bar z},\tx{id})$ induces an injection from $Z^1(F,G/Z(G))$ into the class of all inner twists of $G$ which meets every isomorphism class. Two elements of $Z^1(F,G/Z(G))$ map into the same isomorphism class if and only if they lie in the same orbit for the action of $G(F^s)$ on $Z^1(F,G/Z(G))$ given by $(g \cdot z)(\sigma) = gz(\sigma)\sigma(g)^{-1}$. The orbits space for this action is $H^1(F,G/Z(G))$, and this gives the same identification between that orbit space and the set of isomorphism classes of inner twists as above.

The same simplification applies to the notion of pure inner twist. There we work with the set $Z^1(F,G)$ and obtain the embedding $z \mapsto (G_z,\tx{id},z)$ from that set into the class of pure inner twists of $G$. Again the orbit space for the action of $G(F^s)$ on $Z^1(F,G)$ by $(g\cdot z)(\sigma)=gz(\sigma)\sigma(g)^{-1}$ equals $H^1(F,G)$ and is identified with the set of isomorphism classes of pure inner twists.

We now take $F=\R$, hence $F^s=\C$. Then $\Gamma=\{1,c\}$, where $c$ is complex conjugation. An element of $Z^1(F,G)$ can be identified with the image of $c$, which is an element of $G(\C)$. We can also think of this as an element of the coset $G(\C) \rtimes c \subset G(\C) \rtimes \Gamma$. The elements of $Z^1(F,G)$ correspond precisely to the elements of $G(\C) \rtimes c$ of order $2$.

In \cite{ABV92} the notion of strong real form was introduced. This is an element $\delta \in G(\C) \rtimes c$, such that $\delta^2$ is a finite order element of $Z(G)(\C)$. In \cite{KalRI} this notion was given a cohomological interpretation and was extended to non-archimedean local fields. Following the latter reference, a \emph{rigid inner twist} of $G$ is a triple $(G_1,\xi,z)$, where $G_1$ is a connected reductive $\R$-group, $\xi : G_\C \to G_{1,\C}$ is an isomorphism and $z \in Z^1_\tx{bas}(\mc{E}_\R,G)$, subject to $\xi^{-1}\sigma(\xi)=\tx{Ad}(\bar z_\sigma)$. Here $1 \to u_{\R}(\C) \to \mc{E}_\R \to \Gamma \to 1$ is a certain extension of the Galois group, $Z^1_\tx{bas}(\mc{E}_\R,G)$ denotes the group of continuous 1-cocycles $\mc{E}_\R \to G(\C)$ whose restriction to $u_\R(\C)$ takes values in $Z(G)(\C)$, and $\bar z$ is again the image of $z$ under the natural projection map $G \to G/Z(G)$; it factors through the quotient $\mc{E}_\R \to \Gamma$. An isomorphism of pure inner twists $(G_1,\xi_1,z_1) \to (G_2,\xi_2,z_2)$ is a pair $(f,g)$ consisting of an isomorphism $f : G_1 \to G_2$ of $\R$-groups and $g \in G_0(\C)$ such that $\xi_2^{-1}\circ f \circ \xi_1=\tx{Ad}(g)$ and $z_2(e)=gz_1(e)\sigma_e(g)^{-1}$ for all $e \in \mc{E}_\R$, where $\sigma_e \in \Gamma$ is the image of $e$. The set of isomorphism classes of rigid inner twists is in bijection with $H^1_\tx{bas}(\mc{E}_\R,G)$. It is shown in \cite[\S5.2]{KalRI} that an element of $G(\C) \rtimes c$ whose square is a finite order element of $Z(G)(\C)$ naturally gives an element of $Z^1_\tx{bas}(\mc{E}_\R,G)$, and this sets up a bijection between the set of $G(\C)$-conjugacy classes of the former and the set of cohomology classes of the latter.

Each of these three notions partitions the set of isomorphism classes of connected reductive $\R$-groups into equivalence classes. The notion of inner twist is the most classical, stemming from the classification of reductive groups, and each equivalence class has exactly one quasi-split member. Unfortunately, this notion is not rigid enough for representation theory -- the group of automorphisms of an inner twist $(G_1,\xi)$ is $(G_1/Z(G_1))(\R)$, and the conjugation by such an element can be an outer automorphism of the Lie group $G_1(\R)$.

The notion of pure inner twist is sufficiently rigid -- the group of automorphisms of a pure inner twist $(G_1,\xi,z)$ is $G_1(\R)$, hence acts by inner automorphisms of $G_1(\R)$. Unfortunately, the equivalence classes induced by this notion are generally smaller, and not all of them contain a quasi-split member.

The notion of a rigid inner twist is again sufficiently rigid, and in addition the equivalence classes in induces coincide with those induced by the notion of inner twist. This will be the notion that we will use.

For further discussion of this topic we refer the reader to \cite{Vog93} and \cite{KalSimons}.

Given an inner twist $\xi : G \to G_1$, the isomorphism $\xi$ induces an isomorphism $\tx{brd}(G) \to \tx{brd}(G_1)$, which is $\Gamma$-equivariant, even though $\xi$ itself is not. This leads to an identification of dual groups $\hat G = \hat G_1$ and $L$\-groups $^LG = {^LG}_1$.

\subsection{Admissible embeddings of tori} \label{sub:adm}

Let $F$ be a field and let $G$ be a connected reductive $F$-group. Let $\hat G$ be its dual group, defined over any base field, which we take to be $\C$, equipped with a $\Gamma$-action.

Let $S$ be an $F$-torus. We recall from \cite[\S5.1]{KalRSP} that, given a $\Gamma$-stable $\hat G$-conjugacy class $\hat J$ of embeddings $\hat S \to \hat G$ whose images are maximal tori, there is an associated $\Gamma$-stable $G(F^s)$-conjugacy class $J$ of embeddings $S \to G$. 

To obtain $J$, we first assume that $G$ is quasi-split. Fix $\Gamma$-stable Borel pairs $(\hat T,\hat B)$ and $(T,B)$ in $\hat G$ and $G$, respectively. Thus we have the identifications $X_*(T)=X^*(\hat T)$ and $\Omega(T,G)=\Omega(\hat T,\hat G)$ as $\Gamma$-modules. There is $\hat\jmath \in \hat J$ with image $\hat T$ and we obtain the isomorphism $X_*(T)=X^*(\hat T) \to X^*(\hat S)=X_*(S)$, hence an isomorphism $j : S_{F^s} \to T_{F^s}$. We let $J$ be the $G(F^s)$-conjugacy class of the composition of $j$ with the inclusion $T \to G$.

For $\sigma \in \Gamma$ the embedding $\sigma(\hat\jmath\,)=\sigma_{\hat G}\circ \hat\jmath \circ \sigma_{\hat S}$ is $\hat G$-conjugate to $\hat\jmath$, but has the same image because $\hat T$ is $\Gamma$-stable, so there exists $w \in \Omega(\hat T,\hat G)$ such that $\sigma(\hat\jmath)=w\circ\hat\jmath$. By construction of $j$ we have $\sigma(j)=w\circ j$, using $\Omega(T,G)=\Omega(\hat T,\hat G)$, and conclude that $J$ is $\Gamma$-stable. Since all $\Gamma$-stable Borel pairs of $\hat G$ are conjugate under $\hat G^\Gamma$, and all $\Gamma$-stable Borel pairs of $G$ are conjugate under $G(F)$, the construction of $J$ does not depend on the choices of Borel pairs.

Now drop the assumption that $G$ is quasi-split. We consider an inner twist $\xi : G_0 \to G$ with $G_0$ quasi-split. It gives an identification $\hat G_0=\hat G$ and we obtain from $\hat J$ a $\Gamma$-stable $G_0(F^s)$-conjugacy class $J_0$ of embeddings $S \to G_0$. Composing with $\xi$ we obtain the desired $G(F^s)$-conjugacy class $J$ of embeddings $S \to G$. It is $\Gamma$-stable, because the $G(F^s)$-conjugacy class of $\xi$ is.

This completes the construction of $J$. We will refer to it as the set of \emph{admissible} embeddings $S \to G$. If we want to record the group $G$ we will write $J^G$.

Write $J(F)$ for the set of $\Gamma$-fixed points in $J$, i.e. the set of embeddings $S \to G$ defined over $F$. When $G=G_0$ is quasi-split, a result of Kottwitz \cite[Corollary 2.2]{Kot82} guarantees that this set is non-empty. For a general $G$ this set may be empty. The group $G(F)$ acts on $J(F)$ by conjugation and we will write $J(F)/G(F)$ for the set of orbits under this action.

Given any two elements $j_1,j_2 \in J(F)$ there exists $g \in G(F^s)$ such that $j_2 = \tx{Ad}(g)\circ j_1$. The map $\sigma \mapsto j_1^{-1}(g^{-1}\sigma(g))$ is a 1-cocycle of $\Gamma$ valued in $S(F^s)$. Its class is independent of $g$ and will be denoted by $\tx{inv}(j_1,j_2)$. For a fixed $j \in J(F)$ the map $j_2 \mapsto \tx{inv}(j_1,j_2)$ is a bijection between $J(F)/G(F)$ and $\tx{ker}(H^1(j_1) : H^1(F,S) \to H^1(F,G))$.

One can combine multiple inner forms in this discussion to obtain a more uniform picture. This requires the use of pure or rigid inner twists. Fix a quasi-split group $G_0$. We start with the case of pure inner twists and consider tuples $(G,\xi,z,j)$, where $(G,\xi,z)$ is a pure inner twist of $G_0$ and $j \in J^G(F)$. An isomorphism $(G_1,\xi_1,z_1,j_1) \to (G_2,\xi_2,z_2,j_2)$ between such tuples is an isomorphism $(f,g) : (G_1,\xi_1,z_1) \to (G_2,\xi_2,z_2)$ of inner twists that satisfies $j_2=f\circ j_1$. Let $\mc{J}(F)$ be the category whose objects are these tuples and whose morphisms are these isomorphisms. Given two tuples as above there exists $g \in G_0(F^s)$ such that $j_2 = \xi_2 \circ \tx{Ad}(g)\circ\xi_1^{-1}\circ j_1$. The map 
\[ \sigma \mapsto j_1^{-1}(\xi_1(g^{-1}z_2(\sigma)\sigma(g)z_1(\sigma)^{-1})) \] 
is a 1-cocycle $\Gamma \to S(F^s)$ whose class $\tx{inv}(j_2,j_1)$ depends only on the isomorphism classes of the tuples. If we fix the tuple $(G_1,\xi_1,z_1,j_1)$ then the map $(G_2,\xi_2,z_2,j_2) \mapsto \tx{inv}(j_1,j_2)$ induces a bijection from the set of isomorphism classes $[\mc{J}(F)]$ to $H^1(F,S)$.

The individual groups can be extracted from the combined picture as follows. For a fixed $(G,\xi,z)$ we can view the set $J^G(F)$ as the set of objects in a category, with set of morphisms $j_1 \to j_2$ given by $\{g \in G(F)|j_2=\tx{Ad}(g)\circ j_1\}$. Then we obtain an embedding (fully faithful functor) from $J^G(F)$ to $\mc{J}(F)$. The category $\mc{J}(F)$ decomposes into blocks, indexed by $H^1(F,G_0)$, and each block is equivalent to $J^G(F)$ for some $(G,\xi,z)$. In particular, the set of isomorphism classes $[\mc{J}(F)]$ is the disjoint union $\bigcup_{H^1(F,G_0)} (J^G(F)/G(F))$. If we choose $j_0 \in J^{G_0}(F)$, then under the bijection $[\mc{J}(F)] \to H^1(F,S)$ given by $(G,\xi,z,j) \mapsto \tx{inv}(j,j_0)$, an individual $J^G(F)/G(F)$ coming from $(G,\xi,z)$ is mapped bijectively onto the fiber of $H^1(j_0) : H^1(F,S) \to H^1(F,G)$ over the class of $z$.

The bijection $[\mc{J}(F)] \to H^1(F,S)$ coming from fixing $(G,\xi,z,j)$ can be understood as the orbit map for a natural action of $H^1(F,S)$ on $[\mc{J}(F)]$ that does not depend on any choices. To see this action it is easier to consider the simplified notation, where a pure inner twist of $G_0$ is understood simply as an element $z \in Z^1(F,G_0)$, in the sense that it corresponds to $(G_z,\tx{id},z)$, where $G_z$ is the $F$-group obtained by twisting the rational structure of $G_0$ by $z$. Then $\mc{J}$ consists of pairs $(z,j)$, where $z \in Z^1(F,G_0)$ and $j \in J^{G_0}$. Such a pair lies in $\mc{J}(F)$ if and only if $j \in J^{G_z}(F)$, which is explicitly given as $\tx{Ad}(z(\sigma))\sigma_{G_0} \circ j \circ \sigma_S^{-1}=j$. The group $G_0(F^s)$ acts on $\mc{J}$ by $g \cdot (z,j) = (gz(\sigma)\sigma(g)^{-1},\tx{Ad}(g)\circ j)$. This action preserves $\mc{J}(F)$ and the orbit space is $[\mc{J}(F)]$. We introduce the action of $Z^1(F,S)$ on $\mc{J}(F)$ as $x \cdot (z,j) = (j(x) \cdot z,j)$ for $x \in Z^1(F,S)$ and $(z,j) \in \mc{J}(F)$. One checks directly that the actions of $Z^1(F,S)$ and $G(F^s)$ on $\mc{J}(F)$ commute and that the action of $S(F^s)$ on $Z^1(F,S)$ is compatible with the action of $Z^1(F,S)$ on $\mc{J}(F)$ in the sense that $(s\cdot x) \cdot (z,j) = j(s) \cdot (x \cdot (z,j))$. Thus we obtain an action of $H^1(F,S)$ on $[\mc{J}(F)]$. 
\begin{lem} \label{lem:simtrans}
The above action is simply transitive.
\end{lem}
\begin{proof}
For simplicity, let $x \in Z^1(F,S)$, $g \in G_0(F^s)$, and $(z,j) \in \mc{J}(F)$ be such that $g \cdot (z,j)=x\cdot (z,j)$. Then $j=\tx{Ad}(g)\circ j$ from which we see that $g=j(s)$ for some $s \in S(F^s)$. We also have $j(s)z(\sigma)\sigma_{G_0}(j(s))^{-1}=j(x(\sigma))\cdot z(\sigma)$ and multiplying on the right by $z(\sigma)^{-1}$ and using that $(z,j) \in \mc{J}(F)$ we conclude $j(s \cdot \sigma_S(s)^{-1})=j(x(\sigma))$, thus $[x]=1$ in $H^1(F,S)$, as desired.

For transitivity, we need to show that any two elements of $[\mc{J}(F)]$ are in the same $H^1(F,S)$ orbit. Since the members of $J^{G_0}$ are all conjugate under $G_0(F^s)$ we may represent the two elements of $[\mc{J}(F)]$ by $(z_1,j)$ and $(z_2,j)$ with the same $j$. The fact that both lie in $\mc{J}(F)$ implies that $\tx{Ad}(z_1(\sigma))\circ\sigma_{G_0}\circ j = \tx{Ad}(z_2(\sigma))\circ\sigma_{G_0}\circ j$. It follows that $\sigma_{G_0}^{-1}(z_2(\sigma)^{-1}z_1(\sigma))$ lies in the image of $j$, and we write it as $j(\sigma_S^{-1}(x(\sigma)))$ for some $x(\sigma) \in S(F^s)$. Using that $(z_1,j),(z_2,j) \in \mc{J}(F)$ we see that $j(x(\sigma))=z_1(\sigma)z_2(\sigma)^{-1}$. From this one easily checks that $x \in Z^1(F,S)$ and concludes $(z_1,j)=x\cdot (z_2,j)$, as desired.
\end{proof}

The procedure that produced $J$ from $\hat J$ is invertible. Given a $\Gamma$-stable $G(F^s)$-conjugacy class $J$ of embeddings $S \to G$, we compose with $\xi^{-1}$ to obtain a $\Gamma$-stable $G_0(F^s)$-conjugacy class $J_0$ of embeddings $S \to G_0$. Pick a $\Gamma$-stable Borel pairs $(T,B)$ of $G_0$ and $(\hat T,\hat B)$ of $\hat G$. Choose $j \in J_0$ giving an isomorphism $S \to T$ and use it to get identifications $X^*(\hat S)=X_*(S)=X_*(T)=X^*(\hat T)$. The resulting isomorphism $\hat S \to \hat T$ is not $\Gamma$-equivariant, but translates the $\Gamma$-structure on $\hat S$ to a twist by $\Omega(\hat T,\hat G)$ of the $\Gamma$-structure on $\hat T$. Therefore the composition of this isomorphism with the tautological inclusion $\hat T \to \hat G$ has the property that its $\hat G$-conjugacy class $\hat J$ is $\Gamma$-stable.

The same discussion applies to the setting of rigid inner forms when the base field is local. One should only replace $H^1(F,-)$ with $H^1(u_F \to \mc{E}_F,Z(G) \to -)$.

\subsection{Representations and elements in inner forms} \label{sub:across}

In \S\ref{sub:adm} we reviewed the idea of grouping the various admissible embeddings of a torus $S$ into the different inner forms of a fixed group $G_0$, thereby obtaining a set $\mc{J}(F)$ with an action of $G_0(F^s)$ on it. Here we follow the same idea, but group elements or representations. 

We focus on $F=\R$, although the same procedure applies for any $F$ of interest. Let $\tilde\Pi$ be the set of pairs $(z,\pi)$, where $z \in Z^1(F,G_0)$ and $\pi$ is an isomorphism class of irreducible admissible representations of $G_z(F)$, where $G_z$ is again the $\R$-group obtained by twisting the $\R$-structure of $G_0$ by $z$. The group $G_0(\C)$ acts on this set by $g\cdot (z,\pi) = (gz(\sigma)\sigma(g)^{-1},\pi \circ\tx{Ad}(g)^{-1})$. Note that $\tx{Ad}(g) : G_z \to G_{gz(\sigma)\sigma(g)^{-1}}$ is an isomorphism of algebraic $\R$-groups. Let $\Pi=\tilde\Pi/G_0(\C)$. This is the set of isomorphism classes of irreducible admissible representations of pure inner forms of $G_0$. The pairs $(z,\pi) \in \tilde\Pi$ with a fixed $z$ correspond to all isomorphism classes of representations of $G_z(\R)$.

Similarly, consider the set of pairs $(z,\delta)$, where $z \in Z^1(F,G_0)$ and $\delta \in G_z(F)$ is a regular semi-simple element. The group $G_0(\C)$ acts on this set by $g\cdot (z,\delta)=(gz(\sigma)\sigma(g)^{-1},g\delta g^{-1})$. If two pairs $(z_1,\delta_1)$ and $(z_2,\delta_2)$ are in the same orbit for this action, we will call them \emph{stably conjugate}. Just as with the case of admissible embeddings, there is a cohomological invariant $\tx{inv}((z_1,\delta_1),(z_2,\delta_2)) \in H^1(F,S_1)$, where $S_1 \subset G_0$ is the centralizer of $\delta_1$, a maximal $\R$-torus of $G_{z_1}$. This invariant is the class of the 1-cocycle $\Gamma \to S_1$ given by
\[ \sigma \mapsto g^{-1}z_2(\sigma)\sigma(g)z_1(\sigma)^{-1}, \]
where $g \in G_0(\C)$ is any element satisfying $(z_2,\delta_2)=g \cdot (z_1,\delta_1)$. If $z_1=z_2=z$ and $\tx{inv}((z,\delta_1),(z,\delta_2))=1$, then $g$ can be chosen to lie in $G_z(\R)$, and we see that $\delta_1$ and $\delta_2$ are $G_z(\R)$-conjugate. More generally, $\tx{inv}((z_1,\delta_1),(z_2,\delta_2))=1$, the element $g$ can be chosen to be an isomorphism $G_1 \to G_2$ defined over $\R$. In either case, we will call $(z_1,\delta_1)$ and $(z_2,\delta_2)$ \emph{rationally} conjugate. With this in mind, we call the orbit space for this action the set of rational classes of regular semi-simple elements of pure inner forms. 

The same discussion applies with pure inner forms replaced by rigid inner forms, using the action of $\mc{E}_\R$ on $G_0(\C)$ via the projection map $\mc{E}_\R \to \Gamma$.

\subsection{Weyl denominators} \label{sub:weyldenom}

Let $G$ be a connected reductive $\R$-group and $T \subset G$ a maximal $\R$-torus. One can consider the function $T(\R) \to \R$ defined as
\[ D_T(t) = D_T^G(t) = \textstyle\prod\limits_{\alpha \in R(T,G)} (1-\alpha(t)). \]
An element $t \in T(\R)$ is called regular if $T$ is the identity component of its centralizer, and strongly regular if $T$ is its centralizer. In either case, $T$ is uniquely determined by $t$, and we can write $D_T(t)=D(t)$.

In this paper we will normalize orbital integrals and characters by multiplying them by $|D(t)|^{1/2}$. Thus, for $t \in T(\R)$ strongly regular and $f \in \mc{C}_c^\infty(G(\R))$ we set
\[ O_\gamma(f) = |D(t)|^{1/2}\int_{G(\R)/T(\R)} f(g\gamma g^{-1})dg/dt, \]
while for $\pi$ admissible representation of $G(\R)$ we have the normalized character function $\Theta_\pi : G(\R)_\tx{sr} \to \C$ determined by
\[ \tx{tr}(\pi(f)) = \Theta_\pi(f) = \int_{G(\R)_\tx{sr}} |D(g)|^{-1/2}\Theta_\pi(g)f(g)dg. \]
This has the advantage that the functions $O_\gamma(f)$ and $\Theta_\pi(\gamma)$  remain bounded as $\gamma$ approaches singular elements in $T(\R)$.

We note here that the distribution $f \mapsto O_\gamma(f)$ depends on the choice of measures $dg$ on $G(\R)$ and $dt$ on $T(\R)$, the distribution $f\mapsto \Theta_\pi(f)$ depends on $dg$, while the function $\gamma \mapsto \Theta_\pi(\gamma)$ does not depend on any choices.

An important role in this paper will be played by a function $D_B$ which has the property that $|D_B|=|D_T|^{1/2}$. To see this function, let us interpret $D_T$ as an element of the group ring $\Z[Q]$, where $Q \subset X^*(T)\otimes\Q$ is the root lattice, we we write the group operation on $Q$ multiplicatively. Given a Borel $\C$-subgroup $B \subset G$ containing $T$ we write $\alpha>0$ when $\alpha$ is a $B$-positive root, and define
\begin{equation} \label{eq:weyldenom}
D_B' = \prod_{\alpha>0} (1-\alpha^{-1}) \in \Q[Q],\qquad D_B = \prod_{\alpha>0} (\alpha^{1/2}-\alpha^{-1/2}) \in \Q[Q].
\end{equation}
In $\Q[Q]$ we have the identity
\begin{equation} \label{eq:dgb}
D_B = \rho \cdot D_B',	
\end{equation}
where $\rho=\prod_{\alpha>0} \alpha^{1/2} \in Q$. This implies
\[ D_T = D_B' \cdot D_{\bar B}' = D_B \cdot D_{\bar B}, \]
where $\bar B$ is the Borel subgroup opposite to $B$. Moreover, for $w \in \Omega(T,G)$ we have
\begin{equation} \label{eq:wsd}
wD_B = D_{w^{-1}Bw} = \tx{sgn}(w)D_B.	
\end{equation}
In particular, $|D_B|$ is independent of the choice of $B$ and hence $|D_T|^{1/2}=|D_B|$, provided we can interpret $D_B$ as a function on $T(\R)$.

It is clear that $D_B'$ is a function on $T(\R)$. If we want to interpret $D_B$ as a function of $T(\R)$, the occurrence of $\alpha^{1/2}$ in the formula causes a problem. From \eqref{eq:dgb} we see that $D_B$ will be a function of $T(\R)$ if and only if $\rho$ is, which is equivalent to the element $\rho$ lying in $X^*(T)$. This is always the case when $G$ is semi-simple and simply connected, but can fail in general. To remedy this situation, one can introduce a double cover of $T(\R)$, which will be discussed in the next section.

\subsection{Double covers of tori and $L$\-embeddings} \label{sub:covtori}

Let $G$ be a connected reductive $\R$-group and let $T \subset G$ be a maximal torus. An obstruction to the element $D_B$ defining a function on $T(\R)$ is the fact that $\rho \in \frac{1}{2}X^*(T)$ may not lie in $X^*(T)$. To remedy this, Adams--Vogan introduce in \cite{AV92}, \cite{AV16} the $\rho$-double cover $T(\R)_\rho$ as the pull-back of the diagram
\[ T(\R)\stackrel{\rho^2}{\lrw} \C^\times \stackrel{(-)^2}{\llw} \C^\times, \]
which comes equipped with a natural character $\rho : T(\R)_\rho \to \C^\times$, namely the projection onto the right factor $\C^\times$. By construction we have an exact sequence
\[ 1 \to \{\pm 1\} \to T(\R)_\rho \to T(\R) \to 1 \]
and $\rho$ is a genuine character, i.e. $\rho(-x)=-\rho(x)$ for $x \in T(\R)_\rho$, where $-x$ denote the product of $x$ and the element $-1$.

While the double cover $T(\R)_\rho$ appears to depend on $\rho$, this is actually not so. Indeed, for any other Borel $\C$-subgroup $B'$ we have $\rho'/\rho \in X^*(T)$, which allows us to define the genuine character $\rho' : T(\R)_\rho \to \C^\times$ as $\rho \cdot (\rho'/\rho)$. Combining this character with the natural projection $T(\R)_\rho \to T(\R)$ gives a map from $T(\R)_\rho$ to the diagram defining $T(\R)_{\rho'}$, hence an isomorphism $T(\R)_\rho \to T(\R)_{\rho'}$.

To emphasize the independence of $T(\R)_\rho$ on $\rho$, and emphasize the dependence on the ambient group $G$, we will write $T(\R)_G$ for this cover. For each Borel $\C$-subgroup $B$ we have the genuine character $\rho_B : T(\R)_G \to \C^\times$.

In this paper we will be particularly interested in the case when $T$ is elliptic. Then there is a different way to obtain the $\rho$-cover that generalizes to all local fields, as discussed in \cite{KalDC}. One first defines the ``big cover'' of $T(\R)$ as follows. Each root provides a homomorphism $\alpha : T(\R) \to \mb{S}^1$. Combining these homomorphisms for a pair $A=\{\alpha,-\alpha\}$ provides a homomorphism $A : T(\R) \to \mb{S}^1_-$, where $\mb{S}^1_- \subset \mb{S}^1 \times \mb{S}^1$ is the kernel of the product map $\mb{S}^1 \times \mb{S}^1 \to \mb{S}^1$. Define the ``big cover'' as the pull-back of
\begin{equation} \label{eq:bigcover}
T(\R) \stackrel{(\alpha)}{\lrw} \prod \mb{S}^1_- \stackrel{z/\bar z}{\llw} \prod(\C^\times/\R_{>0})_-
\end{equation}
where the products run over the set of pairs $A=\{\alpha,-\alpha\}$ consisting of a root and its negative, and $(\C^\times/\R_{>0})_-$ denotes analogously the anti-diagonal in $(\C^\times/\R_{>0}) \times (\C^\times/\R_{>0})$. The result is an extension
\[ 1 \to \prod\{\pm1\} \to T(\R)_{GG} \to T(\R) \to 1. \]
Under the isomorphism $\mb{S}^1 \to \C^\times/\R_{>0}$ the map $z/\bar z$ becomes the squaring map, and we see that $T(\R)_{GG}$ is equipped with a character $\alpha^{1/2} : T(\R)_{GG} \to \mb{S}^1$ for each root $\alpha$, and that $\beta^{1/2}=(\alpha^{1/2})^{-1}$ whenever $\beta=\alpha^{-1}$. There is an obvious surjective homomorphism $T(\R)_{GG} \to T(\R)_G$ whose kernel is the kernel of the multiplication map $\prod\{\pm1\} \to \{\pm1\}$. The function $\alpha^{1/2}-\alpha^{-1/2}$ is well-defined on the big cover, while for any choice Borel $\C$-subgroup $V$ the function
\[ D_B := \prod_{\alpha>0}(\alpha^{1/2}-\alpha^{-1/2}) \]
descends to the double cover $T(\R)_G$.

The action of the Weyl group $\Omega_G(T)(\R)$ lifts naturally to an action on $T(\R)_G$, even on $T(\R)_{GG}$, because $\Omega_G(T)(\R)$ acts naturally on each term in \eqref{eq:bigcover}. The identity \eqref{eq:wsd} holds for this function.

What makes the double cover $T(\R)_G$ very useful in the setting of the Langlands program is the fact that there is an associated $L$\-group $^LT_G$, as well as a \emph{canonical} $\hat G$-conjugacy class of $L$\-embeddings $^LT_G \to {^LG}$, cf. \cite[\S4.1]{KalDC}. The property of the $L$\-group $^LT_G$ is that the set of $\hat T$-conjugacy classes of $L$\-homomorphisms $W_\R \to {^LT}_G$ is in natural bijection with the set of genuine characters of $T(\R)_G$. Therefore, any $L$\-parameter for $G$ that factors through the image of the embedding of $^LT_G$ provides in a canonical way an $\Omega_G(T)(\R)$-orbit of genuine characters of $T(\R)_G$.

In contrast to $^LT_G$, there is generally no canonical $L$\-embedding $^LT \to {^LG}$. In fact, if the Galois form of $^LT$ is used, there is generally no $L$\-embedding $^LT \to {^LG}$ at all, let alone a canonical one. If the Weil form of $^LT$ is used, then there always do exist $L$\-embeddings $^LT \to {^LG}$, but there is generally no canonical choice. If one chooses a genuine character of $T(\R)_G$, then the pointwise product of its $L$\-parameter $W_\R \to {^LT_G}$ with the natural inclusion $\hat T \to {^LT_G}$ does lead to an $L$\-isomorphism $^LT \to {^LT_G}$ between the Weil forms of the $L$\-groups for $T(\R)$ and $T(\R)_G$. Composing this isomorphism with the canonical $L$\-embedding $^LT_G \to {^LG}$ provides an $L$\-embedding $^LT \to {^LG}$, and every $L$\-embedding arises from this construction. A convenient choice for a genuine character on $T(\R)_G$ is the character $\rho$ associated to some Borel $\C$-subgroup of $G$ containing $T$.

It is worth pointing out that all $L$\-embeddings $^LT \to {^LG}$, as well as all $L$\-embeddings $^LT_G \to {^LG}$, that extend a fixed embedding $\hat\jmath : \hat T \to \hat G$, have the same image, namely
\begin{equation} \label{eq:lembim}
\{x \in N_{^LG}(\hat T)\,|\, x\cdot\hat\jmath(t)\cdot x^{-1} = \hat\jmath(\sigma_x(t))\ \forall t \in \hat T\},
\end{equation}
where $\sigma_x \in \Gamma$ is the image of $x$ under the natural projection $^LG \to \Gamma$.

\subsection{Essentially discrete series representations} \label{sub:essds}

We will recall here the definition of eds representations of real reductive groups and their classification due to Harish-Chandra. We will need to allow for a slightly different set-up than that in Harish-Chandra's original papers.

The main results on discrete series representations appear in \cite{HCDSI}, where Harish-Chandra considers a real Lie group $\mc{G}$ that is connected, semi-simple, and acceptable. An example is $\mc{G}=G(\R)$ for a connected semi-simple simply connected $\R$-group $G$. Later, in \cite{HC-R1}, \cite{HC-R2}, \cite{HC-R3}, Harish-Chandra works with a wider class of groups (known as ``Harish-Chandra's class'', consisting of real Lie groups satisfying the conditions of \cite[\S3]{HC-R1}), but his results on discrete series still assume that the center of the group is compact. 

The setting that we will need here is one that encompasses the groups $G(\R)$, where $G$ is a connected reductive $\R$-group, as well as the covers of $G(\R)$ obtained from \cite{KalHDC}, because they will be important for endoscopy, as will be discussed in \S\ref{sub:covendo} below.

Thus, we take $\mc{G}$ to be a real Lie group, $\mf{g}(\R)$ its Lie algebra, and $\mf{g}$ the complexification. We make the following assumptions:
\begin{enumerate}
	\item $\mf{g}$ is a complex reductive Lie algebra.
 	\item If $G_\tx{ad}$ is the complex connected semi-simple adjoint group associated to the derived subalgebra $\mf{g}'$ of $\mf{g}$, i.e. $G_\tx{ad}=\tx{Aut}(\mf{g}')^\circ$, then the action of $\mc{G}$ on $\mf{g}$ comes from a (necessarily unique) homomorphism $\mc{G} \to G_\tx{ad}(\R)$ of real Lie groups. Its kernel is equal to the center of $\mc{G}$.
  	\item If $G_\tx{sc}$ is the complex connected semi-simple simply connected group associated to $\mf{g}'$, i.e. the universal cover of $G_\tx{ad}$, then the natural homomorphism $G_\tx{sc}(\R) \to G_\tx{ad}(\R)$ lifts (necessarily uniquely) to a homomorphism $G_\tx{sc}(\R) \to \mc{G}$ of Lie groups.
   	\item Let $\mc{G}^\natural$ be the image of $G_\tx{sc}(\R) \to \mc{G}$. Then $[\mc{G}^\natural \cdot Z(\mc{G}):\mc{G}] < \infty$.
\end{enumerate}

\begin{lem} A group $\mc{G}$ satisfying the above conditions is in the Harish-Chandra class provided $\pi_0(Z(\mc{G}))$ is finite.
\end{lem}
\begin{proof}
	Condition (1) of \cite[\S3]{HC-R1} follows from Condition 2. above.
	
	Condition (2) of \cite[\S3]{HC-R1} follows from the fact that the analytic subgroup of $\mc{G}$ corresponding to $[\mf{g},\mf{g}]$ is equal to the image $\mc{G}^\natural$ of $G_\tx{sc}(\R) \to \mc{G}$. Any element of the center of $\mc{G}^\natural$ is therefore the image of an element $x \in G_\tx{sc}(\R)$ that maps trivially to $G_\tx{ad}(\R)$, hence lies in the center of $G_\tx{sc}(\R)$. We conclude that the center of $G_\tx{sc}(\R)$, which is finite, surjects onto the center of $\mc{G}^\natural$.

	For condition (3) we use that $\mc{G}^\natural \subset \mc{G}^\circ$ and see that $[\mc{G}^\circ : \mc{G}] \leq [\mc{G}^\natural \cdot Z(\mc{G})^\circ : \mc{G}] \leq [\mc{G}^\natural \cdot Z(\mc{G}) : \mc{G}] \cdot [Z(\mc{G})^\circ:Z(\mc{G})] < \infty$.
\end{proof}

One class of examples of this set-up arise as $\mc{G}=G(\R)$, where $G$ is a connected reductive $\R$-group. Another class of examples comes from the covers $G(\R)_x$ of such groups constructed in \cite{KalHDC} from characters $x : \bar\pi_1(G) \to \C^\times$. In both cases $\pi_0(Z(\mc{G}))$ is finite, so these groups are in  Harish-Chandra's class.

Recall that, for any Lie group $\mc{H}$ one defines $[\mc{H},\mc{H}] \subset \mc{H}$ to be the subgroup generated by the commutators in $\mc{H}$.

\begin{lem} \label{lem:comm}
	The group $\mc{G}^\natural$ is a connected semi-simple Lie group, and a closed Lie subgroup of $\mc{G}$ with Lie algebra $\mf{g}'$. We have
	\[ [\mc{G}^\natural,\mc{G}^\natural] = \mc{G}^\natural = [\mc{G},\mc{G}]\quad\tx{and}\quad Z(\mc{G}^\natural) = Z(\mc{G}) \cap \mc{G}^\natural. \]
	The subgroup $\mc{G}^\natural \cdot Z(\mc{G})$ of $\mc{G}$ is an open and closed Lie subgroup.
\end{lem}
\begin{proof}
	We will cite some results from \cite{BourLie1-3}. For this, we note that what is called a \emph{Lie subgroup} in that reference (\cite[Chapter III, \S1, no. 3, Definition 3]{BourLie1-3}) is usually referred to as a closed Lie subgroup (see \cite[Chapter III, \S1, no. 3, Proposition 5]{BourLie1-3}), and what is called an \emph{integral subgroup} in that reference (\cite[Chapter III, \S6, no. 2, Definition 1]{BourLie1-3}) is usually referred to as a Lie subgroup.

	The Lie group $G_\tx{sc}(\R)$ is connected \cite[\S7.2, Proposition 7.6]{PR94}. Therefore, \cite[Chapter III, \S3, no. 2, Corollary 1]{BourLie1-3} shows that $\mc{G}^\natural$ is a Lie subgroup of $\mc{G}$. Since the kernel of $G_\tx{sc}(\R) \to \mc{G}$ is contained in the kernel of $G_\tx{sc}(\R) \to G_\tx{ad}(\R)$, it is central and finite. Thus $\mc{G}^\natural$ is a finite central quotient of the connected semi-simple Lie group $G_\tx{sc}(\R)$, hence itself a connected semi-simple Lie group. According to \cite[Chapter III, \S9, no. 8, Proposition 29]{BourLie1-3}, $\mc{G}^\natural$ equals its own commutator subgroup. The homomorphism $G_\tx{sc}(\R) \to \mc{G}$ induces an isomorphism on Lie algebras and hence identifies $\mf{g}'$ with the Lie algebra of $\mc{G}^\natural$.
	
	Moreover, the finiteness of the kernel of $G_\tx{sc}(\R) \to \mc{G}^\natural$ implies, via \cite[Chapter III, \S3, no. 2, Corollary 1]{BourLie1-3}, that the Lie algebra of $\mc{G}^\natural$ equals the Lie algebra of $G_\tx{sc}(\R)$, which is $\mf{g}'(\R)$. The latter is the derived subgalgebra of $\mf{g}(\R)$. It now follows from \cite[Chapter III, \S9, no. 2, Proposition 4]{BourLie1-3}, applied with $A=B=\mc{G}$ and $C=\mc{G}^\natural$, that $[\mc{G},\mc{G}]=\mc{G}^\natural$.

	Next we show that $\mc{G}^\natural$ is a closed Lie subgroup, and $\mc{G}^\natural \cdot Z(\mc{G})$ is an open and closed Lie subgroup. The latter is the image of the Lie group homomorphism $G_\tx{sc}(\R) \times Z(\mc{G}) \to \mc{G}$ given by composing the product operation on $\mc{G}$ with the homomorphism $G_\tx{sc}(\R) \to \mc{G}$ and the inclusion $Z(\mc{G}) \to \mc{G}$. This Lie group homomorphism induces an isomorphism on Lie algebras, via the decomposition $\mf{g}=\mf{g}' \oplus \mf{z}$, where $\mf{z}$ is the center of $\mf{g}$, and where we note that $\mf{z}(\R)$ is the center of $\mf{g}(\R)$ and also the Lie algebra of $Z(\mc{G})$, see \cite[Chapter III, \S9, no. 3, Proposition 9]{BourLie1-3}. Therefore $\mc{G}^\natural \cdot Z(\mc{G})$ contains an open neighborhood of the identity of $\mc{G}$, and is hence open in $\mc{G}$, and hence also closed. Since $G_\tx{sc}(\R) \to \mc{G}^\natural$ is a quotient map, a subset $U \subset \mc{G}^\natural$ is open if and only if its preimage $V$ in $G_\tx{sc}(\R)$ is open. Then $V \times Z(\mc{G})$ is open in $G_\tx{sc}(\R) \times Z(\mc{G})$, so its image $U'$ in $\mc{G}^\natural \cdot Z(\mc{G})$ is open. We have $U' \cap \mc{G}^\natural = U$. Since $\mc{G}^\natural \cdot Z(\mc{G})$ is open in $\mc{G}$ we see that $U'$ is open in $\mc{G}$, and conclude that $\mc{G}^\natural$ carries the subspace topology inherited from $\mc{G}$. According to \cite[Chapter III, \S6, no. 2, Proposition 2]{BourLie1-3}, both $\mc{G}^\natural$ and $\mc{G}^\natural \cdot Z(\mc{G})$ are closed Lie subgroups of $\mc{G}$.

	Now we compute $Z(\mc{G}^\natural)$. This group is equal to the kernel of the action of $\mc{G}^\natural$ on the Lie algebra of $\mc{G}^\natural$, which is $\mf{g}'$. By assumption 2. above, this is the same as the kernel of $\mc{G}^\natural \to G_\tx{ad}(\R)$, which in turn equals $Z(\mc{G}) \cap \mc{G}^\natural$.
\end{proof}

\begin{dfn} \label{dfn:eds}
	Let $\pi$ be an admissible representation of $\mc{G}$.
	\begin{enumerate}
		\item $\pi$ is called \emph{square-integrable}, or \emph{in the discrete series}, if every matrix coefficient of $\pi$ is square-integrable, i.e. belongs to $L^2(\mc{G})$.
  		\item $\pi$ is called \emph{relatively square-integrable}, or in the \emph{in the relative discrete series}, if $\pi$ is unitary and 
                  every matrix coefficient of $\pi$ is square-integrable modulo the center of $\mc{G}$, i.e. lies in $L^2(\mc{G}/Z(\mc{G}))$.

  		\item $\pi$ is called  \emph{essentially square-integrable}, or  \emph{in the essentially discrete series} (\emph{eds} for short), if there exists a character $\chi : \mc{G} \to \C^\times$ such that $\pi \otimes \chi$ is relatively square-integrable.
	\end{enumerate}
\end{dfn}

\begin{rem}
It is clear that a square-integrable representation is relatively square-integrable, and a relatively square-integrable representation is essentially square-integrable. When $Z(\mc{G})$ is not compact, a square-integrable representation cannot exist, so the concepts of relatively square-integrable and essentially square-integrable are the only ones that are available.

Some authors say that an irreducible representation $\pi$ (not necessarily unitary) is in the \emph{relative discrete series} if the restrictions of its matrix coefficients to the commutator subgroup are square-integrable. This commutator subgroup equals $\mc{G}^\natural$ by Lemma \ref{lem:comm}, so the next lemma shows that this is equivalent to our notion of essentially square-integrable.
\end{rem}

\begin{lem} \label{lem:eds-equiv}
  Let $\pi$ be an irreducible admissible (not necessarily unitary) representation of $\mc{G}$.
  Then $\pi$ is essentially square-integrable if and only its restriction to $\mc{G}^\natural$ is square-integrable, in which case this restriction is a finite direct sum of irreducible square-integrable representations.
\end{lem}
\begin{proof}
The main part of the proof is contained in the following statement.

\ul{Claim:} There exists a continuous group homomorphism $\chi : \mc{G} \to \R_{>0}$ such that $\pi\otimes\chi$ is unitary if and only if $\pi|_{\mc{G}^{\natural}}$ is unitary, in which case $\pi|_{\mc{G}^{\natural}}$ is a finite direct sum of irreducible unitary representations.

\ul{Proof:} We begin by noticing that the claimed statement is insensitive to replacing $\mc{G}$ by a finite index subgroup $\mc{H}$ that contains $\mc{G}^\natural$. First, the representation $\pi|_{\mc{G}^\natural}$ is unaffected by this change. Second, if there exists $\chi : \mc{G} \to \C^\times$ such that $\pi \otimes \chi$ is unitary, then $\pi|_\mc{H} \otimes \chi|_\mc{H} = (\pi\otimes\chi)|_\mc{H}$ is also unitary. Third, if there exists $\chi : \mc{H} \to \C^\times$ such that $\pi \otimes \chi|_\mc{H}$ is unitary, then the divisibility of $\R_{>0}$ implies that $\chi$ has a unique extension to $\mc{G}$, and now we can use the finiteness of $\mc{G}/\mc{H}$ to average any $\mc{H}$-invariant scalar product on $\pi\otimes\chi$ to obtain a $\mc{G}$-invariant scalar product. Fourth, while $\pi|\mc{H}$ may no longer be irreducible, it is still semi-simple and of finite length \textcolor{red}{ref}, so we can apply the reasoning to each of its finitely many irreducible summands.

This observation allow us to assume that $\mc{G}=\mc{G}^\natural \cdot Z(\mc{G})$. Since $\pi$ is irreducible and admissible, its $(\g,K)$-module is also irreducible and admissible \textcolor{red}{ref}, and Dixmier's version of Schur's lemma
\cite[Lemma 0.5.3]{wallach}
implies that it has a central character
$\omega : Z(\mc{G}) \to \C^\times$, so its restriction to $\mc{G}^\natural$ is still irreducible. If there is $\chi : \mc{G} \to \R_{>0}$ so that $\pi \otimes \chi$ is unitary, then $(\pi\otimes\chi)|_{\mc{G}^\natural}$ is still unitary, but $\chi$ has trivial restriction to $\mc{G}^\natural$ by Lemma \ref{lem:comm}, so $\pi|_{\mc{G}^\natural}$ is unitary. Conversely, assume that $\pi|_{\mc{G}^\natural}$ is unitary. Note that, since the center of $G_\tx{sc}(\R)$ is finite, $|\omega|$ restricts trivially to $Z(\mc{G}) \cap \mc{G}^\natural$, hence extends to $\mc{G}$, and $\pi\otimes|\omega|^{-1}$ is a unitary representation of $\mc{G}$.\hfill $\sslash$

Now suppose $\pi|_{\mc{G}^\natural}$ is square integrable. In particular it is unitary, so by the above claim there is $\chi : \mc{G} \to \R_{>0}$ so that $\pi\otimes\chi$ is unitary. The matrix coefficients of $\pi|_{\mc{G}^\natural}$ are the restrictions of the matrix coefficients of $\pi$ to $\mc{G}^\natural$. It follows that the matrix coefficients of $\pi\otimes\chi$ lie in $L^2(\mc{G}^\natural/Z(\mc{G}^\natural))$. Since $\mc{G}^\natural/Z(\mc{G}^\natural)$ is of finite index in $\mc{G}/Z(\mc{G})$ we see that these matrix coefficients also lie in $L^2(\mc{G}/Z(\mc{G}))$, thus $\pi$ is essentially square-integrable.

Conversely, suppose that $\pi$ is essentially square-integrable, so that $\pi\otimes\chi$ is relatively square-integrable for some $\chi : \mc{G} \to \C^\times$. By above claim, $(\pi\otimes\chi)|_{\mc{G}^\natural}$ is a finite direct sum of irreducible unitary representations. By Lemma \ref{lem:comm}, $\chi|_{\mc{G}^\natural}$ is trivial, so $\pi|_{\mc{G}^\natural}$ is a finite direct sum of irreducible unitary representations. The matrix coefficients of these representations are the restrictions of those of $\pi\otimes\chi$, hence belong to $L^2(\mc{G}^\natural)$.
\end{proof}

\begin{lem} \label{lem:ani}
Assume that $G_\tx{ad}$ has an anisotropic maximal torus $S_\tx{ad}$. The preimage $\mc{S} \subset \mc{G}$ of $S_\tx{ad}(\R)$ equals $Z(\mc{G}) \cdot \mc{S}^\natural$, where $\mc{S}^\natural = \mc{S} \cap \mc{G}^\natural$. In particular, $\mc{S}$ is abelian.
\end{lem}
\begin{proof}
	Let $S_\tx{sc} \subset G_\tx{sc}$ be the preimage of $S_\tx{ad}$. It is also an anisotropic maximal torus. Then $S_\tx{sc}(\R)$ and $S_\tx{ad}(\R)$ are connected compact abelian Lie groups. Therefore, the natural map $S_\tx{sc}(\R) \to S_\tx{ad}(\R)$ is surjective. This implies $\mc{S} = Z(\mc{G}) \cdot \mc{S}^\natural$.
\end{proof}

Let $S_\tx{ad}(\R)_\pm \to S_\tx{ad}(\R)$ be the double cover defined in \S\ref{sub:covtori}. We define the double cover $\mc{S}_\pm \to \mc{S}$ as the fiber product of $\mc{S} \to S_\tx{ad}(\R) \from S_\tx{ad}(\R)_\pm$. The action of $\Omega(S_\tx{ad},G_\tx{ad})$ on $S_\tx{ad}(\R)$ lifts naturally to both $\mc{S}$ and $S_\tx{ad}(\R)_\pm$, hence also to $\mc{S}_\pm$.

\begin{thm} \label{thm:eds}
The set of eds representations of $\mc{G}$ is in bijection with the set of $\mc{G}$-conjugacy classes of pairs $(\mc{S},\tau)$, where $\mc{S}$ is the preimage in $\mc{G}$ is a compact maximal torus of $G_\tx{ad}(\R)$ and $\tau$ is a regular genuine character of $\mc{S}_\pm$. This representation is uniquely determined, among irreducible tempered representations, by the fact that its character $\Theta$ is given on $\mc{S}$ by the formula
\begin{equation} \label{eq:charfmla}
		(-1)^{q(G_\tx{sc})}\sum_{w \in N(\mc{S},\mc{G})/\mc{S}} \frac{\tau}{d_\tau}(w\dot\delta) = (-1)^{q(G_\tx{sc})}\sum_{w \in N(\mc{S},\mc{G})/\mc{S}} \frac{\tau'}{d'_\tau}(w\delta).
\end{equation}
Furthermore, $\Theta$ is supported on $Z(\mc{G})\cdot \mc{G}^\natural$.
\end{thm}

\begin{rem}
	We want to clarify the terms ``anisotropic'', ``elliptic'', and ``compact'', for maximal tori. We call an algebraic $\R$-torus $S$ anisotropic if it contains no non-trivial split torus. This is equivalent to the action of complex conjugation on its character module being multiplication by $-1$. This is also equivalent to the real Lie group $S(\R)$ being compact. We call a maximal torus $S \subset G$ elliptic, is $S/Z(G)$ is anisotropic. We may thus call $\mc{S}$ an ``elliptic'' maximal torus of $\mc{G}$.
\end{rem}

Before we give the proof we explain the notation. Due to the construction of the covers $S_\tx{sc}(\R)_\pm$ and $S_\tx{ad}(\R)_\pm$, the natural map $S_\tx{sc}(\R) \to S_\tx{ad}(\R)$ lifts naturally to a map $S_\tx{sc}(\R)_\pm \to S_\tx{ad}(\R)_\pm$, which leads to a lift of $S_\tx{sc}(\R) \to \mc{S}$ to a map $S_\tx{sc}(\R)_\pm \to \mc{S}_\pm$. But the cover $S_\tx{sc}(\R)_\pm$ splits canonically, because $\rho$ is divisible by $2$ in $X^*(S_\tx{sc})$. Therefore, we have a map $S_\tx{sc}(\R) \to \mc{S}_\pm$. Pulling back $\tau$ under this map produces a character $\tau_\tx{sc}$ of $S_\tx{sc}(\R)$. This character is algebraic since $S_\tx{sc}(\R)$ is compact, i.e. it is the restriction to $S_\tx{sc}(\R)$ of a (unique) element of  $X^*(S_\tx{sc})=\tx{Hom}_\tx{alg.grp}(S_\tx{sc}(\C),\C^\times)$. The regularity assumption on $\tau$ is that this element does not lie on any (absolute) root hyperplane. Note that this element of $X^*(S_\tx{sc})$ also coincides with the differential $d\tau_\tx{sc} : \tx{Lie}(S_\tx{sc})(\R) \to \tx{Lie}(\mb{S}^1)$ under the identifications $\tx{Lie}(S_\tx{sc})(\R)=[X_*(S_\tx{sc}) \otimes_\Z \C]^\Gamma = X_*(S_\tx{sc})\otimes_\Z i\R$ and $\tx{Lie}(\mb{S}^1)=i\R$. Hence $d\tau_\tx{sc}$ is a regular element of $X^*(S_\tx{sc})$, and specifies a Weyl chamber, i.e. a Borel $\C$-subgroup $B_\tau$ of $G_\tx{ad}$ containing $S_\tx{ad}$. Write $D_\tau$ in place of $D_{B_\tau}$ for the Weyl denominator \eqref{eq:weyldenom}. Since our convention (cf. \S\ref{sub:weyldenom}) is to normalize orbital integrals and characters by the absolute value of this denominator, we will only need $d_\tau=\tx{arg}D_\tau$. We compose $d_\tau$ with the map $\mc{S}_\pm \to S_\tx{ad}(\R)_\pm$. Both $\tau$ and $d_\tau$ are genuine functions of $\mc{S}_\pm$, so their quotient $\tau/d_\tau$ descends to $\mc{S}$.

In the second sum we have set $\tau'=\tau \cdot \rho_\tau^{-1}$, and $d_\tau'=d_\tau \cdot \rho_\tau^{-1}$, where $\rho_\tau=\rho_{B_\tau}$
and we are using the notation of \eqref{eq:dgb}. In this way, both numerator and denominator are functions on $\mc{S}$. Note that $\rho_B$ takes values in $\mb{S}^1$ because $S_\tx{ad}$ is elliptic.

\begin{dfn} \label{dfn:hcpar}
	We shall call the $\mc{G}$-conjugacy class of the pair $(\mc{S},\tau)$ the \emph{Harish-Chandra parameter} of the corresponding eds representation.
\end{dfn}

\begin{rem} \label{rem:hcpar}
	Given a pair $(\mc{S},\tau)$ as above we can consider the differential $d\tau : \mf{s}(\R) \to \tx{Lie}(\C^\times)=\C$, where $\mf{s}=\tx{Lie}(S)$. It is an $\R$-linear map which we may extend to a $\C$-linear map $\mf{s} \to \C$, and thus obtain an element of the dual space $\mf{s}^*$. We have the decomposition $\mf{s}=\mf{z} \oplus \mf{s}'$, where $\mf{z}=\tx{Lie}(Z(\mc{G}))$ and $\mf{s}'=\mf{s} \cap \mf{g}'$ and $\mf{g}'=\tx{Lie}(\mc{G}^\natural)$. This dualizes to $\mf{s}^*=\mf{z}^* \oplus (\mf{s}')^*$ and we can decompose $d\tau = d\tau_\mf{z} + d\tau'$. Note that $d\tau'$ coincides with the restriction of $d\tau$ to $\mf{s'} \subset \mf{s}$ and hence with the differential of $\tau_\tx{sc}$, i.e. $d\tau'=d\tau_\tx{sc}$. As discussed above we have $d\tau' \in X^*(S_\tx{sc}) \subset i\mf{s}'(\R)^*$. Finally note that the quotient $\mf{g}^* \to \mf{s}^*$ splits canonically by taking the subspace of $\mf{g}^*$ that is fixed by the coadjoint action of $\mc{S}$. The same holds for the quotient $\mf{g}' \to \mf{s}'$. Therefore we can identify $d\tau$ with an element of $\mf{z}^* \oplus i\mf{g}'(\R) \subset \mf{g}^*$.
\end{rem}

\begin{proof}[Proof of Theorem \ref{thm:eds}]
	When $\mc{G}=G(\R)$ with $G$ connected semi-simple simply connected $\R$-group, this theorem is a celebrated result of Harish-Chandra \cite{HCDSI}. The general case reduces to this by elementary considerations, which we now give for the convenience of the reader.

Let $\pi$ be an eds representation of $\mc{G}$. Since $\pi$ is irreducible and $Z(\mc{G})\cdot \mc{G}^\natural$ is of finite index in $\mc{G}$, the restriction of $\pi$ to $Z(\mc{G})\cdot \mc{G}^\natural$ is a finite direct sum of irreducible representations. They all share the same central character $\omega$, namely that of $\pi$. Therefore, each of them is of the form $\omega \cdot \pi^\natural$, with $\pi^\natural$ an irreducible representation of $\mc{G}^\natural$. The square-integrability of $\pi^\natural$ follows from that of $\pi$. The inflation of $\pi^\natural$ to $G_\tx{sc}(\R)$ is still square-integrable. By Harish-Chandra's result applied to $G_\tx{sc}(\R)$, there exists an anisotropic maximal torus $S_\tx{sc}$ of $G_\tx{sc}$ and a character $\tau_{\tx{sc}}$ of $S_{\tx{sc}}(\R)$ whose differential $d\tau_{\tx{sc}}\in X^*(S_{\tx{sc}})$ is regular, with the property that
$d\tau_{\tx{sc}}-\rho_\tau$ is in the same Weyl chamber as $\rho_\tau$, and so
the character of $\pi^\natural$ on $S_{\tx{sc}}(\R)$ is given by the character formula \eqref{eq:charfmla} for the group
 $G_{\tx{sc}}(\R)$ and the character $\tau_{\tx{sc}}$.

Since $\pi^\natural$ is pulled back from $\mc{G}^\natural$ and the central character of this pull-back is $\tau'_\tx{sc}|_{Z(G_\tx{sc})(\R)}$, we see that $\tau'_\tx{sc}$ factors through the surjection $S_\tx{sc}(\R) \to \mc{S}^\natural$. Its restriction to $\mc{S}^\natural \cap Z(\mc{G})$ matches the central character $\omega$ of $\pi$. Therefore $\tau'_\tx{sc}$ and $\omega$ splice to a character $\tau'$ of $Z(\mc{G}) \cdot \mc{S}^\natural$, which equals  $\mc{S}$ according to Lemma \ref{lem:ani}. Let $\tau=\tau' \cdot \rho_\tau$. Thus we have a pair $(\mc{S},\tau)$ as in the statement of the theorem. Since the $\mc{G}$-orbit of $\pi^\natural$ is determined by $\pi$, so is the $\mc{G}$-orbit of $(\mc{S},\tau)$.

	Conversely, given a $\mc{G}$-orbit of pairs $(\mc{S},\tau)$ we can invert the above procedure and obtain a $\mc{G}$-orbit of representations of $Z(\mc{G}) \cdot \mc{G}^\natural$, all sharing the same character of $Z(\mc{G})$, hence of the form $\omega \cdot \pi^\natural$, and with $\pi^\natural$ square-integrable.

	Next we claim that the induction of $\omega\cdot\pi^\natural$ to $\mc{G}$ is irreducible, and when $\omega\cdot\pi^\natural$ arises as an irreducible constituent of the restriction of an eds representation $\pi$ of $\mc{G}$, then $\pi$ is equal to that induction. Indeed, let $\pi_1$ be this induction. The latter claim reduces to the former by using Frobenius reciprocity to obtain a map $\pi \to \pi_1$ and using the irreducibility of $\pi$ and $\pi_1$. To show that $\pi_1$ is irreducible it suffices to check that the stabilizer of $\omega\cdot\pi^\natural$ for the action of $\mc{G}$ equals $\mc{Z}(G) \cdot \mc{G}^\natural$. If $g \in \mc{G}$ stabilizes $\omega\cdot\pi^\natural$, then it stabilizes the $G_\tx{sc}(\R)$-conjugacy class of $(S_\tx{sc},\tau_\tx{sc})$. Multiplying $g$ by an element of $\mc{G}^\natural$ we may assume it stabilizes the pair $(S_\tx{sc},\tau_\tx{sc})$. The regularity of $\tau_\tx{sc}$ now implies that the image of $g$ in $G_\tx{ad}(\R)$ lies in $S_\tx{ad}(\R)$, hence $g$ lies in $\mc{S}$, which by Lemma \ref{lem:ani} equals $Z(\mc{G})\cdot \mc{S}^\natural$ and hence lies in $Z(\mc{G})\cdot \mc{G}^\natural$. The claim has thus been proved.

	We have thus established that (isomorphism classes of) eds representations of $\mc{G}$ are in bijection with $\mc{G}$-conjugacy classes of pairs $(\mc{S},\tau)$ as in the statement of the theorem.

	Next we consider the character of a representation $\pi$ corresponding to $(\mc{S},\tau)$ evaluated at a regular element $s \in \mc{G}$. Since $\pi$ is induced from $\omega\cdot\pi^\natural$, the Frobenius formula gives
	\[ \Theta_\pi(s) = \sum_{\substack{g \in \mc{G}/Z(\mc{G})\cdot \mc{G}^\natural\\ g^{-1}sg \in Z(\mc{G})\cdot \mc{G}^\natural}} \Theta_{\omega\cdot\pi^\natural}(g^{-1}sg).\]
	Noting that $Z(\mc{G}) \cdot \mc{G}^\natural$ is normal in $\mc{G}$ we see that this character vanishes unless $s \in Z(\mc{G}) \cdot \mc{G}^\natural$. Next we take $s \in \mc{S} = Z(\mc{G}) \cdot \mc{S}^\natural$. Then the second condition in the sum is vacuous. Since all anisotropic tori in $G_\tx{ad}$ are conjugate under $G_\tx{sc}(\R)$ every coset in $\mc{G}/Z(\mc{G})\cdot \mc{G}^\natural$ has a representative lying in $N(\mc{S},\mc{G})$. The desired character formula for $\pi$ now follows from the corresponding formula for $\pi^\natural$.

	To show that $\pi$ is uniquely determined by its character restricted to $\mc{S}$ it is enough to observe that this is true for $\pi^\natural$ by Harish-Chandra's results, but the formula for $\pi$ shows that $\omega \cdot \pi^\natural$ is contained in the restriction of $\pi$ to $Z(\mc{G}) \cdot \mc{G}^\natural$.	
\end{proof}

% \subsection{Orbital integrals} \label{sub:orbit}

% Let $\gamma \in G(\R)_\tx{sr}$ be a strongly regular semi-simple element, i.e. one whose centralizer is a maximal torus $T$. Choose Haar measures $dg$ and $dt$ on $G(\R)$ and $T(\R)$. For a smooth compactly supported function $f : G(\R) \to \C$, its orbital integral is defined as
% \[ O_\gamma(f,dg/dt) = \int_{G(\R)/T(\R)} f(g \gamma g^{-1})dg/dt. \]
% Its stable orbital integral is defined as
% \[ SO_\gamma(f,dg/dt) = \sum_{\gamma'}O_{\gamma'}(f,dg/dt), \]
% where $\gamma'$ runs over a set of representatives for the $G(\R)$-conjugacy classes of strongly regular semi-simple elements of $G(\R)$ that are $G(\C)$-conjugate to $\gamma$. If $T'$ is the centralizer of such $\gamma'$, then any $g \in G(\C)$ with $g\gamma g^{-1}=\gamma'$ induces an isomorphism $\tx{Ad}(g) : T \to T'$ of $\R$-tori that depends only on $\gamma$ and $\gamma'$, but not on the choice of $g$. We use that isomorphism to transport $dt$ to $T'(\R)$.

\subsection{Endoscopic groups and double covers} \label{sub:covendo}

For the next few sections we will work with an arbitrary local field $F$; the case $F=\R$ is one example. Let $G$ be a connected reductive $F$-group.

The notion of endoscopic data is introduced in \cite[\S1.2]{LS87}, and is a variation of the notion of endoscopic pairs or endoscopic triples discussed in \cite{Kot84} and \cite{Kot86}. It can be described equivalently as follows. 

An endoscopic datum for $G$ is a tuple $(H,s,\mc{H},\eta)$ consisting of
\begin{enumerate}[label=(\arabic*)]
	\item a quasi-split connected reductive group $H$,
	\item an extension $1 \to \hat H \to \mc{H} \to \Gamma \to 1$ of topological groups,
	\item a semi-simple element $s \in Z(\hat H)$, and
	\item an $L$\-embedding $\mc{H} \to {^LG}$.
\end{enumerate}
It is required that
\begin{enumerate}[label=(\alph*)]
	\item the extension $\mc{H}$ admits a splitting by a continuous group homomorphism $\Gamma \to \mc{H}$,
	\item the homomorphism $\Gamma \to \tx{Out}(\hat H)$ provided by $\mc{H}$ coincides with the one provided by the extension $^LH$,
	\item $\eta$ identifies $\hat H$ with the identity component of the centralizer of $\eta(s)$ in $\hat G$,
	\item there exists $z \in Z(\hat G)$ such that $s\eta^{-1}(z) \in Z(\hat H)^\Gamma$.
\end{enumerate}
The map $\eta$ produces a $\Gamma$-equivariant embedding $Z(\hat G) \to Z(\hat H)$, that we will use without explicit notation.

An isomorphism $(H_1,s_1,\mc{H}_1,\eta_1) \to (H_2,s_2,\mc{H}_2,\eta_2)$ is an element $g \in \hat G$ that satisfies the following properties. First, $\tx{Ad}(g)\eta_1(\mc{H}_1)=\eta_2(\mc{H}_2)$. In particular, $\eta_2^{-1}\circ\tx{Ad}(g)\circ\eta_1$ is an $L$\-isomorphism $\mc{H}_1 \to \mc{H}_2$, and restricts to a $\Gamma$-equivariant isomorphism $Z(\hat H_1) \to Z(\hat H_2)$. The second condition is that the resulting isomorphism $\pi_0(Z(\hat H_1)/Z(\hat G)) \to \pi_0(Z(\hat H_2)/Z(\hat G))$ maps the coset of $s_1$ to the coset of $s_2$.

In this paper we are working with pure (resp. rigid) inner twists, and this necessitates a slight refinement of the notion of endoscopic datum. A \emph{pure refined} endoscopic datum is one in which it is required that $s \in Z(\hat H)^\Gamma$ in point (3), and this eliminates the need for condition (d). An isomorphism of such data is required to map the coset of $s_1$ to the coset of $s_2$ under $\pi_0(Z(\hat H_1)^\Gamma) \to \pi_0(Z(\hat H_2)^\Gamma)$, without dividing by $Z(\hat G)$. A \emph{rigid refined} endoscopic datum replaces $s \in Z(\hat H)^\Gamma$ by $\dot s \in Z(\hat{\bar H} )^+$. An isomorphism of such data is required to map the coset of $\dot s_1$ to the coset of $\dot s_2$ under $\pi_0(Z(\hat{\bar H_1})^+) \to \pi_0(Z(\hat{ \bar H_2})^+)$. We recall here that $\hat{\bar G}$ is the universal cover (as a complex Lie group) of $\hat G$, and $\hat{\bar H}$ is the fiber product of $\hat H \to \hat G \from \hat{\bar G}$ and the superscript $+$ denotes the inverse image in $\hat{\bar H}.$

Given an $L$\-parameter $\varphi : W_\R \to {^LG}$ and a semi-simple element $s \in S_\varphi$, where $S_\varphi=\tx{Cent}(\varphi,\hat G)$, one obtains a pure refined endoscopic datum as follows. Set $\hat H = \tx{Cent}(s,\hat G)^\circ$. The homomorphism  $\varphi : W_\R \to \tx{Cent}(s,\hat G) \to \tx{Out}(\hat H)$ factors through the projection $W_\R \to \Gamma$. There is a unique (up to isomorphism) quasi-split connected reductive $\R$-group $H$ with dual group $\hat H$ such that the homomorphism $\Gamma \to \tx{Out}(H)=\tx{Out}(\hat H)$ induced by the $\R$-structure of $H$ matches the one induced by $\varphi$. Set $\mc{H}=\hat H \cdot \varphi(W_\R)$, and let $\eta$ be the tautological inclusion $\mc{H} \to {^LG}$. In the rigid setting, the same construction works starting with $\dot s \in S_\varphi^+$, where $S_\varphi^+$ is the preimage in $\hat{\bar G}$ of $S_\varphi$.  

By construction the parameter $\varphi$ takes values in $\mc{H}$. However, the extensions $\mc{H}$ and $^LH$ of $\Gamma$ by $\hat H$ need not be isomorphic, and even if they are, there is no natural isomorphism between them. Therefore, $\varphi$ is \emph{not} a parameter for $H$ in any natural way. There are two ways to remedy this situation.

The classical approach is to choose a $z$-extension $H_1 \to H$ and an $L$\-embedding $\mc{H} \to {^LH_1}$ that extends the natural embedding $\hat H \to \hat H_1$. These choices (which always exist, cf. \cite[\S2.2]{KS99}) are called a $z$-pair. They provide a parameter $\varphi_1$ for $H_1$.

An approach introduced in \cite{KalHDC} is to extend the theory of double covers of tori from \cite{KalDC} to the setting of quasi-split connected reductive groups. The datum $\mc{H}$ then leads to a \emph{canonical} double cover $H(F)_\pm$ of $H(F)$ and a \emph{canonical} isomorphism $^LH_\pm \to \mc{H}$. In this way, $\varphi$ naturally becomes a parameter for $H(F)_\pm$, and there are no choices involved. We will write $\varphi' : W_\R \to {^LH_\pm}$ for this parameter, so that $\varphi$ is the composition of $\varphi'$ with the natural isomorphism $^LH_\pm \to \mc{H}$ and the inclusion $\mc{H} \to {^LG}$.

\subsection{Transfer of orbital integrals} \label{sub:transfer}

We continue with a local field $F$ (the main example for this note being $F=\R$), a connected reductive $F$-group $G$, and an endoscopic datum $(H,s,\mc{H},\eta)$ for $G$. We shall further assume that $F$ has characteristic zero, because Theorem \ref{thm:orbtrans} below has so far been proved only under this assumption. In the pure setting we demand $s \in Z(\hat H)^\Gamma$, and in the rigid setting we demand $s \in Z(\hat{\bar H})^+$. We will review here the transfer of orbital integrals, which is dual to the transfer of characters that is the subject of this paper.

As explained in \S\ref{sub:covendo}, in the classical setting order to formulate endoscopic transfer (both geometric and spectral), one has to make an arbitrary choice of a $z$-pair $(H_1,\eta_1)$, where $H_1 \to H$ is a surjective homomorphism of algebraic groups whose kernel is an induced torus, and $\eta_1 : \mc{H} \to {^LH_1}$ is an $L$\-embedding. 

Alternatively, \cite{KalHDC} shows that there is a canonical double cover $H(F)_\pm \to H(F)$ and a canonical isomorphism $^LH_\pm \to \mc{H}$, whose composition with $\eta$ then becomes an $L$\-embedding $^LH_\pm \to {^LG}$. In this framework, no auxiliary choices are needed.

The transfer of orbital integrals and characters between $G$ and $H$ is governed by the transfer factor. In the classical case this is a function
\[ \Delta : H_1(F)^\tx{sr} \times G(F)^\tx{sr} \to \C \]
that depends on the $z$-pair datum, while in the setting of covers it is a function
\[ \Delta : H(F)_\pm^\tx{sr} \times G(F)^\tx{sr} \to \C \]
that is genuine in the first argument. 

In the classical case, it is given as the product
\[ \epsilon \cdot \Delta_I^{-1}\Delta_{II}\Delta_{III_1}^{-1}\Delta_{III_2}. \]
The individual factors are defined in \cite{LS87}, except for $\epsilon$, which is defined in the more general twisted setting in \cite[\S5.3]{KS99}, and $\Delta_{III_1}$, whose relative definition is given in \cite{LS87}, but whose absolute definition is given in \cite{KalECI} in the setting of pure inner forms, and in \cite{KalRI} in the setting of rigid inner forms. The inverses appear due to the conventions of \cite[(1.0.4)]{KS12}, which we will use in this paper. The term $\Delta_{IV}$ is missing because we have normalized orbital integrals and characters by the Weyl denominator. We will not review the construction of the individual pieces here, as it has been reviewed in various other places, such as \cite[\S3.5,\S4.2,\S4.3]{KalIMS}. The individual factors depend on auxiliary data, known as $a$-data and $\chi$-data. The total factor depends on a choice of Whittaker datum, and $z$-datum.

In the case of covers, the transfer factor becomes the product
\[ \epsilon \cdot \Delta_I^{-1} \cdot \Delta_{III}.\]
The terms $\Delta_I$ and $\Delta_{III}$ are slightly different from the original ones, and are defined in \cite[\S4.3]{KalHDC}. Neither of them depends on the auxiliary $a$-data and $\chi$-data, but they are defined on certain covers of tori, and one could argue that the elements of those covers count as auxiliary data. The product does not require a $z$-datum, but it does depend on a Whittaker datum.

Recall the notation $O_\gamma(f)$ for normalized orbital integrals from \S\ref{sub:weyldenom}. We define the stable orbital integral
\[ SO_\gamma(f) = \sum_{\gamma'} O_{\gamma'}(f), \]
where $\gamma'$ runs over a set of representatives for the $G(F)$-conjugacy classes of those elements that are stably conjugate, i.e. $G(F_s)$-conjugate, to $\gamma$.

In the setting of covers, orbital integrals and stable orbital integrals are defined analogously, using the fact that every admissible isomorphism $T \to T'$ between maximal tori of $G$ maps $R(T,G)$ to $R(T',G)$ and hence lifts canonically to an isomorphism $T(F)_G \to T'(F)_G$. We apply this to the isomorphism induced by conjugation or stable conjugation. 

We now state the theorem asserting transfer of orbital integrals. In the case $F=\R$ it is a fundamental result of Shelstad, \cite{She82}, \cite{SheTE1}. In the case of $F/\Q_p$ it is a culmination of the work of many people, including Langlands and Shelstad \cite{LS87}, \cite{LS90}, Waldspurger \cite{Wal97}, \cite{Wal06ECC}, and Ngo \cite{Ngo10}. We state two versions, one using the cover $H(F)_\pm$ and one using a $z$-pair $(H_1,\eta_1)$. 

\begin{thm} \label{thm:orbtrans}
Let $f \in \mc{C}_c^\infty(G(F))$.
\begin{enumerate}
	\item There exists a genuine function $f^{H_\pm} \in \mc{C}_c^\infty(H(F)_\pm)$ such that for all $\dot\gamma \in H(F)_\pm^\tx{sr}$
	\[ SO_{\dot\gamma}(f^{H_\pm}) = \sum_\delta \Delta(\dot\gamma,\delta) O_\delta(f). \]
	\item Assume we have chosen a $z$-pair $(H_1,\eta_1)$. There exists a genuine function $f^{H_1} \in \mc{C}_c^\infty(H_1(F))$ such that for all $\gamma_1 \in H_1(F)^\tx{sr}$
	\[ SO_{\gamma_1}(f^{H_1}) = \sum_\delta \Delta(\gamma_1,\delta) O_\delta(f). \]
\end{enumerate}

In both cases $\delta$ runs over the set of $G(F)$-conjugacy classes in $G(F)^\tx{sr}$.
\end{thm}

\begin{dfn} \label{dfn:matching}
The functions $f$ and $f^{H_\pm}$ are called \emph{matching}. The functions $f$ and $f^{H_1}$ are called \emph{matching}.
\end{dfn}

\begin{rem} \label{rem:matchmeasures}
	Since the orbital integral $O_\delta(f)$ depends on the choice of measures on $G(F)$ and $T(F)$, $T=\tx{Cent}(\delta,G)$, the concept of matching functions also depends on the choice of measures on $G(F)$, $H(F)$, and all tori in $G$ and $H$. There is a way to synchronize the various tori, see Remark \ref{rem:measures}.

	We note further that the normalization of orbital integrals we are using (\S\ref{sub:weyldenom}) cancels the normalization of the transfer factor we are using (it is missing the $\Delta_{IV}$ piece) and therefore the notion of matching functions of Definition \ref{dfn:matching} is the same as that obtained by using un-normalized orbital integrals and a transfer factor containing the $\Delta_{IV}$ piece.
\end{rem}

For the computations of this paper it would be useful to review some basic properties of the transfer factor. The first property concerns the behavior of $\Delta$ under stable conjugacy in the variables $\gamma$ and $\delta$. If $\dot\gamma$, resp. $\gamma_1$, is replaced by a stable conjugate, then the value of $\Delta$ doesn't change. 

To speak about stably conjugacy in $\delta$, we need to be a bit more precise about the group $G$. In order to have a normalized transfer factor, one must realize $G$ as a pure (or rigid) inner form of its quasi-split form, thus fix a pure (or rigid) inner twist $(\xi,z) : G_0 \to G$, with $G_0$ quasi-split. It can be notationally convenient to write $\Delta(\dot\gamma,(z,\delta))$ in place of $\Delta(\dot\gamma,\delta)$ in order to record $z$. Then we have 
\begin{equation} \label{eq:tfstab}
	\Delta(\dot\gamma,(z_2,\delta_2)) = \Delta(\dot\gamma,(z_1,\delta_1)) \cdot \<\tx{inv}((z_1,\delta_1),(z_2,\delta_2)),\hat \jmath\,^{-1}(s)\>.
\end{equation}
This is in the setting of covers and we refer to \cite[Lemma 4.3.1]{KalHDC}. In the classical setting, where $\dot\gamma$ is replaced by $\gamma_1$, we refer to \cite[Definitions 4.2.7,4.3.11]{KalIMS}, or \cite[\S4.1]{LS87}.

To explain this formula let us denote by $\gamma \in H(F)$ the image of $\dot\gamma$ resp. $\gamma_1$, and let $S \subset H$ be the centralizer of $\gamma \in H(F)$. Let $T \subset G_{z_1}$ be the centralizer of $\delta_1$. Using the discussion of \S\ref{sub:adm} we obtain from the inclusion $S \to H$ a $\Gamma$-stable $\hat H$-conjugacy class of embeddings $\hat S \to \hat H$. Composing this with the inclusion $\eta : \hat H \to \hat G$ we obtain a $\Gamma$-stable $\hat G$-conjugacy class of embeddings $\hat S \to \hat G$, hence conversely a $\Gamma$-stable $G(F^s)$-conjugacy class of embeddings $S \to G$. Among them, there is a unique one that maps $\gamma$ to $\delta_1$. It is automatically defined over $F$, and we call it $j$. Its dual-inverse $\hat\jmath^{-1}$ is an isomorphism $\hat S \to \hat T$. Composing with the canonical inclusion $Z(\hat H) \to \hat S$ we can use it to transport the element $s \in Z(\hat H)^\Gamma$ to $\hat T^\Gamma$, where it can be paired with elements of $H^1(F,T)$ by Tate-Nakayama duality. This is in the setting of pure inner twists, and the setting of rigid inner twists is analogous.

The second property is related to the definition of $\Delta_{III}$. We  first present an idealized situation. Assume that an $L$\-embedding $^LH \to {^LG}$ exists and has been fixed. Assume further that $L$\-embeddings $^LS \to {^LH}$ and $^LT \to {^LG}$ have been fixed. These always exist if we use the Weil-forms of the $L$\-groups, but are not unique. The isomorphism $j : S \to T$ from the previous paragraph induces an $L$\-isomorphism $^LS \to {^LT}$. We now have four maps that fit in a square, but this square has no reason to commute. More precisely, there does exist an $L$\-isomorphism $^LS \to {^LT}$ that does make the square commute, but it is not necessarily the one induced from $j$. Rather, it differs from it by multiplication by an element $a \in H^1(W_F,\hat S)$. This element is the $L$\-parameter of a character of $S(F)$, which we may denote by $\<a,-\>$. By definition
\[ \Delta_{III}(\gamma,\delta)=\<a,\gamma\>. \]
Even though $\delta$ doesn't appear on the right, it influences the construction via the map $j$, which depends on $\delta$. 

The key property of this definition that we will need is the following. Assume given $L$\-parameters $\varphi_S : W_F \to {^LS}$ and $\varphi_T : W_F \to {^LT}$, with the property that composing them with the $L$\-embeddings $^LS \to {^LH}$, $^LH \to {^LG}$, and $^LT \to {^LG}$, provides a commutative triangle. If $\theta_S$ and $\theta_T$ are the characters of $S(F)$ and $T(F)$ corresponding to these parameters, then we have
\begin{equation} \label{eq:tfd3}
\Delta_{III}(\gamma,\delta) = \theta_T(\delta)/\theta_S(\gamma).	
\end{equation}
To see this one observes that in the diagram
\[ \xymatrix{
	&^LS\ar[dd]_a\ar[r]&^LH\ar[dd]\\
	W_F\ar[ru]^{\varphi_S}\ar[rd]_{\varphi_T}\\
	&^LT\ar[r]&^LG,
}
\]
the square commutes by definition of $a$, hence the trangle commutes by assumptions on $\varphi_S$ and $\varphi_T$.

We now comment on the practical, rather than idealized, situation. In the classical setting an $L$\-embedding $^LH \to {^LG}$ doesn't always exist. Instead, one has to choose a $z$-pair $(H_1,\eta_1)$ and then has $^LG \leftarrow \mc{H} \to {^LH_1}$. This makes the first variable of $\Delta_{III}$ an element of $S_1(F)$, where $S_1$ is the preimage of $S$ in $H_1$. Moreover, the $L$\-embeddings $^LS \to {^LH}$ and $^LT \to {^LG}$ are specified in terms of auxiliary data, called $\chi$-data, and the term $\Delta_{III}$ depends on that choice, as well as on the choice of $z$-pair.

In the setting of covers, all $L$\-embeddings are canonical, and no auxiliary data need to be chosen, but the $L$\-groups that are involved are those of certain covers. Therefore, $\Delta_{III}$, while independent of any choices, has its first variable coming from the double cover $S(F)_\pm$ of $S(F)$ that is the pull-back of $S(F)$ under the canonical double cover $H(F)_\pm \to H(F)$. With these provisos, the obvious analog of \eqref{eq:tfd3} still holds.

% this factor depends on the choice of $\chi$-data. This choice produces $L$\-embeddings $^LS \to {^LH}$ and $^LT \to {^LG}$. If we assume that $H_1=H$ and $\eta_1 : {^LH} \to {^LG}$ is an $L$\-embedding, then composing we obtain $L$\-embeddings $^LS \to {^LG}$ and $^LT \to {^LG}$. On the other hand we have the isomorphism $j : S \to T$ reviewed above, and it induces an $L$\-isomorphism $^LS \to {^LG}$. Composing again we now have two $L$\-embedings $^LS \to {^LG}$. They necessarily differ by an element of $H^1(W_\R,\hat S)$, which is the parameter of a character $\chi$ of $S(\R)$. Then $\Delta_{III}(\gamma,\delta)=\chi(\gamma)$. From this definition one sees the following property. If $\varphi : W_\R \to {^LS}$ is a the parameter of a character $\theta$ of $S(\R)$ and $\varphi' : W_\R \to {^LT}$ is the parameter of a character $\theta'$ of $T(\R)$, and if composing these parameters with $^LS \to {^LH} \to {^LG}$ and $^LT \to {^LG}$ we obtain the same parameter for $^LG$, then
% \begin{equation} \label{eq:tfd3}
% 	\Delta_{III}(\gamma,\delta) = \theta(\gamma_1,\delta))/\theta'((\gamma_1,\delta)),
% \end{equation}
% where we have transported $\theta_1$ and $\theta_1'$ to characters $\theta$ and $\theta'$ on $S_1 \times_S T$.

% The setting of non-trivial $z$-pair is similar, but more notationally heavy, and one has $\Delta_{III}(\gamma_1)=\chi(\gamma_1)$ for a character $\chi$ of $S_1(\R)$, where $S_1 \subset H_1$ is the preimage of $S$. 

Finally, we  need to review the factor $\epsilon$ and compute it in the case of the base field $\R$, where the computation is rather straightforward. Consider the universal maximal torus $T_0^G$ of $G$ and $T_0^H$ of $H$. The complexified character modules $V_G := X^*(T_0^G)\otimes_\Z\C$ and $V_H := X^*(T_0^H)\otimes_\Z\C$ are self-dual Artin representations of the same dimension. Choose an $\R$-pinning of the quasi-split form $G_0$ of $G$ and a non-trivial additive character $\Lambda : \R \to \C$, which combine to give the Whittaker datum fixed for $G_0$. Then
\[ \epsilon = \epsilon(1/2,V_G-V_H,\Lambda), \]
where we have use Langlands' convention \cite[(3.6.4)]{TateCor} for the $\epsilon$-factor. The pinning is used in the construction of $G_0$.

\begin{lem} \label{lem:epsilon}
Let $\Lambda(x)=e^{irx}$ with $r>0$. Then
\[ \epsilon = (-1)^{q(H)-q(G_0)}i^{r_G/2-r_H/2}, \]
where $r_G$ is the number of roots in the absolute root system of $G$, and $r_H$ is the analogous number for $H$.
\end{lem}
\begin{proof}
Let $A_0^G \subset T_0^G$ be the maximal split torus ($X_*(A_0^G)=X_*(T_0^G)^\Gamma$), and $S_0^G \subset T_0^G$ the maximal anisotropic torus ($X^*(S_0^G)=X^*(T_0^G)/X^*(T_0^G)^\Gamma$). Then $X^*(T_0^G)_\C=X^*(A_0^G)_\C \oplus X^*(S_0^G)_\C$, where we have abbreviated $\otimes_\Z\C$ by the subscript $\C$. One has $\epsilon(1/2,\textbf{1},\Lambda)=1$ and $\epsilon(1/2,\tx{sgn},\Lambda)=i$ according to \cite[(3.2.4)]{TateCor}, hence
\[ \epsilon(1/2,V_G,\Lambda)=i^{d-\dim(A_0^G)}, \]
where $d=\dim(T_0^G)$. We use the same computation for $H$ and conclude
\[ \epsilon=\frac{\epsilon(1/2,X^*(T_0^G)_\C,\Lambda)}{\epsilon(1/2,X^*(T_0^H)_\C,\Lambda)} = \frac{i^{d-\dim(A_0^G)}}{i^{d-\dim(A_0^H)}}=i^{\dim(A_0^H)-\dim(A_0^G)}. \]

The Iwasawa decomposition $\tx{Lie}(G_0)=\mf{a} \oplus\mf{n} \oplus \mf{k}$ shows  $2q(G_0)=\dim(\mf{a})+\dim(\mf{n})=\dim(A_0^G)+r_G/2$. We note that $q(G_0)$ is an integer, becaus e $G_0$ has an elliptic maximal torus, and $q(G_0)$ equals the number of positive non-compact roots with respect to any Weyl chamber. Therefore
\[ \dim(A_0^H)-\dim(A_0^G) = 2(q(H)-q(G_0))+(r_G/2-r_H/2).\qedhere \]
\end{proof}

\subsection{The Weyl integration formula and its stable analog} \label{sub:weylint}

In this section we work with an arbitrary local field $F$ of characteristic zero. While our intended application is $F=\R$, the arguments work for an arbitrary $F$ and do not admit any significant simplification for $F=\R$, so we take this opportunity to record the statement for a general $F$. 

Let $G$ be a connected reductive $F$-group. Given a maximal torus $T \subset G$, write $\Omega(T,G)=N(T,G)/T$ for its absolute Weyl group. This is a finite $F$-group and we may consider the group $\Omega(T,G)(F)$ of its $F$-points. Write $\Omega_F(T,G)=N(T,G)(F)/T(F)$. This is an abstract finite group and we have $\Omega_F(T,G) \subset \Omega(T,G)(F)$. The inclusion is often proper.

\begin{thm}[Weyl integration formula] \label{thm:weyl}
	Let $f$ be a smooth compactly supported functon on $G(F)$. Let $\mc{T}$ be a set of representatives for the $G(F)$-conjugacy classes of maximal tori of $G$.
	Then
	\[ \int_{G(F)}f(g)dg = \sum_{T \in \mc{T}}|\Omega_F(T,G)|^{-1}\int_{T(F)_\tx{sr}} |D_T(\gamma)|^{1/2}O_\gamma(f)d\gamma. \]
\end{thm}

In this theorem we have chosen a Haar measure $dg$ on $G(\R)$, and a Haar measure $d\gamma$ on $T(\R)$ for any $T \in \mc{T}$. The orbital integral $O_\gamma(f)$ is formed with respect to the quotient measure $dg/d\gamma$. The power $1/2$ occurs in the Weyl discriminant because $O_\gamma$ has been normalized in \S\ref{sub:weyldenom}.

\begin{cor}[Stable Weyl integration formula] \label{cor:stabweyl}
	Let $f$ be a smooth compactly supported functon on $G(F)$. Let $\mc{ST}$ be a set of representatives for the stable conjugacy classes of maximal tori of $G$.
	Then
	\[ \int_{G(F)}f(g)dg = \sum_{T \in \mc{ST}}|\Omega(T,G)(F)|^{-1}\int_{T(F)_\tx{sr}} |D_T(\gamma)|^{1/2}SO_\gamma(f)d\gamma. \]
\end{cor}
\begin{proof}
	Let $T \in \mc{ST}$ and let $T_1,\dots,T_n \in \mc{T}$ be those tori that  are stably conjugate to $T$. In fact, it is known that $n=1$ (a special feature of $F$) but we will not need to know that.
	
	It is enough to show that 
	\[ |\Omega(T,G)(F)|^{-1}\int_{T(F)_\tx{sr}} |D_T(\gamma)|^{1/2}SO_\gamma(f)d\gamma = \sum_{i=1}^n|\Omega_F(T_i,G)|^{-1}\int_{T_i(F)_\tx{sr}} |D_{T_i}(\gamma)|^{1/2}O_{\gamma_i}(f)d\gamma_i. \]
	
	Two elements of $T(F)_\tx{sr}$ are stably conjugate if and only if they are conjugate under $\Omega(T,G)(F)$, and are rationally conjugate if and only if they are conjugate under $\Omega_F(T,G)$. Since $SO_\gamma(f)$ and $|D_T(\gamma)|$ are invariant under stable conjugacy in $\gamma$, and $O_\gamma(f)$ is invariant under rational conjugacy in $\gamma$, and since the action of $\Omega(T,G)(F)$ on $T(F)_\tx{sr}$ is free,  we can write the above identity as 
	\[ \int_{T(F)_\tx{sr}/\Omega(T,G)(F)} |D_{T}(\gamma)|^{1/2} SO_\gamma(f)d\gamma = \sum_{i=1}^n\int_{T_i(F)_\tx{sr}/\Omega_F(T_i,G)} |D_{T_i}(\gamma)|^{1/2}O_{\gamma_i}(f)d\gamma_i. \]
	The domain of integration on the left is the set of elements of $T(F)_\tx{sr}$ up to stable conjugacy, and the domain of integration on the right is the set of elements of $T_i(F)_\tx{sr}$ up to rational conjugacy.

	Choose $g_1,\dots,g_n \in G(\C)$ so that $\tx{Ad}(g_i) : T \to T_i$ is an isomorphism of $F$-tori and induces an isomorphism of $F$-groups $\Omega(T,G) \to \Omega(T_i,G)$. For any $\gamma \in T(F)_\tx{sr}$, the set 
	\[ \bigcup_{i=1}^n [\Omega_F(T_i,G) \lmod \Omega(T_i,G)(F)] \cdot \tx{Ad}(g_i)\gamma \]
	represents the $G(F)$-conjugacy classes in the stable class of $\gamma$. Therefore 
	\[ SJ(\gamma,f) = \sum_{i=1}^n \sum_{w \in \Omega_F(T_i,G) \lmod \Omega(T_i,G)(F)} J(w\tx{Ad}(g_i)\gamma,f).\]
	Moreover $D_T(\gamma)=D_{T_i}(\tx{Ad}(g_i)\gamma)$.
	As $\gamma$ runs over all elements of $T(F)_\tx{sr}/\Omega(T,G)(F)$, $\tx{Ad}(g_i)\gamma$ runs over all elements of $T_i(F)_\tx{sr}/\Omega(T_i,G)(F)$.
\end{proof}

\begin{rem} \label{rem:weyl1}
	As discussed in the proof, we can rewrite the above formulas as
	\[ \int_{G(F)}f(g)dg = \sum_{T \in \mc{T}}\int_{T(F)_\tx{sr}/\Omega_F(T,G)} |D_T(\gamma)|^{1/2} O_\gamma(f)d\gamma \]
	and
	\[ \int_{G(F)}f(g)dg = \sum_{T \in \mc{ST}}\int_{T(F)_\tx{sr}/\Omega(T,G)(F)} |D_T(\gamma)|^{1/2} SO_\gamma(f)d\gamma \]
	respectively.
\end{rem}

\subsection{Endoscopic lifting of distributions}

In this subsection we continue to work with an arbitrary local field $F$ of characteristic zero. Let $(H,s,\mc{H},\eta)$ be an endoscopic datum of $G$. The transfer of orbital integrals given by Theorem \ref{thm:orbtrans} defines a dual transfer of distributions. Let $\mc{I}(G)$ denote the space of invariant distributions on $G(F)$. This space consists of linear functionals on $\mc{C}^\infty_c(G(F))$, with a certain continuity property in the case $F=\R$, and which are invariant under the conjugation action of $G(F)$. By a deep theorem of Harish-Chandra this invariance property is equivalent to the requirement that such a functional vanishes on any test function for which all regular semi-simple orbital integrals vanish. Let $\mc{S}(G) \subset \mc{I}(G)$ denote the subspace of \emph{stable} distributions, defined to be those that vanish on all test functions for which all stable regular semi-simple orbital integrals vanish.

Analogously we have the spaces $\mc{S}(H) \subset \mc{I}(H)$. But the transfer of functions doesn't involve $H(F)$. Rather, it involves either the canonical double cover $H(F)_\pm$, or a chosen $z$-extension $H_1(F)$. Therefore, we write $\mc{I}_\tx{gen}(H_\pm)$ for the space of $H(F)$-invariant (continuous when $F=\R$) linear functionals on the space $\mc{C}^\infty_{c,\tx{gen}}(H(F)_\pm)$ of genuine test functions on $H(F)_\pm$, and similarly $\mc{I}_\tx{gen}(H_1)$ for the space of $H(F)$-invariant (continuous when $F=\R$) linear functionals on the space $\mc{C}^\infty_{c,\tx{gen}}(H_1(F))$ of genuine test functions on $H_1(F)$.

\begin{dfn} \label{dfn:stabtrans}
The \emph{endoscopic lifting} is the linear map 
\[ \tx{Lift} : \mc{S}_\tx{gen}(H_\pm) \to \mc{I}(G) \quad \tx{resp.}\quad \tx{Lift} : \mc{S}_\tx{gen}(H_1) \to \mc{I}(G)\] 
defined by $\tx{Lift}(d)(f)=d(f^{H_\pm})$ resp. $\tx{Lift}(d)(f)=d(f^{H_1})$.
\end{dfn}

An invariant distribution can be given by integration against a locally integrable class function $\phi \in L^1_\tx{loc}(G(F))$, namely 
\[ d_\phi(f) = \int_{G(F)_\tx{rs}} |D(x)|^{-1/2}\phi(x)f(x)dx. \]
Here locally integrable means that $\phi$ is measurable on $G(F)$ and integrable on any compact subset of $G(F)$, and again $d_\phi$ depends on the choice of Haar measure $dx$. We say that the distribution $d_\phi$ is represented by $\phi$. We have inserted the factor $|D(x)|^{-1/2}$ as a means of normalization, following our earlier convention. Note that only the restriction of $\phi$ to $G(F)_\tx{rs}$ is relevant, since the complement of this set has measure zero.

The distribution $d_\phi$ will be stable if and only if the function $\phi$ is stably invariant, i.e. stable on regular semi-simple conjugacy classes. If we replace $G(F)$ by $H(F)_\pm$ or $H_1(F)$, then we will be interested in functions that are locally integrable, stably invariant, and anti-genuine, meaning they transform by the inverse of the  character of $H(F)_\pm \to H(F)$ resp. $H_1(F) \to H(F)$; for the double cover $H(F)_\pm$ this is equivalent to genuine, since the central character has order $2$.

We can now apply the Weyl integration formula and its stable analog to show that the endoscopic lifting of a stable distribution represented by a function is an invariant distribution that is again represented by a function, and moreover the two functions are related by an explicit formula.

\begin{lem} \label{lem:equi}
	Let $\phi^H$ be a stably-invariant locally integrable anti-genuine function on $H(F)_\pm$ resp. $H_1(F)$. Then 
	\[ \tx{Lift}(d_{\phi^H}) = d_{\phi^G}, \]
	where $\phi^G$ is the locally integrable class function on $G(F)$ given by the following formula
	\[ \phi^G(\delta) = \sum_{\gamma \in H(F)_\tx{rs}/\tx{st} } \Delta(\dot\gamma,\delta) \phi^H(\dot\gamma). \]
	\end{lem}
	\begin{proof}
	The proofs of the two cases $H(F)_\pm$ and $H_1(F)$ are almost identical, so we will only give the proof for $H(F)_\pm$.
	
	We have $d_{\phi^H}(f^{H_\pm}) = \int_{H(F)}|D^H(\gamma)|^{-1/2}\phi^H(\dot\gamma)f^{H_1}(\dot\gamma)d\gamma$, where $\dot\gamma \in H(F)_\pm$ is any preimage of $\gamma$. Note that, while the factors in the integrand depends on $\dot\gamma$, the dependence cancels out in the product, because both $\phi^H$ and $f^{H_\pm}$ are genuine.
	
	We apply the stable Weyl integration formula (Corollary \ref{cor:stabweyl}, but in the form of Remark \ref{rem:weyl1}) and use the stable invariance of $\phi^H$ to rewrite the above as
	\[ \sum_{T_H \in \mc{ST}_H} \int_{T_H(F)_\tx{sr}/\Omega(T_H,H)(F)}\phi^H(\dot\gamma) SO_{\dot\gamma}(f^{H_\pm})d\gamma.\]
	According to Theorem \ref{thm:orbtrans} this equals
	\begin{equation} \label{eq:x1}
	\sum_{T_H \in \mc{ST}_H} \int_{T_H(F)_\tx{sr}/\Omega(T_H,H)(F)}\phi^H(\dot\gamma)  \sum_\delta \Delta(\dot\gamma,\delta) O_\delta(f) d\gamma.
	\end{equation}
	We will now rework the indexing sets. Let $\mf{X}$ be the set of pairs $([[\gamma]],[\delta])$, where $[[\gamma]]$ is a stable class of strongly regular semi-simple elements of $H(F)$, $[\delta]$ is a rational class of strongly regular semi-simple elements of $G(F)$, and $\gamma$ and $\delta$ are related. Projecting onto the first coordinate provides a surjective map
	\[ \mf{X} \to \bigcup_{T_H \in \mc{ST}_H} T_H(F)_\tx{sr}/\Omega(T_H,H)(F) \]
	with finite fibers. Pulling back the measure on the target that is comprised of the various Haar measures on the tori $T_H(F)$, we obtain a measure on $\mf{X}$, and \eqref{eq:x1} becomes
	\begin{equation} \label{eq:x2}
	\int_{\mf{X}} \phi^H(\dot\gamma) \Delta(\dot\gamma,\delta) O_\delta(f) d([[\gamma]],[\delta]).
	\end{equation}
	On the other hand, projecting onto the second coordinate provides a surjective map
	\[ \mf{X} \to \bigcup_{T \in \mc{T}_G} T(F)_\tx{sr}/\Omega_F(T,G) \]
	and pulling back the measure on the target that is comprised of the various Haar measures on the tori $T(F)$, we obtain a measure on $\mf{X}$. Since the measures on the tori in $H$ and those on the tori in $G$ are synchronized as in Remark \ref{rem:measures}, the two measures on $\mf{X}$ agree. Therefore, we can rewrite \eqref{eq:x2} as 
	\[ 
	\sum_{T \in \mc{T}} \int_{T(F)_\tx{sr}/\Omega_F(T,G)}O_\delta(f) \sum_\gamma \Delta(\dot\gamma,\delta) \phi^H(\dot\gamma) d\delta,	
	\] 
	where now $\gamma$ runs over the set of stable classes of strongly regular semi-simple elements of $H(F)$. Applying the Weyl integration formula (Theorem \ref{thm:weyl}, in the form of Remark \ref{rem:weyl1}) we can rewrite this as 
	\begin{equation} \label{eq:x3}
	\int_{G(F)_\tx{sr}}f(\delta) |D^G(\delta)|^{-1/2}\sum_\gamma \Delta(\dot\gamma,\delta) \phi^H(\dot\gamma) d\delta,
	\end{equation}
	which is equal to $d_{\phi^G}(f)$ provided we define
	\begin{equation} \label{eq:x4}
		\phi^G(\delta) = \sum_\gamma \Delta(\dot\gamma,\delta) \phi^H(\dot\gamma).
	\end{equation}
	It remains to remark that this definition does indeed produce a locally integrable class function. The conjugation-invariance is clear from the corresponding property of the transfer factor. For the local integrability, it is enough to show that for any regular semi-simple $\delta$ there is an open neighborhood $U$ on which $\phi^G$ is integrable. This follows from the fact that there are only finitely many stable classes of elements of $H(F)$ that match $\delta$, and the local constancy of $\Delta(\dot\gamma,\delta)$. More precisely, let $\gamma_1,\dots,\gamma_n \in H(F)$ be representatives for the stable classes such that $\Delta(\dot\gamma,\delta) \neq 0$ and let $S_1,\dots,S_n$ be the centralizers of $\gamma_1,\dots,\gamma_n$. There is a unique admissible isomorphism $j_i : S \to S_i$ mapping $\delta$ to $\gamma_i$; here $S$ is the centralizer of $\delta$. Choosing $U$ small enough we can ensure that $j_1(\delta'),\dots,j_n(\delta')$ is a set of representatives for the stable classes that match $\delta'$, for all $\delta' \in U$. Shrinking $U$ further if required, the double cover $S_i(F)_\pm$ splits over $U_i=j_i(U \cap S(F))$, and \cite[Corollary 4.3.4]{KalHDC} shows that the function $\Delta(j_i(\delta'),\delta')$ is constant in $\delta' \in U$, where we have composed $j_i$ with an arbitrary splitting of the cover $S_i(F)_\pm$ over $U_i$. Now the local integrability of $\phi^G$ follows from that of $\phi^H$.
\end{proof}

\subsection{Derivatives of class functions}

This section is specific to the base field $\R$. In the last section we discussed distributions that are representably by locally integrable class functions. The distributions that appear in the endoscopic character identities are of this form. We will now recall a result of Harish-Chandra (Proposition \ref{pro:diff}) about differentiating class functions on a semi-simple Lie group, which will play an important technical role in the proof of the endoscopic character identities.

For the duration of this section we switch notation and let $G$ be a connected semi-simple Lie group (rather than an algebraic $\R$-group) and let $T \subset G$ be a Cartan subgroup. The main case of interest for us will be when $G$ is the group of $\R$-points of a connected semi-simple  algebraic $\R$-group and $T$ is the group of $\R$-points of a maximal torus. Let $\mf{g}_0$ and $\mf{t}_0$ be the Lie algebras of $G$ and $T$ and let $\mf{g}=\mf{g}_0\otimes_\R\C$ and $\mf{t}=\mf{t}_0\otimes_\R\C$ be their complexifications. We denote by $G'$ the open subset of strongly regular semi-simple elements of $G$ and set $T'=T \cap G'$.

Recall the $\C$-algebra $\mc{D}(G)$ of differential operators on $G$. More generally, we have the $\C$-algebra $\mc{D}(V)$ for an open subset $V \subset G$. It is a subalgebra of the space $\tx{End}_\C(\mc{C}^\infty(V))$ defined as follows. If $U \subset V$ is an open subset contained in a local coordinate chart, then $\mc{D}(U)$ is the subalgebra of $\tx{End}_\C(\mc{C}^\infty(U))$ generated by the following two kinds of operators: multiplication by an element of $\mc{C}^\infty(U)$, and partial derivative with respect to a local coordinate. An element $D \in \tx{End}_\C(\mc{C}^\infty(V))$ lies in $\mc{D}(V)$ if there is an open cover of $V$ by open subsets $U$, each lying in a coordinate chart, such that $D$ maps $\mc{C}^\infty(U)$ to itself and the resulting element of $\tx{End}_\C(\mc{C}^\infty(U))$ lies in $\mc{D}(U)$, for all $U$ in that open cover. 

The actions of $G$ on itself by left and right translations induce actions of $G$ on $\mc{C}^\infty(G)$, hence also on $\tx{End}_\C(\mc{C}^\infty(G))$. Both of these actions preserve the subalgebra $\mc{D}(G)$. More precisely, we have for $x,y \in G$, $f \in \mc{C}^\infty(G)$, and $D \in \mc{D}(G)$
\[ L_xf(y) = f(x^{-1}y), R_xf(y)=f(yx), L_xD = L_x \circ D \circ L_{x^{-1}}, R_xD = R_x \circ D \circ R_{x^{-1}}.\]
Furthermore, we have the conjugation action $C_x=L_x \circ R_{x^{-1}}$. An element of $\mc{D}(G)$ is called \emph{left-}, \emph{right-}, or \emph{conjugation-invariant}, if it is invariant under the appropriate action of $G$.

We recall that any $X \in \mf{g}$ acts as a left-invariant differential operator on the space of smooth functions on $G$ by the formula
\[ (Xf)(g) = \frac{d}{dt}\Big|_{t=0}f(g \cdot \exp(tX)). \]
The resulting map $\tau : \mf{g} \to \mc{D}(G)$ of complex vector spaces takes the Lie bracket on $\mf{g}$ to the commutator bracket of the associative algebra $\mc{D}(G)$, and thus extends to a homomorphism of $\C$-algebras
$\tau : \mc{U}(\mf{g}) \to \mc{D}(G)$.
Let $\mc{D}_l(G) \subset \mc{D}(G)$ denote the subalgebra of left-invariant differential operators. Then $\tau$ is an isomorphism onto $\mc{D}_l(G)$:
\[
\tau : \mc{U}(\mf{g}) \overset\simeq\longrightarrow \mc{D}_l(G)\subset \mc{D}(G)
\]
The restriction of $\tau$ to the center $\mc{Z}(\mf{g})$ of $\mc{U}(\mf{g})$ induces an isomorphism between $\mc{Z}(\mf{g})$ and the subalgebra of differential operators that are both left- and right-invariant.

Henceforth we suppress $\tau$ from the notation and identify an element of $\mc{U}(\g)$ with a left invariant differential operator.
We apply this to $\mc{U}(\mf{t})$ as well.

The Harish-Chandra isomorphism  $\gamma : \mc{Z}(\mf{g}) \overset\simeq\to \mc{S}(\mf{t})^\Omega$ of commutative $\C$-algebras identifies the center $\mc{Z}(\mf{g})$ of the universal enveloping algebra $\mc{U}(\mf{g})$  of $\mf{g}$ with the invariants under the Weyl group $\Omega=\Omega(T,G)$ in the symmetric algebra $\mc{S}(\mf{t})$, the latter being equal to the universal enveloping algebra $\mc{U}(\mf{t})$ of $\mf{t}$.

Here is the result of Harish-Chandra which we need.

\begin{pro} \label{pro:diff}
	Let $f$ be a smooth conjugation-invariant function $G' \to \C$. For any $z \in \mc{Z}(\mf{g})$ we have the identity
	\[ (|D_T|^{1/2} \cdot zf)|_{T'} =  \gamma(z)(f|_{T'} \cdot |D_T|^{1/2}). \]
\end{pro}

This result is not stated in exactly  this form in the works of Harish-Chandra, so we give the proof, starting with some preliminaries.

The first step is to recall that Harish-Chandra constructs in \cite[\S4]{HC_characters} a family of retractions $\mc{D}(G) \to \mc{D}_l(G)$ to the inclusion $\mc{D}_l(G) \to \mc{D}(G)$. This family is indexed by the set of points of $G$, and the retraction $\mc{D}(G) \to \mc{D}_l(G)$ associated to $y \in G$ is denoted by $D \mapsto D_y$. The left-invariant operator $D_y$ is called the ``local expression of $D$ at $y$'', and has the defining property that 
\begin{equation} \label{eq:diff1}
 (Df)(y) = (D_yf)(y)\quad(\forall f \in \mc{C}^\infty(G)).
\end{equation}
This is \cite[\S4, Corollary to Lemma 13, page 112]{HC_characters}.

The next step concerns the passage from $G$ to $T$. Suppose $a\in T'$. 
Let $U_G$ be an open neighborhood of $a$ in $G'$ and set
\[
\label{e:UT}
U_T = a^{-1}(U_G\cap T)\subset a^{-1}T'.
\]
This is an open neighborhood $T$ containing $1$.

 In \cite[\S 5]{HC_inv_eigen} Harish-Chandra constructs a map
\[ \delta_a : \mc{D}(U_G) \to \mc{D}(U_T)\]
which is characterized by the property 
\[  \delta_a(D)_y = \alpha_y(D_{ay}) \quad (\forall y \in U_T) \]
where $\alpha_y : \mc{U}(\mf{g}) \to \mc{U}(\mf{t})$ $(y\in a^{-1}T')$ is defined  in \cite[\S5, Corollary 1, page 463]{HC_inv_eigen}.

Now assume $f \in \mc{C}^\infty(U_G)$  is locally invariant. This means: the map $\R \to \R, t \mapsto f(\exp(tX)y\exp(-tX))$ has trivial derivative at $t=0$  for all $X \in \mf{g}$ and $y \in U_G$, cf. \cite[\S8]{HC_inv_dist}. In particular, the restriction to $U_G$ of any conjugation-invariant function on $G'$ is locally invariant.
Then 
\cite[Lemma 18]{HC_inv_eigen} gives the identity:
\begin{equation} \label{eq:diff3}
	D(f)(ay) = \delta_a(D)_y(f)(ay)\quad(y\in U_T).
\end{equation}

To aide the reader, we mention that \cite[Lemma 18]{HC_inv_eigen} uses the notation $\Xi'=\Xi'(a)=a^{-1}(\Xi\cap G')$ on page 462, which is \emph{not} a set of regular elements in the usual sense, but rather the $a^{-1}$-translate of a set of regular elements. We are taking $\Xi=T$, and $U_T=a^{-1}U_G \cap a^{-1}T'=a^{-1}(U_G\cap T')$. In particular, note that $1 \in U_T$.

We mention here that  \cite[Lemma 18]{HC_inv_eigen} has a typographical error: it states \eqref{eq:diff3} with $\delta_a(D)$ in place of
$\delta_a(D)_y$.  However $\delta_a(D)$ is a differential operator on $U_T$, but $ay\not\in U_T$, so $\delta_a(D)(ay)$ is not well defined.
However the last line of the proof  says
\[ D_{ay}f(ay) = \delta_a(D)_y(f)(ay), \]
which is exactly \eqref{eq:diff3}, when we note that by \eqref{eq:diff1}  the left hand side is $Df(ay)$.
Finally we note that \cite[Lemma 18]{HC_inv_eigen} is used only in the proof of loc. cit. Lemma 19, and it is in fact \eqref{eq:diff3} which is used.

The final step is to connect the operator $\delta_a(z)$, for $z \in \mc{Z}(\mf{g})$, to the operators $z$ and $\gamma(z)$.
Set 
\[
|\nu_a(h)| = |D_T(ah)|\quad(h\in T).
\]
Viewing $|\nu_a|$ as the operator of multiplication by this function, 
\cite[\S6, Lemma 13]{HC_inv_eigen} states that
\begin{equation} \label{eq:diff2}
\delta_a(z) = |\nu_a|^{-1/2} \circ \gamma(z) \circ |\nu_a|^{1/2}.
\end{equation}

We can now combine the three steps outlined so far to prove Proposition \ref{pro:diff}.
\begin{proof}[Proof of Proposition \ref{pro:diff}]
Suppose $a\in T'$. Let $y=1\in U_T$ (cf. \eqref{e:UT}). Apply \eqref{eq:diff3} with $D=z$ 
(recall we've identified $\mc{U}(\g)$ with $\mc{D}_l(\g)$ via $\tau$)
to conclude:
\[
 (zf)(a) = \delta_a(z)_1(f)(a) =[L_{a^{-1}}\circ\delta_a(z)_1](f)(1)
\]
Since $\delta_a(z)_1$ is left-invariant on $T$ we have
\[
[L_{a^{-1}}\circ\delta_a(z)_1](f))(1)=[\delta_a(z)_1](L_{a^{-1}}(f))(1)=\delta_a(z)(L_{a^{-1}}(f))(1)
\]
where the last equality is \eqref{eq:diff1}.
We now apply \eqref{eq:diff2} and see
\[ \delta_a(z)(L_{a^{-1}}(f))(1) = [|\nu_a|^{-1/2} \circ \gamma(z) \circ |\nu_a|^{1/2}](L_{a^{-1}}(f))(1). 
\]

A direct computation shows 
$$
|\nu_a|^{1/2} \circ L_{a^{-1}}=L_{a^{-1}} \circ |\nu_1|^{1/2}=L_{a^{-1}} \circ |\Delta_T|^{1/2}.
$$
Using this and the left-invariance of $\gamma(z)$ we see 
\[ 
\begin{aligned}
[|\nu_a|^{-1/2} \circ \gamma(z) \circ |\nu_a|^{1/2}](L_{a^{-1}}f)(1) &= L_{a^{-1}}\circ(|\Delta_T|^{-1/2} \circ \gamma(z) \circ |\Delta_T|^{1/2})(f)(1)\\
&=(|\Delta_T|^{-1/2} \circ \gamma(z) \circ |\Delta_T|^{1/2})(f)(a)
\end{aligned}
\]

\end{proof}

\section{Genericity of essentially discrete series representations} \label{sec:gen}

Work of Kostant \cite{Kos78} and Vogan \cite{Vog78} provides a description of which eds representations are generic (Definition \ref{dfn:generic} below). More precisely, \cite[Theorem 6.2(a,f)]{Vog78} shows that such a representation is \emph{large} if its Harish-Chandra parameter lies in a Weyl chamber for which all simple roots are non-compact, and \cite[Theorem L]{Kos78} shows that \emph{large} is equivalent to \emph{generic}. In order to make the internal structure of $L$-packets precise, one needs finer information -- a condition for an eds to be $\mf{w}$-generic for a given Whittaker datum $\mf{w}$ (Definition \ref{dfn:whit}). Since in general there are multiple choices of $\mf{w}$, this information is not obtainable from \cite[Theorem 6.2(a,f)]{Vog78}.
 
Proposition \ref{p:whittaker} of this section provides an answer to this question. It is an exact analog of a result in the p-adic case, originally due to DeBacker and Reeder \cite[Proposition 4.10]{DR10} under unramifiedness assumptions, but whose proof applies in much greater generality, see \cite[Lemma 6.2.2]{KalRSP} and \cite[\S4.4]{FKS}. This allows us to prove the strong form of Shahidi's generic packet conjecture \cite[\S9]{Sha90} for discrete series packets, originally established by Shelstad \cite{SheTE3} using a less direct argument. Note that the general form of the conjecture (for tempered $L$\-packets) reduces immediately to the case of discrete $L$\-packets.

\subsection{Whittaker data and regular nilpotent elements} \label{sub:whit}

We assume  $G$ is a quasi-split connected reductive
$\R$-group. Suppose $B$ is a Borel $\R$-subgroup of $G$, $N$ its
unipotent radical, and $\n=\mathrm{Lie}(N)$.
The real Lie group $N(\R)$ is connected, nilpotent,
and simply connected (see \cite[Theorem 6.46]{KnappLie}, where the assumption that $G$ is semi-simple is unnecessary). The connectedness of $N(\R)$ implies that any character $\eta : N(\R) \to \C^\times$ is determined by its differential $d\eta : \n(\R) \to \C$, which is an $\R$-linear form. Furthermore the exponential map $\exp : \n(\R) \to N(\R)$ is a diffeomorphism by \cite[Theorem 1.127]{KnappLie}, and we have $\eta(\exp(Y))=e^{d\eta (Y)}$ for $Y \in \n(\R)$.

\begin{dfn} \label{dfn:whit}
A  \emph{Whittaker datum} is a $G(\R)$-conjugacy class of  pairs  $(B,\eta)$ where $B$ is an $\R$-Borel subgroup
	and $\eta$ is a non-degenerate unitary character of $N(\R)$. We  write $\w=[(B,\eta)]$ for the $G(\R)$-conjugacy class of $(B,\eta)$. 		
\end{dfn}

We  now give two constructions of Whittaker data.

We say an element $X$ of $\g^*$ is \emph{nilpotent} if the $G$-orbit of $X$
is a (weak) cone (closed under multiplication by $\C^*$), and is
\emph{regular} if $\dim_{\Cent(G)}(X)=\mathrm{rank}(G)$.  Suppose
$X\in i\g(\R)^*\subset \g^*$ is a regular nilpotent element.  The
kernel of $iX$ intersects $\g(\R)$ in a Borel subalgebra $\overline{\mathfrak b}(\R)$, and in this way defines an $\R$-Borel subgroup $\overline B$ of $G$.
Choose an $\R$-Cartan subgroup $T$ of
$\overline B$, and let $B$ be the corresponding opposite Borel
subgroup (characterized by $B\cap \overline B=T$). 
Let $N$ be the nilradical of $B$.  Then $X$ defines a unitary
character of $N(\R)$ by
\[ 
\eta_X(\exp Y)=e^{X(Y)}\quad (Y\in\n(\R)).
\] 
Since all $\R$-Cartan subgroups of $B$ are $N(\R)$-conjugate,
\cite[Theorem 19.2]{Bor91} this is independent of the choice of $T$. We define $\w_X$ to be the $G(\R)$-conjugacy class of $(B,\eta_X)$. 

The second construction starts with a character $\Lambda : \R \to \mb{S}^1$ and an $\R$-pinning $\P=(T,B,\{X_\alpha\})$ of $G$ (thus $T$ and $B$ are defined over $\R$, $X_\alpha \in \mf{g}_\alpha(\C)$ for each absolute $B$-simple root of $T$ in $G$, and the set $\{X_\alpha\}$ is permuted by complex conjugation).  
Let $u_\alpha$ be the isomorphism of $\C$-groups taking $\mb{G}_a$ to $U_\alpha$ (it can be explicitly given as $u_\alpha(t)=\exp(t X_\alpha)$). 
If $N$ is the unipotent radical of $B$ then the
composition $\prod_{\alpha \in \Delta} U_\alpha \to N \to N/[N,N]$ is
an isomorphism of complex algebraic groups, which we compose with
$(u_\alpha)$. The inverse of the result, composed with the product map
$\prod_\alpha \mb{G}_a \to \mb{G}_a$, becomes a homomorphism of
algebraic groups $\gamma : N/[N,N] \to \mb{G}_a$ which is  defined over $\R$, and 
 $\eta_{\P,\Lambda}=\Lambda\circ \gamma$ is a non-degenerate unitary character of $N(\R)$.
The $G(\R)$-conjugacy class of $\eta_{\P,\Lambda}$ only depends on the $G(\R)$-conjugacy class of  $\P$.
  Write $\w_{\P,\Lambda}$ for the Whittaker datum
$[(B,\eta_{\P,\Lambda})]$. 

If we consider an element $i=\sqrt{-1} \in \C$ fixed, we obtain the ``standard'' character $\Lambda : \R \to \mb{S}^1$ given by $x \mapsto e^{ix}$. We can then relate the two constructions as follows. For each $B$-simple root $\alpha$ let $X_{-\alpha} \in \mf{g}_{-\alpha}$ be determined by $[X_\alpha,X_{-\alpha}]=H_\alpha$, where $H_\alpha=d\alpha^\vee(1) \in \tx{Lie}(T)$ is the coroot. Define $X_\mc{P} \in i\mf{n}(\R)^*$ by $(X_\mc{P},Y)=i\sum_\alpha [Y_\alpha,X_{-\alpha}]/H_\alpha$, where $Y_\alpha \in \mf{g}_\alpha$ is the projection of $Y$ to the direct summand $\mf{g}_\alpha \subset \mf{g}$.

\begin{lem} \label{lem:w1}
\begin{enumerate}
\item The map $X \mapsto \w_X$ is an $\tx{Aut}(G)(\R)$-equivariant bijection from the set of $G(\R)$-conjugacy classes of regular nilpotent elements of $i\g(\R)^*$ to the set of Whittaker data of $G$.
\item Fix a non-trivial unitary character $\Lambda$ of $\R$. The map $\P\mapsto \w_{\P,\Lambda}$ is an $\tx{Aut}(G)(\R)$-equivariant bijection from $G(\R)$-conjugacy classes of real pinnings to Whittaker data.
\item If we take $\Lambda(x)=e^{ix}$ then $\eta_{\mc{P},\Lambda}=\eta_{X_\mc{P}}$ and hence $\mf{w}_{\mc{P},\Lambda}=\mf{w}_{X_\mc{P}}$.
\end{enumerate}
\end{lem}
\begin{proof}

The $\tx{Aut}(G)(\R)$-equivariance of the maps in 1) and  2) are clear. We need to prove they are bijections.

1) We construct the inverse map.  Suppose $\w$ is the $G(\R)$-conjugacy class of
$(B,\eta)$. Choose a maximal $\R$-torus $T$ of $B$, use it to define
$\overline B$, and write $\g=\n\oplus \overline{\mathfrak b}$ as 
before. Then the unitary character $\eta$ of $N$ lifts to an element
of $i\n(\R)^*$. Using the decomposition we view this as an element of
$i\g(\R)^*$, which is nilpotent (because it lies in $i\n(\R)^*$) and regular (because the character is non-degenerate). As before the $G(\R)$-orbit of
this element is independent of the choice of $T$.

2) Let $(T,B,\{X_\alpha\})$ be an $\R$-pinning and set
$X_-=\sum_\alpha X_{-\alpha}$. Let $H \in \mf{t}(\R)$ be the sum of
the positive coroots. Since $\{H_\alpha\}$ forms a basis of the coroot
lattice there exist positive integers $n_\alpha$ such that
$H=\sum_\alpha n_\alpha H_\alpha$. Since $H$ is $\Gamma$-fixed, the
collection of integers $\{n_\alpha\}$ is $\Gamma$-stable. Then
$X_+=\sum_\alpha n_\alpha X_\alpha \in \mf{g}(\R)$ and $(X_+,H,X_-)$
forms an $\mf{sl}_2$-triple with all three elements in
$\mf{g}(\R)$. All three elements $X_+$, $H$, and $X_-$, are
regular. The intersection of their centralizers is the center of
$G$. In particular, the only unipotent element that centralizes the
triple $(X_+,H,X_-)$ is the identity element.

For injectivity, assume two $\R$-pinnings
$(T^{(i)},B^{(i)},\{X_\alpha^{(i)}\})$, $i=1,2$, lead to the same
regular nilpotent element $X_- \in \mf{g}(\R)$. Let $B^-$ be the
unique (due to regularity of $X_-$) Borel subgroup such that $X_-$ is
contained in the Lie algebra of its unipotent radical $U^-$. The two
triples $(X_+^{(i)},H^{(i)},X_-)$ are then conjugate by a unique
element of the centralizer of $X_-$ in $U^-$. This element must then
lie in $U^-(\R)$. Since $H^{(i)}$ and $X_+^{(i)}$ are regular
elements, after conjugating by that element of $U^-(\R)$ we may assume
that $T^{(1)}=T^{(2)}$ and $B^{(1)}=B^{(2)}$. But then each
$X_\alpha^{(i)}$ can be recovered as the $\alpha$-component of
$X_+^{(i)}$. We conclude that the two pinnings are indeed conjugate
under $G(\R)$.

For surjectivity consider a regular nilpotent element
$X_- \in \mf{g}(\R)$ and let $B^-$ be the unique Borel subgroup such
that $X_-$ is contained in the Lie algebra of the unipotent radical
$U^-$ of $B^-$. If we can find a maximal torus $T \subset B^-$ such
that the decomposition of $X_-$ according to the root spaces for $T$
has non-zero components only for the $B^-$-simple roots, then writing
$X_-=\sum_\alpha X_{-\alpha}$ accordingly and setting
$X_\alpha \in \mf{g}_\alpha$ such that
$[X_\alpha,X_{-\alpha}]=H_\alpha$ we obtain an $\R$-pinning
$(T,B,\{X_\alpha\})$ giving rise to $X_-$, where $B$ is the
$T$-opposite of $B^-$. In order to find such a $T$, we first choose an
arbitrary maximal $\R$-torus $T \subset B^-$ and want to conjugate
this torus under $U^-(\R)$ to obtain the desired torus. This is
equivalent to conjugating $X_-$ under $U^-(\R)$ to achieve that its
components for roots that are not $B^-$-simple are all zero. To see
that this is possible, write $X_-=\sum_\alpha X_{-\alpha}$, where
$\alpha$ runs over \emph{all} roots of $T$ in $B^-$, and set
$X_-'=\sum_\alpha X_{-\alpha}$, where now $\alpha$ runs over
\emph{only the simple} roots of $T$ in $B^-$. Thus
$X_--X_-' \in [\mf{n}^-,\mf{n}^-](\R)$. Since both $X_-$ and $X_-'$
are regular and contained in $\mf{n}^-$, they are conjugate under an
element of $B^-(\C)$. But conjugation by $U^-(\C)$ doesn't change the
simple components, while conjugation by $T(\C)$ does so unless the
element lies in $Z_G(\C) \subset T(\C)$. Therefore we conclude that
$X_-$ and $X_-'$ are conjugate under an element $u \in U^-(\C)$. Now
$u^{-1}\sigma(u)$ lies in the centralizer of $X_-$ in $U^-(\C)$, which
is a connected unipotent subgroup $U'(\C)$ of $U^-(\C)$. But
$H^1(\Gamma,U'(\C))=\{0\}$ and we see that $uU'(\C)$ intersects
$U^-(\R)$, i.e. $X_-$ and $X_-'$ are indeed conjugate under $U^-(\R)$,
as claimed.

3) It is enough to check that $\eta_{\mc{P},\Lambda}(\exp(Y))=\eta_{X_\mc{P}}(\exp(Y))$ for $Y \in \mf{n}(\R)$. We have $Y \in \sum_\alpha t_\alpha X_\alpha + [\mf{n},\mf{n}]$ for some $t_\alpha \in \C$ with $\sum_\alpha t_\alpha \in \R$. Then $X_\mc{P}(Y)=i\sum_\alpha t_\alpha$, so $\eta_{X_\mc{P}}=e^{i\sum_\alpha t_\alpha}$, but we also have $\eta_{\mc{P},\Lambda}=e^{i\sum_\alpha t_\alpha}$ when $\Lambda(X)=e^{ix}$.
\end{proof}

In the literature Whittaker models are usually defined by a regular
nilpotent orbit in $\g(\R)$ (not the dual). It is helpful to relate
these formulations.

To do so, fix a non-degenerate, $G(\R)$-invariant bilinear form $\kappa$ on $\g'(\R)=[\g,\g](\R)$, and extend it by complex-linearity to 
a $G(\C)$-invariant form on $\g'$, still denoted $\kappa$ (for example, we could take the Killing form). Suppose $T\subset B$ are a Cartan and Borel subgroup, both defined over $\R$. 
Let $\overline B$ be the $T$-opposite Borel subgroup, and let $N,\overline N$ be the nilpotent radicals of $B,\overline B$, respectively.

Given $\P,\Lambda$ and $\kappa$ write $\Lambda(t)=e^{st}$ for some $s\in i\R$ and define
$$
X_-(P,\kappa,\Lambda)=s\sum_\alpha \kappa(X_{-\alpha},X_\alpha)^{-1}X_{-\alpha}\in i\overline n(\R).
$$
Here the sum is over the absolute $B$-simple roots, and as above $X_{-\alpha}$ satisfies $[X_{\alpha},X_{-\alpha}]=H_\alpha$.
Then a straightforward calculation gives
$$
\eta_{\P,\Lambda}(\exp Y)=e^{\kappa(X_-(\P,\kappa,\Lambda),Y)}\quad (Y\in \n(\R)).
$$
Note that, in particular, the right hand side is independent of $\kappa$.

In other words let $\phi_\kappa$ be the isomorphism $\overline\n(\R)\rightarrow \n(\R)^*$ given by
$$
\phi_\kappa(Y)=\kappa(X,Y)\quad (Y\in \n(\R)).
$$
Then
$$
\eta_{\P,\Lambda}=\eta_{\phi_\kappa(X_-(\P,\kappa,\Lambda))}.
$$
See \cite[Proof of Lemma 6.2.2]{KalRSP}.

It is helpful to understand the set of all Whittaker data.
Let
$$
Q(G(\R))=G_{\mathrm{ad}}(\R)/\mathrm{ad}(G(\R)).
$$
This is a finite abelian $2$-group, which 
acts by automorphisms on $G(\R)$, and therefore  induces an action of $Q(G(\R))$ on the set of Whittaker data.

\begin{lem}
\label{l:Q}
	$Q(G(\R))$ acts simply transitively on each of:
\begin{enumerate}
\item the set of Whittaker data;
\item the set of regular nilpotent orbits in $i\g(\R)^*$;
% \item the set of generic eds representations in an L-packet. \warn{it's too early to talk about $L$-packets here, we will only construct them in the next section. it's also not necessary in my view}
\end{enumerate}
\end{lem}

See \cite[Lemma 14.14]{ABV92}.

\begin{proof}
	  We use Lemma \ref{lem:w1}(1) and the fact that $\mc{P} \mapsto \w_{\mc{P},\Lambda}$ is $G_\tx{ad}(\R)$-equivariant, to reduce to the transitivity of the action of $G_\tx{ad}(\R)$ on the set of $\R$-pinnings. Since all Borel pairs over $\R$ are conjugate under $G(\R)$, this reduces the statement to the simple transitivity of the action of $T_\tx{ad}(\R)$ on the possible choices of non-zero simple root vectors. If the simple root $\alpha$ is real, then so is the corresponding fundamental coweight $\varpi_\alpha$, which then gives a 1-parameter subgroup $\R^\times \to T_\tx{ad}(\R)$, which clearly acts simply transitively on $\mf{g}_\alpha(\R) \sm \{0\}$. If the simple root is complex, then a pair $(X_\alpha,X_{\sigma\alpha})$ contributing to the pinning is an element of $(\mf{g}_\alpha \oplus \mf{g}_{\sigma\alpha})(\R) = \C$. The fundamental coweight $\varpi_\alpha$ induces a 1-parameter subgroup $\C^\times \to T_\tx{ad}(\R)$ which again acts simply transitively on $(\mf{g}_\alpha \oplus \mf{g}_{\sigma\alpha})(\R)$.

The second assertion is immediate from Lemma \ref{lem:w1}(2).
\end{proof}

For the next lemma recall the notation $\Omega_\R(S,G)$ from \S\ref{sub:weylint}.

\begin{lem} \label{lem:g2}
	Assume there exists and fix an elliptic maximal torus $S \subset G$. The inclusion $N(S_\tx{ad},G_\tx{ad})(\R) \to G_\tx{ad}(\R)$ induces a bijection 
	\[ \Omega_\R(S_\tx{ad},G_\tx{ad})/\Omega_\R(S,G) \to Q(G(\R)). \]
\end{lem}
\begin{proof}
	This follows from the fact that any two elliptic maximal tori of $G$ are conjugate under $G(\R)$ -- this is a well-know fact that is seen as follows. It is enough to show conjugacy under $G_\tx{der}(\R)$, so we may assume $G$ is semi-simple, hence $S$ is anisotropic. Then $S(\R)$ is compact \cite[\S24.6(c)]{Bor91}, hence contained in a maximal compact subgroup of $G(\R)$. But any two maximal compact subgroups of $G(\R)$ are conjugate \cite[\S24.6(a)]{Bor91}, and any two maximal tori of a compact Lie group are conjugate \cite[Corollary 4.35]{KnappLie}.
\end{proof}

\subsection{Kostant sections}
\label{s:kostant}

Let $X\in\g$ be a regular nilpotent element. It can be extended to an $SL(2)$-triple $(X,H,Y)$, see \cite[Chapter 3]{CM}.  The \emph{Kostant Section} $\K(X)$ associated to 
$(X,H,Y)$ is the affine space $X+\Cent_\g(Y) \subset \g$.
Kostant showed \cite{Kos63} that the Kostant section meets every regular orbit in a unique point. If $X,H,Y \in \g(\R)$, then $\K(X)$ is Galois stable. Since any two choices of triples that contain $X$ and lie in $\g(\R)$ are conjugate under the centralizer of $X$ in $G(\R)$, it is clear that the $G(\R)$-conjugacy class of $\K(X)$ only depends on the $G(\R)$-conjugacy class of $X$. Note that any $X \in \g(\R)$ can be extended to a triple $(X,H,Y)$ with $H,Y \in \g(\R)$.

Suppose $X\in \g(\R)$ and $\O$ is a regular orbit in $\g(\C)$ which is defined over $\R$, i.e. stable under complex conjugation. Then the unique point in $\K(X)\cap \O$ is also stable under complex conjugation, i.e. contained in $\g(\R)$. The $G(\R)$-orbit of $\K(X)\cap \O$ depends only on $\O$ and the $G(\R)$-orbit of $X$.

We can just as well use elements $X \in i\g(\R)$. Then $\K(X)$ is sent to $-\K(X)$ by complex conjugation.

We now define the analog of a Kostant section for $\g^*$. Choose an isomorphism
$\phi : \g \to \g^*$ of $\C$-vector spaces that is 
equivariant for the
action of $G$ and of the Galois group.
Then for a regular nilpotent element $X\in \g^*$ define
$\K(X)\subset \g^*$ as $\K(X) = \phi(\K(\phi^{-1}(X)))$. The
rationality and uniqueness properties of $\K(X)$ are inherited from
the analogous properties in the case of $\g$. 

Choosing the isomorphism $\phi$ is equivalent to choosing a non-degenerate, $G$
and Galois invariant, bilinear form on $\g$. For example we can use any
extension of the Killing form on the derived algebra by a non-degenerate bilinear form on the
center. In fact we claim the definition of $\K(X)$ is independent of
the choice as we now show.

Decompose  $\g = \mf{z} \oplus \g_1 \oplus \dots \oplus \g_n$, where $\mf{z}$ is the center of $\g$ and $\g_i$ are the simple factors. This decomposition is canonical, up to permuting the $\g_i$.  The dual decomposes accordingly as $\g^*=\mf{z}^*\oplus \g_1^* \oplus \dots \oplus \g_n^*$, where $\mf{z}^*$ is the annihilator of $\mf{g}_1 \oplus \dots \oplus \mf{g}_n$, and also equal to the fixed point space for the coadjoint action of $G$, and $\mf{g}_i^*$ is the annihilator of $\mf{z} \oplus \mf{g}_1 \oplus \dots \oplus \mf{\hat  g}_i \oplus \dots \oplus \mf{g}_n$ (with ``hat'' signifying that we are omitting this summand), and also equal to the dual space of $\mf{g}_i$. The isomorphism $\phi$ breaks up as $\phi_\mf{z}\oplus\phi_1 \oplus \dots \oplus \phi_n$, where $\phi_\mf{z} : \mf{z} \to \mf{z}^*$ is an isomorphism of $\C$-vector spaces equivariant under complex conjugation, and $\phi_i : \g_i \to \g_i^*$ is an isomorphism of complex vector spaces equivariant under $G$ (equivalently, its simple piece $G_i$) and complex conjugation. Of these, $\phi_\mf{z}$ is arbitrary, while $\phi_i$ is uniquely determined up to a non-zero real scalar. Now $X=X_1+\dots+X_n$ and $Y=Y_1+\dots+Y_n$ break up accordingly (neither has a central piece), and we have 
\[ \K(\phi^{-1}(X)) = \mf{z} \oplus \bigoplus (\phi_i^{-1}(X_i)+\Cent_{\g_i}(\phi_i^{-1}(Y_i))),  \]
hence
\[ \K(X) = \mf{z}^* \oplus \bigoplus (X_i+\phi_i(\Cent_{\g_i}(\phi_i^{-1}(Y_i)))). \]
But it is clear that the term $\phi_i(\Cent_{\g_i}(\phi_i^{-1}(Y_i)))$ does not change if we multiply $\phi_i$ by a non-zero scalar.

Note in particular that if we decompose $\mf{g}=\mf{z} \oplus \mf{g}'$ and analogously $\mf{g}^*=\mf{z}^* \oplus (\mf{g}')^*$, with $\mf{g}'=[\g,\g]$, then a regular nilpotent element $X \in \g$ lies in $\g'$ and $\K(X)=\mf{z}\oplus \K(X)'$ with $\K(X)'=\K(X) \cap \mf{g}'$, and dually a regular nilpotent element $X \in \g^*$ lies in $(\g')^*$ and $\K(X)=\mf{z}^* \oplus \K(X)'$ with $\K(X)'=\K(X) \cap (\mf{g}')^*$.

\subsection{Generic eds representations} \label{sub:gen}

\begin{dfn} \label{dfn:generic}
Suppose $\w=[(B,\eta)]$ is a Whittaker datum. We say that a representation $\pi$ of $G(\R)$ is \emph{$\w$-generic} if there is a non-zero smooth vector $v$ in the space of $\pi$ such that $\pi(x)(v)=\eta(x)v$ for all $x\in N(\R)$. We say $\pi$ is \emph{generic} if it is $\w$-generic for some $\w$. 
\end{dfn}

\begin{rem}
	It is clear from the definition that $\pi$ is $\mf{w}$-generic if and only if its restriction to $G_\tx{sc}(\R)$ is so. We could therefore immediately restrict to the case where $G$ is semi-simple and simply connected. We will not do this in order to make referencing this section easier. But notice that this observation allows one to treat all groups discussed in \S\ref{sub:essds}.
\end{rem}

Suppose $\pi$ is an irreducible eds representation and let $(S,\tau)$ be its Harish-Chandra parameter as in Remark  \ref{rem:hcpar}, and write $d\tau=d\tau_\mf{z}+d\tau' \in \mf{z}^* \oplus i\mf{g}^*(\R) \subset \g^*$ accordingly. Let $\O_\pi \subset \mf{z}^* \oplus i\mf{g}'(\R)^* \subset \mf{g}^*$ be the $G(\R)$-orbit of $d\tau$, and $\O_\pi' \subset i\g'(\R)^*$ be the $G(\R)$-orbit of $d\tau'$. Given a regular nilpotent element $X\in i\g(\R)^*$ we have the Kostant section $\K(X) = \mf{z}^* \oplus \K(X)'$ defined in \S\ref{s:kostant}. The purpose of this subsection is to prove the following result.

\begin{pro}
  \label{p:whittaker}
Suppose $\pi$ is an eds representation and $X\in i\g(\R)^*$ is a regular nilpotent element. The following are equivalent
\begin{enumerate}
	\item $\pi$ is $\w_X$-generic;
 	\item $\O_\pi$ meets $\K(X)$;
  	\item $\O_\pi'$ meets $\K(X)'$.
\end{enumerate}
\end{pro}

We begin with some preparations. If $W$ is a subset of a real or complex vector space $V$, define  $\AC(W)$,  the  {\it asymptotic cone} of $W$ as in   \cite[Proposition 3.7]{bvlocal}, \cite[Definition 2.9]{avav}:
$$
\AC(W)=\{v\in V\mid \exists t_n\in \R_{>0},t_n\rightarrow 0,w_n\in W,\lim_{i\rightarrow\infty}t_n w_n=v\}
$$
This is a closed cone. 
%If $\O$ is a $G(\R)$-orbit it  is a finite union of nilpotent orbits in $\g(\R)$.

\begin{lem} \label{lem:g1}
  Let $\epsilon \in \{1,-1,i,-i\}$. Let $\O_\R\subset \mf{z}^* \oplus \epsilon\g'(\R)^* \subset \g^*$ be a regular $G(\R)$-orbit and $X\in \epsilon\g(\R)^*$ be a regular nilpotent element.
Then  $X\in \AC(\O_\R)$ if and only if $\K(X)$ meets $\O_\R$.
\end{lem}

\begin{proof}
Choose an isomorphism $\phi : \g \to \g^*$ as in \S\ref{s:kostant}. It is immediate that $\phi(\AC(\phi^{-1}(\O_\R)))=\AC(\O_\R)$. This reduces the proof to the analogous statement but with $\g^*$ replaced by $\g$. It is also clear that $\epsilon(\AC(\epsilon^{-1}(\O_\R)))=\AC(\O_\R)$, so we may take $\epsilon=1$. We thus have a regular nilpotent element $X \in \g(\R)$ and a regular $G(\R)$-orbit $\O_\R \subset \g(\R)$. 

Assume first that $X \in \AC(\O_\R)$. By \cite{Kos63}, $\K(X)$ is transversal to any $G(\C)$-orbit in $\g(\C)$. Therefore the set $\tx{Ad}(G(\C))\K(X)$ contains an open ball around $X$ in $\g(\C)$.
This remains valid over $\R$ as well: $\tx{Ad}(G(\R))\K(X)(\R)$ contains an open ball around $X$ in $\g(\R)$.

Choose $t_n \in \R_{>0}$ with $t_n \to 0$ and $w_n \in \O_\R$ such that $t_nw_n \to X$. We conclude that there exists $g \in G(\R)$ such that $t_nw_n \in \tx{Ad}(g)\K(X)(\R)$ for some $n$. This is equivalent to $\tx{Ad}(g)^{-1}w_n \in t_n^{-1}(\K(X))(\R)$. Since $\O_\R$ is a $G(\R)$-orbit, we see $\O_\R \cap t_n^{-1}(\K(X))(\R) \neq \varnothing$. 

We now use the elementary facts that $t\K(X)=\K(tX)$ and $\tx{Ad}(h)\K(X)=\K(\tx{Ad}(h)X)$ for $t \in \R$ and $h \in G(\C)$.
Also it is well known that the $G(\R)$-orbit of $X$ is a cone. (To see this, apply the real version of the Jacobson-Morozov theorem
\cite[Theorem 9.2.1]{CM} to  find a homomorphism $\phi:SL_2\rightarrow G$ such that $d\phi\begin{pmatrix}0&1\\0&0\end{pmatrix}=X$.
Then $\Ad(\phi(\mathrm{diag}(t,t^{-1}))(X)=t^2X$.)
Together these imply
\begin{equation}
\label{e:scale}
G(\R)\cdot \K(X)=G(\R)\cdot \K(tX) \quad \text{ (for all $t>0$).}
\end{equation}
The result follows, since
$\O_\R \cap t_n^{-1}(\K(X))(\R)=\O_\R \cap G(\R)\cdot(t_n^{-1}(\K(X))(\R))=
\O_\R \cap G(\R)\cdot(\K(X)(\R))=\O_\R\cap \K(X)(\R)$.

For the other direction, suppose $\K(X)$ meets $\mathcal O_\R$.
This means there is an $SL(2)$-triple $(X,H,Y)$, and an element $Z\in \Cent_\g(Y)$ such that $W=X+Z \in \O$. 
By \cite{Kostant59}, the centralizer  $\mathrm{Cent}_{\g}(Y)$ is abelian; furthermore there exists a Cartan subalgebra $\mathfrak{h} \subset \g$, and a positive system $\Delta$ for the roots of $\mathfrak{h}$ in $\g$, such that $E$ is contained in the sum of $\Delta$-positive root spaces, while $\mathrm{Cent}_{\g}(F)$ is contained in the sum of $\Delta$-negative roots.  Therefore, if we set $h(t) = \exp(tH)$ for $t \in \R_+$, then 
\[\begin{cases} \mathrm{Ad}(h(t)) \cdot F \longrightarrow 0 & \text{as $t \to +\infty$, while} \\
 \mathrm{Ad}(h(t)) \cdot E = f(t) E & \text{ with $f(t) >0$ and $f(t)\to +\infty$ as $t \to \infty$.}\end{cases}\]
Set $Y(t) = \mathrm{Ad}(h(t)) Y$; then $f(t) \left[ E + Y(t) \right] = \mathrm{Ad}(h(t)) W$ lies on the orbit~$\O$. Thus there exists a continuous path in $\R^+ \cdot \O_\R$ with $E=X$ as a limit point; by definition this means $X \in \AC(\O_\R)$.
\end{proof}

For a closely related result see \cite[Proposition 3.5]{fm}.
  
\begin{proof}[Proof of Proposition \ref{p:whittaker}]

Since $\K(X)=\mf{z}^* \oplus \K(X)'$ and $\O_\pi = d\tau_\mf{z} + \O_\pi'$, the equivalence of 2. and 3. is clear.

Suppose $\pi$ is $\w_X$-generic.
We define the wave-front set $\WF(\pi)$ as in \cite[Section 3]{matumoto}. By definition this is a subset of $i\g(\R)^*$ consisting of nilpotent elements,
so it is a finite union of nilpotent orbits.

By \cite[Theorem A]{matumoto}  $X\in \WF(\pi)$. 
Furthermore  by \cite[Theorem 1.2]{harris} $\WF(\pi)=\AC(\O_\pi)$.
Therefore  $X\in \AC(\O)$, so  $\K(X)$ meets $\O_\pi$ by 
Lemma \ref{lem:g1}.

For the opposite implication consider the Harish-Chandra parameter $(S,\tau)$ of $\pi$. Recall that $\Omega(S,G)$ acts on $S$ by $\R$-automorphisms, due to the ellipticity of $S$. Applying an element of $\Omega(S,G)$ to $\tau$ we obtain $\tau_1$ with the property that the element $H_1 \in i\mf{s}(\R)$ corresponding to $d\tau_1$ under $\kappa$ lies in a Weyl chamber whose simple roots are non-compact (such Weyl chambers exist). By \cite[Theorem 6.2(a,f)]{Vog78} the corresponding representation $\pi_1$ is generic with respect to some Whittaker datum. By Lemmas \ref{lem:g1} and \ref{lem:g2} we can apply an element of $\Omega_\R(S_\tx{ad},G_\tx{ad})$ to $\tau_1$ to obtain $\tau_2$, such that the corresponding $\pi_2$ is $\w_X$-generic. By construction there is $n \in N(S,G)(\C)$ such that $\tau_2=\tx{Ad}(n)\tau$.

By the already proved first implication of this proposition,
$\O_{\pi_2}$ meets $K(X)$. By assumption $\O_\pi$ also meets
$K(X)$. Thus, the $G(\R)$-orbits of $\tau$ and
$\tau_2=\tx{Ad}(n)\tau$ meet $K(X)$. Since the $G(\C)$-orbit of
$H\tau$ meets $K(X)$ in a unique point we conclude that $\tau$
and $\tau_2=\tx{Ad}(n)\tau$ lie in the same $G(\R)$-orbit,
thus (by regularity) in the same $N(S,G)(\R)$-orbit. Modifying $n$ by
$N(S,G)(\R)$ arranges $\tau=\tau_2$,
without changing $\pi_2$, and we conclude $\pi=\pi_2$.
%
%\warn{Uniqueness isn't so easy. There is a reference: \cite[Lemma 14.14]{ABV92}. If you want to include a proof, I think the best one uses the associated variety argument, which was deleted as not needed but could be put back. Then: the associated variety of a discrete series is well known (reference...) to be a single orbit, so by the bijection the same holds for the Whittaker model}
\end{proof}

\begin{cor}
	Let $\pi$ be an eds representation. There exists precisely one Whittaker model $\mf{w}$ for which $\pi$ is $\mf{w}$-generic.
\end{cor}
\begin{proof}
	Let $(S,\tau)$ be the Harish-Chandra parameter of $\pi$.
 If $\w_1,\w_2$ are two Whittaker models for which $\pi$ is generic, let $w \in \Omega_\R(S_\tx{ad},G_\tx{ad})$ be an element s.t. $w\w_1=\w_2$ by Lemmas \ref{lem:g1} and \ref{lem:g2}. Proposition \ref{p:whittaker} shows that the $G(\R)$-orbits of $\tau$ and $w\tau$ meet the same Kostant section (associated to $\w_2$.) These $G(\R)$-orbits must then be equal. Modifying $w$ by an element of $\Omega_\R(S,G)$ we see that $\tau=w\tau$. By regularity of $H_\pi$ this shows $w=1$, but then $\w_2=w\w_1=\w_1$.
\end{proof}

We now mention  a slight strengthening of Proposition \ref{p:whittaker}. 

\begin{lem} \label{lem:g2'}
Suppose $\pi=\pi(S,\tau)$ is a generic eds representation and $X\in i\g(\R)^*$ is a regular nilpotent element.
Then $\pi$ is $\w_X$-generic if and only if
$\K(X)$ meets $G(\R)\cdot\lambda$ where $\lambda\in \mf{z} \oplus i\s'(\R)^*$ is any element in the same Weyl chamber as $d\tau$.

Suppose $\pi=\pi(S,\tau)$ and $\pi'=\pi(S,\tau')$ are generic eds representations, which are $\w, \w'$-generic, respectively.
Then $\w=\w'$ if and only if $d\tau$ and $d\tau'$ are in the same Weyl chamber. 
\end{lem}

\begin{proof}
For the first statement by Lemma \ref{lem:g1} replace the condition on the Kostant section with $X\in \AC(\O_{d\tau})$.
By \cite[Proposition 3.9]{adams_afgoustidis_whittaker}, $\AC(\O_{d\tau})=\AC(\O_\lambda)$ and the result follows.

  This implies the if direction of the second statement. For the other direction see
  See \cite[Corollary  3.12]{adams_afgoustidis_whittaker}.
\end{proof}

See \cite[Appendix]{adams_afgoustidis_whittaker} for details  in the case of $SL(2,\R)$.

\section{Construction of $L$\-packet and internal structure} \label{sec:cons}

Let $G_0$ be a quasi-split connected reductive $\R$-group with dual group $\hat G$ and $L$\-group $^LG$.
%We choose an $\R$-pinning $(T,B,\{X_\alpha\})$ of $G$ and a non-trivial additive character $\psi : \R \to \C^\times$ and denote by $\mf{w}$ the resulting Whittaker datum, as in \S\ref{sec:whit}.
Let $[\varphi] : W_\R \to {^LG}$ be a $\hat G$-conjugacy class of discrete Langlands parameters.

\subsection{Factorization of a parameter} \label{sub:fac}

Before we state the next lemma, we introduce the following notation. Given $a,b \in \C$ with $a-b \in \Z$ and $z \in \C^\times$ we define
\[ z^a \cdot \bar z^b := |z|^{a+b} \cdot (z/|z|)^{a-b}. \]
More generally, if $\hat T$ is a complex torus and $\lambda,\mu \in X_*(\hat T)_\C$ with $\lambda-\mu \in X_*(\hat T)$, and $z \in \C^\times$, we define
\[ \lambda(z) \cdot \mu(\bar z) \in \hat T \]
to be the unique element characterized by 
\[ \chi(\lambda(z) \cdot \mu(\bar z)) = z^{\<\chi,\lambda\>} \cdot \bar z^{\<\chi,\mu\>},\qquad \forall \chi \in X^*(\hat T).\]
It is well known that every continuous group homomorphism $\C^\times \to \C^\times$ is of the form $z \mapsto z^a \cdot \bar z^b$ for unique $a,b \in \C$ with $a-b \in \Z$. Therefore, every continuous group homomorphism $\C^\times \to \hat T$ is of the form $z \mapsto \lambda(z) \cdot \mu(\bar z)$ for unique $\lambda,\mu \in X_*(\hat T)_\C$ with $\lambda-\mu \in X_*(\hat T)$.

\begin{lem} \label{lem:icreg}
Choose any representative $\varphi$ within the conjugacy class $[\varphi]$. 
\begin{enumerate}
\item There exists a maximal torus $\hat T \subset \hat G$ that is normalized by the image of $\varphi$. 
\item There exist $\lambda,\mu \in X_*(\hat T)_\C$ with $\lambda-\mu \in X_*(\hat T)$ such that, for all $z \in \C^\times$,
\[ \varphi(z) = \lambda(z) \cdot \mu(\bar z) \in \hat T. \]
\item The action of $\Gamma=W_\R/\C^\times$ on $\hat T$ by conjugation via $\varphi$ induces multiplication by $-1$ on $X_*(\hat T/Z(\hat G))$.
\item For all $\alpha \in R(\hat T,\hat G)$, the a-priori complex number $\<\lambda,\alpha\>$ is a non-zero half-integer, and equals $-\<\mu,\alpha\>$.
\item $\hat T=\tx{Cent}(\varphi(\C^\times),\hat G)$, thus $\hat T$ is uniquely determined by $\varphi$.
\end{enumerate}
\end{lem}

\begin{rem}
	\begin{enumerate}
		\item Note that 4. implies that the images $\lambda',\mu' \in X_*(\hat T/Z(\hat G))_\C$ lie in $\tfrac{1}{2}X_*(\hat T/Z(\hat G))$ and satisfy $\mu'=-\lambda'$. Although we will not need it, it can be shown that in fact $\lambda',\mu'$ are integral, i.e. they lie in $X_*(\hat T/Z(\hat G))$.
  		\item Furthermore, 4. implies that $\lambda'$ is a regular element of $X_*(\hat T/Z(\hat G))_\R$, and hence determines a Weyl chamber. Let $\Delta$ be the set of simple roots for that chamber. Once the integrality of $\lambda'$ is established, its regularity implies $\<\alpha,\lambda'\> \geq 1$, and hence $\lambda'-\rho \in X_*(\hat T/Z(\hat G))$ is still a dominant integral element of the chamber.
	\end{enumerate}

\end{rem}

\begin{proof}[Proof of Lemma \ref{lem:icreg}]
(1) We consider $^LG = \hat G \rtimes \Gamma$ as a disconnected algebraic group. Let $A \subset {^LG}$ denote the Zariski closure of the image of $\varphi$ and let $B \subset A$ denote the Zariski closure of $\varphi(\C^\times)$; the latter lies in $\hat G$. 

We claim that $A$ consists of semi-simple elements. Since $A^2 \subset B$ it is enough to show that $B$ consists of semi-simple elements. This is equivalent to showing that the adjoint action of $B$ on the Lie algebra of $\hat G$ is semi-simple. Thus consider $\tx{Ad} : \hat G \to \tx{GL}(\mf{\hat g})$. The subgroup $\varphi(\C^\times)$ of $\hat G$ and consists of elements which commute with each other, and are semi-simple elements by definition of $\varphi$, cf. \cite[\S8]{BorCor}. Therefore their actions on $\mf{\hat g}$ can be simultaneously diagonalized, i.e. their images in $\tx{GL}(\mf{\hat g})$ lie in a common torus. The preimage of this torus in $\hat G$ is a closed subgroup consisting of commuting semi-simple elements and contains $B$, proving the claim.

Now $B$ is normal in $A$ of index $2$, $B/B^\circ$ is a finite abelian group, and $B^\circ$ is a connected abelian algebraic groups consisting of semi-simple elements, hence a torus. We conclude that $A$ is supersolvable according to \cite[Definition 5.14]{SS70}, and \cite[Theorem 5.16]{SS70} implies that $A$ normalizes a maximal torus $\hat T \subset \hat G$.

% We now provide an alternative proof of (1).
% The image of $\C^\times$ is connected in the analytic topology, abelian, and consists of semisimple elements. It follows \warn{reference?} that $\hat L=\Cent_{\hat G}(\varphi(\C^\times))$ is a Levi subgroup.
% In particular $\hat L$ is connected, reductive, and contains a maximal torus $\hat S$ of $\hat G$. Since $\varphi(\C^\times)$ is contained in the identity component of the center of $\hat L$, it is contained in $\hat S$.
% Choose an element $g\in \varphi(W_\R)\backslash \varphi(\C^\times)$. Then $g$ normalizes $\hat L$, and since $g^2\in \varphi(\C^\times)$ conjugation by $g$ acts as an involution on $\hat L$ and is thus induces a semi-simple automorphism of $\hat L$. According to \cite[Theorem 7.5]{Ste68end}, it stablizes a Borel pair in $\hat L$.
% The maximal torus $\hat T$ in that Borel pair is a maximal torus of $\hat G$, containing $\varphi(\C^\times)$ (since the latter is contained in the center of of $\hat L$), and
% fixed by the action of $g$, hence and therefore normalized by the image of $\varphi$.

(2) We continue with a maximal torus $\hat T \subset \hat G$ normalized by the image of $\varphi$. Thus $\varphi(\C^\times)$, which is a subgroup of $\hat G$, lies in $N(\hat T,\hat G)$. By continuity of $\varphi$ for the analytic topology, the subset $\varphi(\C^\times)$ of $\hat G$ is connected in the analytic topology. The projection map $N(\hat T,\hat G) \to \Omega(\hat T,\hat G)$ is continuous in the analytic topology, from which we conclude that the image of $\varphi(\C^\times)$ in $\Omega(\hat T,\hat G)$ is trivial, and hence $B \subset \hat T$. In particular, $\varphi(z) \in \hat T$ for all $z \in \C^\times$. The discussion before the statement of the lemma provides unique $\lambda,\mu \in X_*(\hat T)\otimes_\Z\C$ with $\lambda-\mu \in X_*(\hat T)$ such that
\[ \varphi(z) = \lambda(z) \cdot \mu(\bar z),\qquad \forall z \in \C^\times. \]

For the remainder of the proof, we assume without loss of generality that $\hat G$ is adjoint.

(3) Since $\varphi(\C^\times) \subset \hat T$ by (2), the action of $W_\R$ on $\hat T$ by conjugation via $\varphi$ factors through $W_\R/\C^\times=\Gamma$. If the induced action on $X_*(\hat T)$ stabilizes some $0 \neq \nu \in X_*(\hat T)$, then the image of $\varphi$ would be contained in the proper parabolic subgroup $P_\nu$ determined by $\nu$, contradicting the assumed discreteness of $\varphi$. Therefore, the involution of $X_*(\hat T)$ induced by the action of the non-trivial element of $\Gamma$ is given by multiplication by $-1$. 

(4) Using (3) we see that, for all $z \in \C^\times$,
\[ \mu(z)\lambda(\bar z)=\varphi(\bar z)=\varphi(\sigma \cdot z \cdot \sigma^{-1})=\tx{Ad}(\varphi(\sigma))(\lambda(z)\mu(\bar z))=(-\lambda)(z) \cdot (-\mu)(\bar z), \] which shows $\lambda=-\mu$. Thus $\<\lambda,\alpha\>=-\<\mu,\alpha\>$ for all $\alpha \in R(\hat T,\hat G)$. Moreover, since $\lambda-\mu \in X_*(\hat T)$, we conclude $2\lambda \in X_*(\hat T)$.

If there is some $\alpha \in R(\hat T,\hat G)$ with $\<\lambda,\alpha\>=0$, then all elements of the root subgroups $U_\alpha$ and $U_{-\alpha}$ are fixed by $\varphi(\C^\times)$, while these two root subgroups are interchanged by $\varphi(\sigma)$ for any $\sigma \in W_\R$ projecting to the non-trivial element of $\Gamma$. The semi-simple group of rank 1 generated by $U_\alpha$ and $U_{-\alpha}$ is thus stable under the action of $W_\R$ via conjugation by $\varphi$, and this action descends to $\Gamma$ and is thus given by an involution. This involution is necessarily inner, i.e. it coincides with the conjutation action of an element of this group. The subgroup of fixed points is thus at least of dimension $1$ (it contains a maximal torus in this 3-dimensional subgroup). This subgroup lies in $\tx{Cent}(\varphi,\hat G)$ and contradicts the discreteness of $\varphi$.

(5) From (2) and (4) we have $\varphi(z)=(2\lambda)(z/|z|) \in \hat T$ and $2\lambda$ is a regular element of $X_*(\hat T)$. Since the image of $z \mapsto z/|z|$ is the unit circle $\mb{S}^1$ of $\C^\times$, which is Zariski dense in the 1-dimensional torus $\C^\times$, the centralizer of the subgroup $(2\lambda)(\mb{S}^1)$ of $\hat G$ is the same as that of $(2\lambda)(\C^\times)$, which in turn equals $\hat T$.
\end{proof}

We continue with a chosen representative $\varphi$ of the conjugacy class $[\varphi]$. Lemma \ref{lem:icreg} provides the maximal torus $\hat T \subset \hat G$ normalized by $\varphi$ and containing $\varphi(\C^\times)$.
%Conjugating $\varphi$ by $N(\hat T,\hat G)$ we may then assume that this image lies in the chamber determined by $\hat B$. This will be used in Lemma \ref{lem:a} below.
%
% At this point, $\varphi$ is well-defined up to conjugation by $\hat T$. This of course depends on the chosen Borel pair $(\hat T,\hat B)$.
%
The composition of $\varphi$ with the projection $N(\hat T,{^LG}) \to \Omega(\hat T,{^LG})=N(\hat T,{^LG})/\hat T$ factors through a homomorphism $\xi : \Gamma \to \Omega(\hat T,{^LG})$. Let $\hat S$ denote the $\Gamma$-module with underlying abelian group $\hat T$ and $\Gamma$-structure given by $\tx{Ad}\circ\xi$. Let $S$ be the $\R$-torus whose dual is $\hat S$, i.e. the $\R$-torus determined by $X^*(S)=X_*(\hat S)$ as $\Gamma$-modules.

By construction we have $R(\hat T,\hat G) \subset X^*(\hat T)=X^*(\hat S)=X_*(S)$, and we write $R^\vee(S,G)$ for this set. Analogously we have a subset $R(S,G) \subset X^*(S)$. Both of these subsets are $\Gamma$-stable and according to Lemma \ref{lem:icreg} the action of $\sigma$ on $R(S,G)$ is by negation. Thus $S/Z(G)$ is anisotropic, where $Z(G) \subset S$ is the joint kernel of all elements if $R(S,G)$.

Let $S(\R)_G$ be the double cover of $S(\R)$ reviewed in \S\ref{sub:covtori}, associated to the subset $R(S,G) \subset X^*(S)$. As discussed there, there is a canonical $\hat G$-conjugacy class of $L$\-embeddings $^LS_G \to {^LG}$. Inside of this class, there is a unique $\hat S$-conjugacy class, call it $^Lj$, whose restriction to $\hat S$ is the tautological embedding $\hat S \to \hat G$. The image of this $L$\-embedding is described in \eqref{eq:lembim}, and contains the image of $\varphi$ by construction. Thus $\varphi = {^Lj}\circ\varphi_S$ for a unique $\hat S$-conjugacy class of $L$\-homomorphisms $\varphi_S : W_\R \to {^LS_G}$. According to \cite[Theorem 3.15]{KalDC}, $\varphi_S$ corresponds to a genuine character $\tau : S(\R)_G \to \C^\times$.

Finally, the tautological inclusion $\hat T \subset \hat G$ provides an embedding $\hat\jmath : \hat S \to \hat G$. While this embedding is not $\Gamma$-equivariant, its $\hat G$-conjugacy class is, because the embedding $\hat T \to \hat G$ is $\Gamma$-equivariant and the $\Gamma$-structures of $\hat G$ and $\hat S$ differ by twisting by $\hat G$.

The construction of $(S,\tau,\hat\jmath\,)$ depended on the choice of $\varphi$ within its conjugacy class. The next lemma shows that this dependence is irrelevant.

\begin{lem} \label{lem:a}
If $(S_1,\tau_1,\hat\jmath_1)$ and $(S_2,\tau_2,\hat\jmath_2)$ are two pairs obtained from two different choices $\varphi_1$ and $\varphi_2$ of elements of the $\hat G$-conjugacy class $[\varphi]$, there exists a unique isomorphism $S_1 \to S_2$ which identifies $\tau_1$ with $\tau_2$ and whose dual intertwines $\hat\jmath_1$ and $\hat\jmath_2$. It is given by conjugation by an element of $\hat G$.
\end{lem}
\begin{proof}
	The uniqueness claim is clear from the compatibility with $\hat\jmath_i$. We show existence. Let $\hat T_i$ be the centralizer of $\varphi_i(\C^\times)$, a maximal torus according to Lemma \ref{lem:icreg}. Choose any $g \in \hat G$ such that $\tx{Ad}(g) \circ\varphi_1=\varphi_2$. Then $\tx{Ad}(g)\hat T_1=\hat T_2$ and the isomorphism $\tx{Ad}(g) : \hat T_1 \to \hat T_2$ translates the $\Gamma$-action induced by $\tx{Ad}\circ\varphi_1$ to that induced by $\tx{Ad}\circ\varphi_2$. Therefore, $\tx{Ad}(g) : \hat S_1 \to \hat S_2$ is $\Gamma$-equivariant. Tautologically, it intertwines $\hat\jmath_1$ and $\hat\jmath_2$.

	The dual isomorphism $S_2 \to S_1$ identifies the subsets $R(S_1,G) \subset X^*(S_1)$ and $R(S_2,G) \subset X^*(S_2)$, hence lifts canonically to an isomorphism of double covers $S_2(\R)_G \to S_1(\R)_G$. Dually the isomorphism $\hat S_1 \to \hat S_2$ extends canonically to an isomorphism $^LS_{1,G} \to {^LS_{2,G}}$, which commutes with the canonical $L$\-embeddings into $^LG$. Therefore, it translates the factorization of $\varphi_1$ through $^LS_{1,G}$ to the factorization of $\varphi_2$ through $^LS_{2,G}$, which implies that the isomorphism $S_2(\R)_G \to S_1(\R)_G$ identifies $\tau_2$ with $\tau_1$.
\end{proof}

% \begin{cns} \label{cns:stj}
% 	Let $(G,\xi)$ be an inner twist of $G_0$. We construct a natural stable class $J_\xi$ of embeddings $j : S \to G$. \warn{TODO} Choose, in addition to the $\Gamma$-stable Borel pair $(\hat T,\hat B)$ of $\hat G$, an $F$-Borel pair $(T_0,B_0)$ of $G_0$. The resulting identification $X_*(T_0)=X^*(\hat T)$ induces a $\Gamma$-equivariant identification of Weyl groups $\Omega(\hat T,\hat G) = \Omega(T_0,G_0)$, and we obtain the homomorphism $\xi : \Gamma \to \Omega(T,G) \rtimes \Gamma$.
% \end{cns}

\subsection{Construction of the $L$\-packet}

The natural embedding $\hat S \to \hat G$ is not $\Gamma$-equivariant, but its $\hat G$-conjugacy class is. From \S\ref{sub:adm} we obtain the category $\mc{J}$ of embeddings of $S$ into all pure (or rigid) inner forms of $G_0$.

Consider $(G,\xi,z,j) \in \mc{J}(\R)$. As discussed in \S\ref{sub:essds}, there exists a unique eds representation $\pi_j$ of $G(\R)$ associated to the pair $(S,\tau)$, transported to $G$ via $j$. According to \S\ref{sub:essds} the representation $\pi_j$ depends on the $G(\R)$-conjugacy class of $j$, and two distinct such conjugacy classes produce two non-isomorphic representations.

% The images of all elements of $J_\xi$ are elliptic maximal tori in $G$. According to \warn{ref}, these are conjugate to each other under $G(\R)$. To simplify notation we fix one elliptic maximal torus $S' \subset G$ and we write $J_\xi'$ for the subset of $J_\xi$ consisting of all elements with image $S'$. Then $J_\xi'$ is a torsor under $\Omega(S',G)(\R)$.

\begin{dfn}
We define the pure (resp rigid) compound $L$\-packet
\[ \tilde\Pi_\varphi = \{(G,\xi,z,\pi_j)|(G,\xi,z,j) \in \mc{J}(F)\} \subset \tilde\Pi, \]
\[ \Pi_\varphi=\tilde\Pi_\varphi/G_0(\C) \subset \Pi. \]
For each pure (or rigid) inner twist $(G,\xi,z)$ of $G_0$, we define
\[ \Pi_\varphi(G,\xi,z)=\{\pi | (G,\xi,z,\pi) \in \tilde\Pi_\varphi\}. \]
\end{dfn}

\begin{lem}
The set of representations $\Pi_\varphi((G,\xi,z))$ equals $\{\pi_j|j \in J^G(\R)/G(\R))\}$. In particular, it is independent of $z$. It coincides with the set $\Pi_\varphi(G)$ constructed by Langlands in \cite[\S3]{Lan89}.
\end{lem}
\begin{proof}
	As discussed in \S\ref{sub:adm}, the set of $(G,\xi,z,j) \in \mc{J}(\R)$ with fixed triple $(G,\xi,z)$ corresponds to $J^G(\R)$, in particular is independent of $z$. Since the images of the members of $J^G(\R)$ are elliptic maximal tori, and all such are conjugate under $G(\R)$, we can choose representatives of $J^G(\R)/G(\R)$ that all have the same image, call it $S' \subset G$. Then these representatives are a single orbit under $\Omega(S',G)$. In other words, if $\tau'$ is the transport of $\tau$ under one admissible embedding $j : S \to G$ with image $S'$, then all others make out the $\Omega(S',G)$-orbit of $\tau'$. Since we have specified the representations $\pi_j$ by their character values on regular elements of $S'(\R)$ in the same way as in \cite[\S3]{Lan89} or \cite[\S4]{AV16}.
\end{proof}

We have thus recovered the $L$\-packets constructed by Langlands. At the moment they do not depend on the datum $z$ in the triple $(G,\xi,z)$. This datum will play a role in the internal parameterization of these packets, to which we turn next.

\subsection{Internal structure of the compound packet} \label{sub:intstr}

Recall from Lemma \ref{lem:simtrans} that the abelian group $H^1(\Gamma,S)$ in the pure case (resp. $H^1(u \to W,Z(G_0) \to S)$ in the rigid case) acts simply transitively on the set $\mc{J}(\R)/G_0(\C)$. By construction we have a bijection $\mc{J}(\R)/G_0(\C) \to \Pi_\varphi$. At the same time, Lemma \ref{lem:icreg} provides an identification $S_\varphi = \hat S^\Gamma$, hence by Tate-Nakayama duality $\pi_0(S_\varphi)^*=\pi_0(\hat S^\Gamma)^*=H^1(\Gamma,S)$. Analogously, in the rigid setting we obtain $\pi_0(S_\varphi^+)^*=H^1(u \to W,Z(G_0) \to S)$. This provides a simply transitive action of the abelian group $\pi_0(S_\varphi)^*$ in the pure setting, and $\pi_0(S_\varphi^+)^*$ in the rigid setting, on the set $\Pi_\varphi$.

\begin{lem} \label{lem:uniqgen}
The set $\Pi_\varphi((G_0,1,1))$ contains a unique $\mf{w}$-generic member.
\end{lem}
\begin{proof}
By Lemma \ref{lem:w1} we have $\mf{w}=\mf{w}_X$ for some regular nilpotent $X \in i\mf{g}^*(\R)$. Proposition \ref{p:whittaker} states that implies that $\pi_j$ is $\mf{w}$-generic if and only if $dj(d\tau) \in \mf{g}^*$ meets the  Kostant section $\K(X)$. But as $j$ varies over $J^{G_0}(\R)/G_0(\R)$, the element $dj(d\tau)$ varies over the $G_0(\R)$-classes in a fixed stable class. Therefore, $dj(d\tau)$ meets any Kostant section for precisely one $j \in J^{G_0}(\R)/G_0(\R)$. At the same time, $j \mapsto \pi_j$ is a bijection from $j \in J^{G_0}(\R)/G_0(\R)$ to $\Pi_\varphi((G_0,1,1))$ by construction of the latter.
\end{proof}

Taking the unique $\mf{w}$-generic member of $\Pi((G_0,1,1)) \subset \Pi_\varphi$, provided by Lemma \ref{lem:uniqgen}, as a base-point, the simply-transitive action turns into the bijection from 
\[ \iota_\mf{w} : \pi_0(S_\varphi)^* \to  \Pi_\varphi, \quad \tx{resp.}\quad \iota_\mf{w} : \pi_0(S_\varphi^+)^* \to \Pi_\varphi. \]

We can summarize the construction and internal structure of $\Pi_\varphi$ as follows. We use the language of pure inner forms; that of rigid inner forms is entirely analogous. We use the simplified notation to write $\tilde\Pi$ for the set of pairs $(z,\pi)$, where $z \in Z^1(\R,G_0)$ and $\pi$ an isomorphism class of representations of the twist $G_z$, and $\Pi=\tilde\Pi/G_0(\C)$. Then $\tilde\Pi_\varphi \subset \tilde\Pi$ is the image of the map $\mc{J}(F) \to \tilde\Pi$ sending $(z,j)$ to $(z,\pi_j)$. The group $Z^1(\R,S)$ acts on $\mc{J}(\R)$ by $x\cdot (z,j)=(x \cdot z,j)$. If $j_\mf{w} : S \to G_0$ is an embedding for which $\pi_{j_\mf{w}}$ is $\mf{w}$-generic, then we obtain the map $Z^1(\R,S) \to \mc{J}(\R)$ sending $x$ to $(x,j_\mf{w})$. Composing this with the map $\mc{J}(\R) \to \tilde\Pi_\varphi$ and taking quotient under the action of $G_0(\C)$ produces the bijection $H^1(\R,S) \to \Pi_\varphi$, which, together with the isomorphism $\pi_0(S_\varphi)^*=H^1(\R,S)$ produces the desired bijection $\iota_\mf{w} : \pi_0(S_\varphi)^* \to \Pi_\varphi$.

\subsection{Dependence on the choice of Whittaker datum} \label{sub:whitdep}

In order to obtain the bijection from $\pi_0(S_\varphi)^*$ (resp,. $\pi_0(S_\varphi^+)^*)$ to $\Pi_\varphi$, we had to choose a Whittaker datum $\mf{w}$ and apply Lemma \ref{lem:uniqgen}. Another Whittaker datum is of the form $\mf{w'}=\tx{Ad}(\bar g)\mf{w}$ with $\bar g \in G_{0,\tx{ad}}(\R)$. According to Proposition \ref{p:whittaker}, if $\pi$ is $\mf{w}$-generic, then $\tx{Ad}(\bar g)\pi$ is $\mf{w'}$-generic. If $j : S \to G_0$ is the embedding with $\pi=\pi_j$, then $\tx{Ad}(\bar g)\pi_j$ is associated to the embedding $\tx{Ad}(\bar g)\circ j$.

As described in \S\ref{sub:intstr}, the bijection $\pi_0(S_\varphi)^* \to \Pi_\varphi$ is the composition of the orbit map for the action of $\pi_0(S_\varphi)^*=H^1(\R,S)$ on $\mc{J}(\R)/G_0(\C)$ through the embedding $j$ with the $G_0(\C)$-equivariant bijection $\mc{J}(F) \to \tilde\Pi_\varphi$. Neither the latter bijection nor the action of $H^1(\R,S)$ on $\mc{J}(\R)/G_0(\C)$ depend on $\mf{w}$, only the particular point $j \in \mc{J}(F)$ does. The orbit map through $j$ is given by $x \mapsto (x,j)$, while the orbit map through $\tx{Ad}(\bar g)\circ j$ is given by $x \mapsto (x,\tx{Ad}(\bar g)\circ j)$. But $(x,\tx{Ad}(\bar g)\circ j)=g(g^{-1}x\sigma(g),j)g^{-1}$, for any $g \in G_0(\C)$ lifting $\bar g \in G_{0,\tx{ad}}(\R)$.

Now $z(\sigma)=g^{-1}\sigma(g)$ belongs to $Z^1(\R,Z(G_0))$. We conclude that $(x,\tx{Ad}(\bar g),j)$ is equivalent modulo the action of $G_0(\C)$ to $(z\cdot x,j)$. This shows that the bijection $H^1(\R,S) \to \Pi_\varphi$ normalized via $\mf{w'}=\tx{Ad}(g)\mf{w}$ is obtained form the bijection normalized via $\mf{w}$ by shifting the latter via multiplication by $[z] \in H^1(\R,Z(G_0))$. 

Now consider the identification $\pi_0(S_\varphi)^* = H^1(\R,S)$. It can be extended to a commutative diagram
\[ \xymatrix{
	\pi_0(S_\varphi)^*\ar[r]^\cong&H^1(\R,S)\\
	H^1(W_\R,\hat G_\tx{sc} \to \hat G)^*\ar[u]\ar[r]^\cong&H^1(\R,Z(G_0)).\ar[u]
}
\]
We have used  the $W_\R$-cohomology of the crossed module $\hat G_\tx{sc} \to \hat G$ and the duality between it and the $\Gamma$-cohomology of the crossed module $G \to G_\tx{ad}$. The latter is however quasi-isomorphic to $Z(G_0)$. The left vertical map is obtained from the long exact cohomology sequence for the crossed module $\hat G_\tx{sc} \to \hat G$ endowed with $W_\R$-action given by $\tx{Ad}\circ\varphi$. The edge map $S_\varphi=H^0(W_\R,\varphi,\hat G) \to H^1(W_\R,\hat G_\tx{sc} \to \hat G)$ factors through $\pi_0(S_\varphi)$. Note that, the cohomology of $\hat G_\tx{sc} \to \hat G$ is canonically the same whether we take the $W_\R$-action coming from the dual groups, or the one coming from $\tx{Ad}\circ\varphi$, because the two structures differ by a homotopically trivial twist in the sense of \cite[\S2.4]{KalECI}. For an alternative construction without resorting to crossed modules we refer to \cite[\S4]{KalGen}.

If we denote by $(\mf{w},\mf{w'})$ the character of $H^1(W_\R,\hat G_\tx{sc} \to \hat G)$ corresponding to $[z]$, as well as its pull-back to $\pi_0(S_\varphi)$, then we obtain the formula
\begin{equation}
	\iota_\mf{w'}(x) = \iota_\mf{w}(x \cdot (\mf{w},\mf{w'})).
\end{equation}

\subsection{The case of a cover of $G$} \label{sub:packetcover}

We now consider $G$ quasi-split and a cover $G(\R)_x$ of $G(\R)$ coming from a character $x : \tilde\pi_1(G) \to \mu_n(\R)$ as in \cite{KalHDC}. This will be relevant for endoscopic groups $H$, where we will be working with $L$\-packets on the double cover $H(\R)_\pm$. It was shown in \cite[\S2.6]{KalHDC} that the local Langlands correspondence for such covers follows from the case of non-covers. We will give here an alternative approach that does not pass via reduction to non-covers, and instead works directly with the covers. This is possible due to the flexibility of Harish-Chandra's results on representation theory and harmonic analysis, as reviewed in \S\ref{sub:essds}.

The procedure is essentially the same as for the group $G(\R)$. We start with a discrete $L$\-parameter $\varphi : W_\R \to {^LG_x}$. By the same arguments as in \S\ref{sub:fac}, $\hat T = \tx{Cent}(\varphi(\C^\times),\hat G)$ is a maximal torus of $\hat G$, and the image of $\varphi$ lies in $N(\hat T,{^LG_x})$. The composition of $\varphi$ with the adjoint action $\tx{Ad}$ factors through $\Gamma$ and induces a $\Gamma$-structure on $\hat T$, which we call 
$\hat S$. The $\hat G$-conjugacy class of the inclusion $\hat\jmath : \hat S \to \hat G$ is $\Gamma$-stable, hence leads again to the set $\mc{J}(F)$ of admissible embeddings of $S$ into pure (or rigid) inner forms of $G$.

Let $S(\R)_G$ be the double cover of $S(F)$ associated to the subset $R(S,G) \subset X^*(S)$. It is the same double cover that was used in \S\ref{sub:fac}. Now we have another double cover, called $S(F)_x$, coming from the pull-back of $x : \tilde\pi_1(G) \to \mu_n(\C)$ under the natural map $\tilde\pi_1(S) \to \tilde\pi_1(G)$. Each admissible embedding $j : S \to G$ lifts naturally to an embedding $S(\R)_x \to G(\R)_x$.

Consider the Baer sum $S(\R)_{G,x}$ of the two double covers. We have the canonical $L$\-embedding $^LS_G \to {^LG}$. This embedding induces a canonical $L$\-embedding ${^LS_{G,x}} \to {^LG_x}$. The $L$-parameter $\varphi$ factors through it and provides an $L$-parameter $\varphi_S : W_\R \to {^LS_{G,x}}$, hence a genuine character $\tau : S(\R)_{G,x} \to \C^\times$. Thus again we have a triple $(S,\tau,\hat\jmath\,)$. The analog of Lemma \ref{lem:a} holds, with the same proof.

For each $(G,\xi,z,j) \in \mc{J}(\R)$ we use $j$ to identify $S$ with a maximal torus of $G$. We can view $\tau$ as a genuine character of a double cover of $S(\R)_x$. Then Theorem \ref{thm:eds} provides a genuine eds representation $\pi_j$ of $G(\R)_x$ whose character restricted to $S(\R)_x$ is given by \eqref{eq:charfmla}. The definition of the $L$-packet and its internal structure are now done in the same way as in \S\ref{sub:fac} and \S\ref{sub:intstr}.

\section{Endoscopic character identities} \label{sec:endo}

Let $\varphi : W_\R \to {^LG}$ be a discrete parameter. Let $s \in S_\varphi$ (resp. $s \in S_\varphi^+$) be a semi-simple element. Associated to the pair $(\varphi,s)$ is a (pure or rigid) refined endoscopic datum $(H,s,\mc{H},\eta)$ and a factorization $\varphi' : W_\R \to {^LH_\pm}$ of $\varphi$, whose construction was reviewed in \S\ref{sub:covendo}. Here $^LH_\pm$ is the $L$-group of the canonical double cover $H(F)_\pm$ of $H(F)$ associated to $(\mc{H},\eta)$.

% \subsection{Construction of endoscopic datum}

% \warn{Recall here how to construct} $(H,\mc{H},s,\eta)$.

% We first review the factored parameter using covers. Recall from \S\ref{sub:covendo} that there is a natural double cover $H(\R)_\pm$ of $H(\R)$ and a natural identification $^LH_\pm \to \mc{H}$. According to \warn{ref}, the image of $\varphi$ is contained in the image of $\eta$, hence $\varphi = \eta \circ \varphi'$ for an $L$\-parameter $\varphi' : W_\R \to {^LH}_\pm$. Note that $\varphi'$ is automatically discrete.

% We now review the factored parameter in the classical set up. For this, one attempts to choose an $L$\-isomorphism $\xi : {^LH} \to \mc{H}$. This is possible when the derived subgroup of $G$ is simply connected (cf. \cite{Lan79}), and in some other cases, but not in general. Therefore, the general strategy is to choose a $z$-extension $H_1 \to H$ and an $L$\-embedding $^L\eta : \mc{H} \to {^LH_1}$ which extends the tautological embedding $\hat H \to \hat H_1$. The composition of $\varphi' : W_\R \to \mc{H}$ with $^L\eta$ is then an $L$\-parameter $\varphi_1 : W_\R \to {^LH_1}$, again automatically discrete.

% \subsection{Review of transfer factors and the transfer theorem} \label{sub:trans}

% \warn{TODO}

% The reader can skim this and come back to it later.

\subsection{Statement of the main theorem}

As discussed in \S\ref{sub:packetcover} there is an associated $L$\-packet $\Pi_{\varphi'}(H_\pm)$ of genuine representations of $H(F)_\pm$. Consider the virtual character
\[ S\Theta_{\varphi'} := \sum_{\sigma \in \Pi_{\varphi'}(H_\pm)} \<\sigma,s\>\Theta_\sigma = \sum_{\sigma \in \Pi_{\varphi'}(H_\pm)} \<\sigma,1\>\Theta_\sigma = \sum_{\sigma \in \Pi_{\varphi'}(H_\pm)} \Theta_\sigma\]
on $H(\R)_\pm$, where $\<\sigma,-\>$ is the character of the irreducible representation of $\pi_0(S_\varphi)$ (resp. $\pi_0(S_\varphi^+)$) associated to $\sigma$ by the bijection of \S\ref{sub:intstr}. Let us argue the two equalities. Since $Z(\hat H)^\Gamma$ (resp. $Z(\hat{\bar H})^+$) acts trivially on this irreducible representation, and $s$ belongs by construction to this group, we see $\<\sigma,s\>=\<\sigma,1\>$, hence the first equality. The second comes from the fact that $S_\varphi$ is abelian, because it lies in $\hat S$ (and $S_\varphi^+$ lies in $\hat{\bar S}$), where $\hat S$ is the torus involved in the construction of the $L$\-packet on $H$. Note that, while the bijection of \S\ref{sub:intstr} depends on the choice of a Whittaker datum, the argument of \S\ref{sub:whit} shows that the value $\<\sigma,1\>$ does not depend on this choice.

We have a corresponding construction in the classical languange, where the cover $H(F)_\pm$ is replaced by an arbitrary choice of a $z$-pair $(H_1,{^L\eta_1})$ consisting of a $z$-extension $H_1 \to H$ and an $L$-embedding $^L\eta_1 : \mc{H} \to {^LH_1}$. We can then form the composed parameter $\varphi_1 = {^L\eta_1} \circ \varphi'$ and have the $L$-packet $\Pi_{\varphi_1}(H_1)$, from which we can form the virtual character
\[ S\Theta_{\varphi_1} := \sum_{\sigma \in \Pi_{\varphi_1}(H_1)} \<\sigma,s\>\Theta_\sigma = \sum_{\sigma \in \Pi_{\varphi_1}(H_1)} \<\sigma,1\>\Theta_\sigma = \sum_{\sigma \in \Pi_{\varphi_1}(H_1)} \Theta_\sigma\]
on $H_1(F)$.

Let $(G,\xi,z)$ be a pure (resp. rigid) inner twist of $G_0$. We have the virtual character on $G(\R)$ given by
\[ \Theta_{\varphi}^{\mf{w},s} := e(G)\sum_{\pi \in \Pi_\varphi((G,\xi,z))} \<\pi,s\>\Theta_\pi. \]
This virtual character does depend on $\mf{w}$.

The following is the main theorem of this article. It is a fundamental result of Shelstad \cite{She82}, \cite{SheTE2}, \cite{SheTE3}.
\begin{thm} \label{thm:main1}
Let $f \in \mc{C}^\infty_c(G(\R))$ be a test function.
\begin{enumerate}
	\item If $f^{H_\pm} \in \mc{C}^\infty_c(H(\R)_\pm)$ matches $f$ as in Definition \ref{dfn:matching}, then
	\[ \Theta_\varphi^{\mf{w},s}(f) = S\Theta_{\varphi'}(f^{H_\pm}). \]
	\item If $f^{H_1} \in \mc{C}^\infty_c(H_1(\R))$ matches $f$ as in Definition \ref{dfn:matching}, then
	\[ \Theta_\varphi^{\mf{w},s}(f) = S\Theta_{\varphi_1}(f^{H_1}). \]
\end{enumerate}
\end{thm}

\begin{rem} \label{rem:measures}
	Recall from Remark \ref{rem:matchmeasures} that the concept of matching functions depends on choices of measures for $G(\R)$, $H(\R)$, and all tori in those groups. In the above theorem the distribution $\Theta^{\mf{w},s}_\varphi$ depends on the choice of measure on $G(\R)$, and $S\Theta_{\varphi'}$ depends on the choice of measure on $H(\R)$. But the measures on the tori are do not influence these distributions. Therefore, the validity of the claimed identity assumes that the measures on the tori of $G$ and $H$ are synchronized.

	More precisely, if $T_H \subset H$ and $T \subset G$ are tori and $T_H \to T$ is an admissible isomorphism, we demand that it identifies the measures  on $T_H(\R)$ and $T(\R)$. Any two admissible isomorphisms differ by conjugation by $\Omega(T,G)(\R)$. Since this action preserves any Haar measure on $T(\R)$, the choice of admissible isomorphism is irrelevant.
\end{rem}

\subsection{Reduction to the elliptic set}

We now formulate an equivalent version of Theorem \ref{thm:main1} that involves character functions, rather than character distributions. Note that the character functions are canonical, in particular independent of choices of mesures.

\begin{thm} \label{thm:main2}
\begin{enumerate}
	\item For every strongly regular semi-simple element $\delta \in G(\R)$ the following identity holds
	\[ \Theta_\varphi^{\mf{w},s}(\delta) = \sum_{\gamma \in H(\R)/\tx{st}} \Delta[\mf{w},\mf{e},z](\dot\gamma,\delta)S\Theta_{\varphi'}(\dot\gamma). \]
	\item For every strongly regular semi-simple element $\delta \in G(\R)$ the following identity holds
	\[ \Theta_\varphi^{\mf{w},s}(\delta) = \sum_{\gamma \in H(\R)/\tx{st}} \Delta[\mf{w},\mf{e},\mf{z},z](\gamma_1,\delta)S\Theta_{\varphi_1}(\gamma_1). \]
\end{enumerate}
\end{thm}

Lemma \ref{lem:equi} shows that Theorems \ref{thm:main1} and \ref{thm:main2} are equivalent. The next lemma reduces the proof further to the set of elliptic elements.

\begin{lem} \label{lem:redell}
If Theorem \ref{thm:main2} holds for all elliptic $\delta$, then it holds for all $\delta$.
\end{lem}
\begin{proof}

  Again the proofs of 1. and 2. are essentially the same.

  \underline{Step 1:} Let $G(\R)^\natural$ be the image of $G_\tx{sc}(\R) \to G(\R)$. Both sides are supported on $Z_G(\R) \cdot G(\R)^\natural$, and transform by the same character of $Z_G(\R)$. 

  \begin{proof}
	Letting $Z_G(\R)$ act on the elliptic maximal torus of $G$ by multiplication, the assumption that the identity of Theorem \ref{thm:main2} holds for elliptic elements implies that both sides of that identity transform under the same character of $Z_G(\R)$. It thus remains to prove the support condition.	
	
	Theorem \ref{thm:eds} shows that the character of an eds representation of $G(\R)$ is supported on $Z_G(\R) \cdot G(\R)^\natural$. It follows that the function $\Theta_\varphi^{\mf{w},s}$ is supported on $Z_G(\R) \cdot G(\R)^\natural$. In the same way, we see that the function $S\Theta_{\varphi_\pm}$ is supported on $Z_{H(\R)_\pm} \cdot H(\R)_\pm^\natural$, where $H(\R)_\pm^\natural$ is the image of the natural splitting $H_\tx{sc}(\R) \to H(\R)_\pm$.

	To show that the right hand side of the identity of Theorem \ref{thm:main2} is supported on $Z_G(\R) \cdot G(\R)^\natural$, consider a strongly regular semi-simple element $\delta \in G(\R)$ for which the right-hand side is non-zero. Thus there exists $\dot\gamma \in H(\R)_\pm$ such that $S\Theta_{\varphi'}(\dot\gamma) \neq 0$ and $\Delta(\dot\gamma,\delta)\neq 0$. 
	
	Let $T \subset G$ be the centralizer of $\delta$ and let $S \subset H$ be the centralizer of the image $\gamma \in H(\R)$ of $\dot\gamma$. Since $S\Theta_{\varphi'}(\dot\gamma) \neq 0$ we know $\gamma \in Z_H(\R) \cdot S(\R)^\natural$, where $S(\R)^\natural$ is the image of $S_\tx{sc}(\R) \to S(\R)$, and $S_\tx{sc}$ is the preimage of $S$ in $H_\tx{sc}$. Since $\Delta(\dot\gamma,\delta)\neq 0$ there is a (necessarily unique) admissible isomorphism $j : S \to T$ mapping $\gamma$ to $\delta$. The induced isomorphism $X_*(S) \to X_*(T)$ identifies the coroot lattice $Q^\vee(S,H)$ with a sublattice of the coroot lattice $Q^\vee(T,G)$, and hence restricts to a homomorphism $X_*(S_\tx{sc}) \to X_*(T_\tx{sc})$. This means that $j : S(\R) \to T(\R)$ maps $S(\R)^\natural$ into $T(\R)^\natural$. This reduces to showing that $j$ maps $Z_H(\R)$ to $Z_G(\R) \cdot T(\R)^\natural$, for which it is enough to show that the map $(S/Z_G)(\R) \to (T/Z_G)(\R)$ sends $(Z_H/Z_G)(\R)$ into the image of $T_\tx{sc}(\R) \to T_\tx{ad}(\R)$; here $T_\tx{sc}$ is the preimage of $T$ in $G_\tx{sc}$ and $T_\tx{ad}$ is the image of $T$ in $G_\tx{ad}$. 
	
	If we assume that $T$ is elliptic then the claim follows from the fact that both $T_\tx{sc}$ and $T_\tx{ad}$ are anisotropic, hence their $\R$-points are connected, and hence the map $T_\tx{sc}(\R) \to T_\tx{ad}(\R)$ is surjective.

	We will reduce the general case to this special case by means of the following general consideration. Consider maximal tori $S_1,S_2 \subset H$ and $T_1,T_2 \subset G$ and let $j_i : S_i \to T_i$ be admissible isomorphisms. By definition this means that there exist $h \in H(\C)$ and $g \in G(\C)$ such that $\tx{Ad}(g)\circ j_1=j_2 \circ \tx{Ad}(h)$. Since $j_1,j_2$ are defined over $\R$ we see $\tx{Ad}(\sigma(g))\circ j_1=j_2 \circ \tx{Ad}(\sigma(h))$. Combining with the previous equation and noting that $h^{-1}\sigma(h) \in N(S_1,H)(\C)$ and $g^{-1}\sigma(g) \in N(T_1,G)(\C)$ we conclude that the action of $h^{-1}\sigma(h)$ on $S_1$ is transported via $j_1$ to the action of $g^{-1}\sigma(g)$ on $T_1$. Use $j_1$ to identify the Weyl group of $(S_1,H)$ with a subgroup of the Weyl group of $(T_1,G)$, and the root system $R(S_1,H)$ with a subset of the root system $R(T_1,G)$. Now let $z \in Z_H(\R)$ and assume that $j_1(z) \in T_1(\R)$ lies in $Z_G(\R)\cdot T_1(\R)^\natural$. We claim that then $j_2(z) \in T_2(\R)$ lies in $Z_G(\R)\cdot T_2(\R)^\natural$. To see this, write $j_1(z)=x\bar y_1$ with $x \in Z_G(\R)$ and $\bar y_1 \in T_1(\R)$ the image of $y_1 \in T_{1,\tx{sc}}(\R)$. Since $j_2(z)=j_2(\tx{Ad}(h)z)=\tx{Ad}(g)j_1(z)$ we see that $j_2(z)=x\bar y_2$ with $y_2 = \tx{Ad}(g)y_1 \in T_{2,\tx{sc}}(\C)$. It is enough to check that $y_2 \in T_{2,\tx{sc}}(\R)$. But $\sigma(y_2)=\tx{Ad}(\sigma(g))y_1=\tx{Ad}(g)\tx{Ad}(g^{-1}\sigma(g))y_1$ and it is enough to check that the action of $g^{-1}\sigma(g) \in N(T_1,G)(\C)$ on $T_{1,\tx{sc}}$ fixes $y_1$. We have established that the image of $g^{-1}\sigma(g)$ in the Weyl group of $(T_1,G)$ lies in the Weyl group of $(S_1,H)$, and therefore is a product of simple reflections for roots in the subset $R(S_1,H)$ of $R(T_1,G)$. Therefore, the Weyl element $g^{-1}\sigma(g)$ fixes every element of $T_{1,\tx{sc}}(\C)$ that is killed by all elements of $R(S_1,H) \subset R(T_1,G)$. But $y_1$ is indeed killed by all elements of $R(S_1,H)$, because $\bar y_1$ is.

	We apply this general consideration to the case $S_1=S$, $T_1=T$, $j_1=j$, $S_2$ the maximal elliptic torus in $H$, $T_2$ the maximal elliptic torus in $G$, and $j_2 : S_2 \to T_2$ any admissible isomorphism.
  \end{proof}

With Step 1 complete, it is thus enough to compare both sides of the identity of Theorem \ref{thm:main2} after pulling back to $G_\tx{sc}(\R)$. The pull-back of $\Theta_{\varphi}^{\mf{w},s}$ to $G_\tx{sc}(\R)$ is equal to the analogous function for that group. Indeed, the $L$-group of $G_\tx{sc}$ is $^LG/Z(\hat G)$, so we can compose $\varphi$ with the projection modulo $Z(\hat G)$ and map $s$ modulo $Z(\hat G)$. The Whittaker datum $\mf{w}$ for $G$ produces one for $G_\tx{sc}$ by pull-back, and the function $\Theta_{\varphi}^{\mf{w},s}$ for $G_\tx{sc}$ equals the pullback to $G_\tx{sc}$ of the function $\Theta_{\varphi}^{\mf{w},s}$ for $G$.

The same argument applies to the right-hand side of the identity of Theorem \ref{thm:main2}. We can replace $\mc{H}$ by $\mc{H}/Z(\hat G)$ to obtain an endoscopic pair for $G_\tx{sc}$. The transfer factors for that pair and $G_\tx{sc}$ are the pull-back of the transfer factors for $G$.

This allows us to assume that $G$ is semi-simple and simply connected in what follows, and hence apply the results of \cite{HCDSI}.

  \underline{Step 2}: Both sides are invariant eigendistributions, with the same infinitesimal character, whose values on the regular set are bounded.  

	\begin{proof}

		It is known \cite[Theorem 3]{HCDSI} that the un-normalized character of a discrete series representation $\pi$ of a $G(\R)$ is represented by a conjugation-invariant function, denoted by $\Theta$ in loc. cit., which is an eigenfunction for the center $\mf{z}$ of the universal enveloping algebra and such that the normalized function $|D_T|^{1/2}\Theta$ is bounded. Let us explain more precisely the eigenfunction property. Let $(S,\tau)$ be the Harish-Chandra parameter of $\pi$. Let $\gamma : \mf{z} \to S(\mf{s})^{\Omega(S,G)}$ be the Harish-Chandra isomorphism. The differential $d\tau : \mf{s} \to \C$ is a $\C$-linear form and induces an algebra character $S(\mf{s}) \to \C$. The eigenfunction property is $z \cdot \Theta = d\tau(\gamma(z)) \cdot \Theta$.

		These statements carry over to the left-hand side of the identity of Theorem \ref{thm:main2}, which is just a complex linear combination of such functions, but we need to be mindful that we are using here normalized characters. Therefore, these properties now read
		\begin{enumerate}
			\item $z \cdot (|D(-)|^{-1/2}\Theta_\varphi^{\mf{w},s}) = d\tau(\gamma(z)) \cdot (|D(-)|^{-1/2}\Theta_\varphi^{\mf{w},s})$ for all $z \in \mf{z}$.
			\item $\sup_{x \in G(\R)_\tx{rs}}|\Theta_\varphi^{\mf{w},s}(x)|<\infty$.
		\end{enumerate}

		We now claim that the right-hand side of the identity of Theorem \ref{thm:main2} is also has those properties. Invariance (under conjugation by $G(\R)$) follows from the corresponding property of the transfer factor $\Delta(\dot\gamma,\delta)$ in the variable $\delta$. 

		To see boundedness, we note that for any fixed $\delta$, there are only finitely many $\dot\gamma$ with $\Delta(\dot\gamma,\delta) \neq 0$, and in fact their number is bounded by $|\Omega(T,G)|$. The values of the transfer factor are roots of unity (\cite[Lemma 4.3.3]{KalHDC}). The boundedness of the right-hand side now follows from the boundedness of $S\Theta_{\varphi'}$, the latter function being the analog of $\Theta_\varphi^{\mf{w},s}$ for the group $H$ and $s=1$.

		To check the eigendistribution property for the function $\delta \mapsto \sum_\gamma \Delta(\dot\gamma,\delta)S\Theta_{\varphi'}(\dot\gamma)$ we fix $\delta$ and let $T \subset G$ be its centralizer. To apply a differential operator to this function at $\delta$, it is enough to consider the function in a neighborhood of $\delta$. Let $\gamma_1\dots,\gamma_n$ be representatives for the stable classes in $H(\R)$ that transfer to the stable class of $\delta$ and let $S_1,\dots,S_n \subset H$ be their centralizers. Let $j_i : T \to S_i$ be the unique admissible isomorphism with $j_i(\delta)=\gamma_i$. There is a connected open neighborhood $U \subset T(\R)$ of $\delta$ so that, for each $\delta' \in U$, $j_1(\delta'),\dots,j_n(\delta')$ is a set of representatives for the stable classes in $H(\R)$ that transfer to $\delta'$, and such that the double cover $S_i(\R)_\pm$ splits over $j_i(U)$. We let $U_i \subset S_i(\R)_\pm$ be an arbitrary choice of one of the two connected components of the preimage of $j_i(U)$ in $S_i(\R)_\pm$. Then $j_i$ lifts to a diffeomorphism $U \to U_i$, which we still call $j_i$. The restriction to $U$ of the function $\delta \mapsto \sum_\gamma \Delta(\dot\gamma,\delta)S\Theta_{\varphi'}(\dot\gamma)$ is now given by $\delta' \mapsto \sum_{i=1}^n \Delta(j_i(\delta'),\delta')S\Theta_{\varphi'}(j_i(\delta'))$. 
		
		It will be enough to consider each summand individually. By \cite[Corollary 4.3.4]{KalHDC}, the function $\delta' \mapsto \Delta(j_i(\delta'),\delta')$ is locally constant, so after possibly shrinking $U$ we may assume that it is constant. Thus the summand under consideration becomes $\delta' \mapsto S\Theta_{\varphi'}(j_i(\delta'))$, a function defined on a connected open neighborhood $U$ in $G(\R)$ of the fixed strongly regular semi-simple element $\delta$.

		We apply Proposition \ref{pro:diff} to the smooth class function 
		\[ f_i(\delta') = |D_T(\delta')|^{-1/2}S\Theta_{\varphi'}(j_i(\delta')) \]
		and see that for $z \in \mf{z}$ we have
		\begin{equation} \label{eq:e1}
		(z \cdot f_i)|_{T(\R) \cap U} = |D_T(-)|^{-1/2}\gamma(z) \cdot (S\Theta_{\varphi'}\circ j_i)|_{T(\R) \cap U}.
		\end{equation}
		The admissible isomorphism $j_i : T \to S_i$ induces an isomorphism $S(\mf{t}) \to S(\mf{s}_i)$ which embeds $S(\mf{t})^{\Omega(T,G)}$ into $S(\mf{s}_i)^{\Omega(S_i,H)}$. We have
		\begin{equation} \label{eq:e2}
		\gamma(z) \cdot (S\Theta_{\varphi'}\circ j_i)|_{T(\R) \cap U} = (j_i(\gamma(z)) \cdot S\Theta_{\varphi'}|_{S_i(\R)_\pm \cap U_i}) \circ j_i.
		\end{equation}
		Let $\gamma_i : \mf{z}_H \to S(\mf{s}_i)^{\Omega(S_i,H)}$ be the Harish-Chandra isomorphism for $(H,S_i)$ and let $z_H \in \mf{z}_H$ be the preimage under $\gamma_i$ of $j_i(\gamma(z))$. We apply again Proposition \ref{pro:diff}, this time to the smooth class function $|D_{S_i}(-)|^{-1/2}S\Theta_{\varphi'}$ on $H(\R)_\pm$ to see
		\begin{equation} \label{eq:e3}
		j_i(\gamma(z)) \cdot S\Theta_{\varphi'}|_{S_i(\R)_\pm \cap U_i} = |D_{S_i}(-)|^{1/2}\cdot z_H \cdot (|D_{S_i}(-)|^{-1/2}S\Theta_{\varphi'})|_{S_i(\R)_\pm \cap U_i}.
		\end{equation}
		The eigenfunction property of $S\Theta_{\varphi'}$ gives
		\begin{equation} \label{eq:e4}
		z_H \cdot (|D_{S_i}(-)|^{-1/2}S\Theta_{\varphi'}) = d\tau'(\gamma'(z_H)) \cdot (|D_{S_i}(-)|^{-1/2}S\Theta_{\varphi'}),
		\end{equation}
		where $(S',\tau')$ is the Harish-Chandra parameter of any element in the $L$-packet $\Pi_{\varphi'}$, unique up to stable conjugacy, and $\gamma' : \mf{z}_H \to S(\mf{s'})^{\Omega(S',H)}$ is the Harish-Chandra isomorphism. The scalar $d\tau'(\gamma'(z_H))$ can now switch places with the Weyl discriminant and combining \eqref{eq:e2}, \eqref{eq:e3}, and \eqref{eq:e4}, we obtain
		\begin{equation} \label{eq:e5}
		\gamma(z) \cdot (S\Theta_{\varphi'}\circ j_i) = d\tau'(\gamma'(z_H)) \cdot S\Theta_{\varphi'} \circ j_i 
		\end{equation}
		as functions on $T(\R) \cap U$.

		By definition, the two Harish-Chandra isomorphisms $\gamma_i$ and $\gamma'$ are related by composition by $\tx{Ad}(h)$ for any $h \in H(\C)$ with $\tx{Ad}(h)S_i=S'$. Note that the $\R$-structure here is irrelevant, because the Harish-Chandra homomorphism is defined over $\C$. Furthermore, transporting $d\tau'$ under $\tx{Ad}(h)$ and then under the isomorphism $j_i$ produces $d\tau$ 
		%\warn{do we need more detail here?}. 
		With this we obtain $d\tau'(\gamma'(z_H))=d\tau(\gamma(z))$. Combining this with \eqref{eq:e1} and \eqref{eq:e5} we obtain
		\[  z \cdot (|D_T(-)|^{-1/2}S\Theta_{\varphi'}\circ j_i) = d\tau(\gamma(z)) \cdot (|D_T(-)|^{-1/2}S\Theta_{\varphi'}\circ j_i). \] 
		The desired eigendistribution property of the function $\delta \mapsto \sum_\gamma \Delta(\dot\gamma,\delta)S\Theta_{\varphi'}(\dot\gamma)$ has thus been proved.
	\end{proof}

	\underline{Step 3:} We have shown that both sides of the identity of Theorem \ref{thm:main2} are invariant functions on $G(\R)_\tx{sr}$, transform under the same character of $\mf{z}$, and are bounded (having already been normalized). We are assuming that they take equal values on all elliptic elements. Therefore \cite[Lemma 44]{HCDSI} implies that they are equal.
\end{proof}

\subsection{The left hand side}

% Let $\pi_\mf{w} \in \Pi_\varphi(G_0)$ be the unique $\mf{w}$-generic constituent. It corresponds to a triple $(S,\rho,\tau)$, where $S \subset G_0$ is an elliptic maximal torus, $\rho$ is a Weyl chamber in $X^*(S/Z_G)$, and $\tau$ is a character of $S(\R)$ whose differential is $\rho$-dominant. This triple is unique up to $G(\R)$-conjugacy.

% Fix based $\chi$-data for $(S,\rho)$ and consider the resulting $\hat G$-conjugacy class of embeddings ${^LS} \to {^LG}$. There exists a member of this class whose image contains the image of $\varphi$, and such that the factored parameter $\varphi_S : W_\R \to {^LS}$ gives rise to the character $\tau$. This embedding is well-defined up to conjugation by $\hat S$. It gives an isomorphism $\hat S^\Gamma \to S_\varphi$ that is independent of all choices (it depends on the choice of $\varphi$ within its $\hat G$-conjugacy class, but so does $S_\varphi$, and the whole situation is independent of this choice in the obvious way).

In this subsection we will provide a formula for the left hand side of the identity in Theorem \ref{thm:main2}, i.e. $\Theta^\mf{w}_{\varphi,s}(\delta)$, for strongly regular semi-simple elliptic $\delta \in G(\R)$. The end result is \eqref{eq:lhs}.

The members of $\Pi_\varphi((G,\xi,z))$ are parameterized by the set of $G(\R)$-conjugacy classes of admissible embedding $j : S \to G$. Given such an embedding let $\pi_j$ be the corresponding representation. Let $j_\mf{w} : S \to G_0$ be the unique embedding for which $\pi_{j_\mf{w}}$ is the unique $\mf{w}$-generic member of $\Pi_\varphi((G_0,1,1))$. Then $\tx{inv}(j_\mf{w},j) \in H^1(\Gamma,S)=\pi_0(\hat S^\Gamma)^*=\pi_0(S_\varphi)^*$ equals $\rho_{\pi_j}$. Therefore the left hand side becomes
\[ \Theta^\mf{w}_{\varphi,s}(\delta)=e(G)\sum_j \<s,\tx{inv}(j_\mf{w},j)\>\Theta_{\pi_j}(\delta), \]
where $j$ runs over (a set of representatives for) the set of $G(\R)$-conjugacy classes in $J_\xi$. Harish-Chandra's character formula \eqref{eq:charfmla} states
\[ \Theta_{\pi_j}(\delta) = (-1)^{q(G)}\sum_{w \in W_\R(G,jS)}\frac{\tau'}{d_\tau'}(j^{-1}w^{-1}\delta),\]
where we have conjugated $\delta$ within $G(\R)$ to land in $jS(\R)$. Combining the two formulas and using $e(G)=(-1)^{q(G_0)-q(G)}$, we obtain
\[ \Theta^\mf{w}_{\varphi,s}(\delta) = (-1)^{q(G_0)} \sum_j \<s,\tx{inv}(j_\mf{w},j)\> \sum_{w \in W_\R(G,jS)}\frac{\tau'}{d_\tau'}(j^{-1}w^{-1}\delta). \]

% The product in the denominator can be rewritten as
% \[ \prod\limits_{\substack{\alpha \in R(S,G)\\ \<\alpha,d\tau\>>0}}(1- \alpha(j^{-1}w^{-1}\delta)^{-1}). \]
Instead of conjugating $\delta$ to land in $jS$, we can conjugate $j$ by $G(\R)$ to achieve this, without changing $\pi_j$. With this shift in point of view we can combine the two sums and arrive at
\[ \Theta^\mf{w}_{\varphi,s}(\delta) = (-1)^{q(G_0)} \sum_j \<s,\tx{inv}(j_\mf{w},j)\> \frac{\tau'}{d_\tau'}(j^{-1}w^{-1}\delta), \]
where now the sum runs over the set of those $j \in J_\xi$ whose image contains $\delta$.

As $j$ runs over this set, $j_\mf{w}j^{-1}(\delta)$ runs over the set of elements $\delta_0 \in S_\mf{w}(\R)$ that are stably conjugate to $\delta$, where $S_\mf{w} \subset G_0$ is the image of $j_\mf{w}$, an elliptic maximal torus of $G_0$. Moreover, $j_\mf{w}$ transports $\tx{inv}(j_\mf{w},j) \in H^1(\R,S)$ to $\tx{inv}(\delta_0,\delta) \in H^1(\R,S_\mf{w})$. So we arrive at
\begin{equation} \label{eq:lhs}
\Theta^\mf{w}_{\varphi,s}(\delta) = (-1)^{q(G_0)} \sum_{\delta_0} \<s_\mf{w},\tx{inv}(\delta_0,\delta)\> \frac{\tau_\mf{w}'}{d_{\mf{w}}'}(\delta_0),
\end{equation}
where the sum runs over the set of elements $\delta_0 \in S_\mf{w}(\R)$ that are stably conjugate to $\delta$, and we have used the subscript $\mf{w}$ to indicate various transports under $j_\mf{w} : S \to S_\mf{w}$.

\subsection{The right hand side: covers} \label{sub:rhs_cover}

In this subsection we will show that the right hand side of the identity of Theorem \ref{thm:main2}(1) is also equal to \eqref{eq:lhs}.

We begin by applying \eqref{eq:lhs} to the group $H$, the parameter $\varphi'$, and the trivial endoscopic element, and obtain
\[ S\Theta_{\varphi'}(\dot\gamma) = \sum_{\dot\gamma_0} \frac{\tau_H}{d_{H}}(\dot\gamma_0), \]
where we have fixed an embedding $j_H : S \to H$ for which the corresponding eds representation of $H(\R)_\pm$ is generic with respect to some Whittaker datum and denote by subscript $H$ the various transports under $j_H$, $\dot\gamma_0$ runs over the elements of $S_H(\R)_\pm$ that are $H$-stably conjugate to $\dot\gamma$, and we have used $\tau_H/d_H=\tau'_H/d'_H$.

The right hand side of Theorem \ref{thm:main2}(1) then becomes
\[ \sum_{\gamma \in H(\R)/\tx{st}} \Delta[\mf{w},\mf{e},z](\dot\gamma,\delta)\sum_{\dot\gamma_0} \frac{\tau_H}{d_{H}}(\dot\gamma_0). \]
The first sum runs over elements of $H(\R)$ that are related to $\delta$, up to stable conjugacy under $H$. Each such stable conjugacy class consists of regular semi-simple elliptic elements (because $\delta$ is such), and hence intersects $S_H(\R)$. The second sum runs over elements $\dot\gamma_0 \in S_H(\R)_\pm$ that lie in the $H$-stable class of the lift $\dot\gamma \in S_H(\R)_\pm$ of $\gamma$. Since $\Delta$ is $H$-stably invariant in the first factor, its values at $\dot\gamma$ and $\dot\gamma_0$ are the same. We can combine the two sums together and obtain
\[ \sum_{\gamma_0 \in S_H(\R)} \Delta[\mf{w},\mf{e},z](\dot\gamma_0,\delta) \frac{\tau_H}{d_{H}}(\dot\gamma_0), \]
where now $\gamma_0$ runs over all elements of $S_H(\R)$, equivalently all those that are related to $\delta$, since the transfer factor vanishes for the others.

Having fixed the embeddings $j_\mf{w}$ and $j_H$, they provide an isomorphism $S_\mf{w} \to S_H$, and this isomorphism induces a bijection
\[ \delta_0 \leftrightarrow \gamma_0 \]
between the set of elements of $S_\mf{w}(\R)$ that are stably conjugate to $\delta$ and the set of elements of $S_H(\R)$ related to $\delta$. Using the basic property \eqref{eq:tfstab} of transfer factors we obtain
\begin{equation} \label{eq:rhs1}
\sum_{\delta_0 \in S_\mf{w}(\R)} \Delta[\mf{w},\mf{e},z](\dot\gamma_0,\delta_0)\<s_\mf{w},\tx{inv}(\delta_0,\delta)\> \frac{\tau_H}{d_{H}}(\dot\gamma_0),
\end{equation}
where $\delta_0$ runs over the elements of $S_\mf{w}(\R)$ that are stably conjugate to $\delta$, $\gamma_0 \in S_H(\R)$ denotes the element corresponding to $\delta_0$ under above bijection, and $\dot\gamma_0 \in S_H(\R)_\pm$ is an arbitrary lift of $\gamma_0$.

We now unpack the transfer factor. It is given as
\[ \Delta(\dot\gamma_0,\delta_0) = \epsilon\Delta_I^{-1}(\dot\gamma_0,\dot\delta_0)\Delta_{III}(\dot\gamma_0,\dot\delta_0), \]
where we recall that the term $\Delta_{IV}$ is missing because we are working with normalized characters and orbital integrals, and $\dot\delta_0 \in S_\mf{w}(\R)_{G/H}$ is an arbitrary lift of $\delta_0$.
%Here $\dot\delta_0 \in S_\mf{w}(\R)_\pm$ is the image of $\dot\gamma_0$ under the isomorphism $S_\mf{w}(\R)_\pm \to S_H(\R)_\pm$, and $\ddot\delta_0$ is a lift of this element to $S_\mf{w}(R)_\pm \times_{S_\mf{w}(\R)} S_\mf{w}(\R)_{G/H}$.

We claim that
\[ \Delta_{III}(\dot\gamma_0,\dot\delta_0) = \frac{\tau_\mf{w}}{\tau_H}(\dot\delta_0). \]
This is the property \eqref{eq:tfd3} that was reviewed in an idelized situation. We say here a few more words about the actual situation in the setting of double covers. We have the commutative diagram
\[ \xymatrix{
	&^LS_{H,\pm}\ar[dd]_a\ar[r]&^LH_\pm\ar[dd]\\
	W_\R\ar[ru]^{\varphi^H_S}\ar@/^4pc/[rru]_-{\varphi^H}\ar[rd]_{\varphi_S}\ar@/_4pc/[rrd]_-{\varphi}\\
	&^LS_G\ar[r]&^LG,
}
\]
Here $^LS_G$ is the $L$\-group of the double cover $S(\R)_G$ and the bottom horizontal map is the canonical $L$\-embedding. We have written $^LS_{H,\pm}$ for the $L$\-group of the double cover of $S_H(\R)$ obtained as the Baer sum of the canonical double cover $S_H(\R)_H$ coming from $R(S_H,H) \subset X^*(S_H)$ and the double cover $S_H(\R)_\pm$ that is the preimage of $S_H(\R)$ in the double cover $H(\R)_\pm \to H(\R)$. The top horizontal map is the canonical $L$\-embedding of the $L$\-group of $S_H(\R)_H$ into $^LH$, which remains an $L$\-embedding after passing to the $\pm$-covers. The right vertical map is the canonical $L$\-embedding $^LH_\pm \to {^LG}$. The map $a$ is the unique map making the square, hence also the triangle, commute. It is the paremter of a genuine character of the double cover of $S(\R)$ obtained as the Baer sum of $S_H(\R)_H$, $S_H(\R)_\pm$, and $S(\R)_G$, where we have identified $S_H$ with $S$. This Baer sum is the same as for $S_H(\R)_\pm$ and $S(\R)_{G/H}$. The claimed identity now follows.

The following result completes the proof of Theorem \ref{thm:main2}(1).

\begin{pro} \label{pro:magic}
For any $\dot\delta_0 \in S_\mf{w}(\R)_{G/H}$, the following identity holds
\[  \Delta_{I}(\dot\gamma_0,\dot\delta_0) = \epsilon\cdot(-1)^{q(G_0)-q(H)}\cdot\prod_{\substack{\alpha \in R(S,G/H)\\ \<\alpha^\vee,d\tau_\mf{w}\>>0}}\arg(\alpha^{1/2}(\dot\delta_0) - \alpha^{-1/2}(\dot\delta_0)), \]
where $\dot\gamma_0 \in S_H(\R)_{G/H}$ is the image of $\dot\delta_0$ under the isomorphism $j_H \circ j_\mf{w}^{-1}$, and we have chosen a pinning $\mc{P}$ and a unitary character $\Lambda : \R \to \C^\times$ that together produce the fixed Whittaker datum $\mf{w}$ as in \S\ref{sub:whit}, and we have used $\Lambda$ for $\epsilon$ and $\mc{P}$ for $\Delta_{I}(\dot\gamma_0,\dot\delta_0)$.
\end{pro}
\begin{proof}
Since $d\tau_\mf{w}$ is a regular element of $X_*(S)_\R$ the condition $\<\alpha^\vee,d\tau_\mf{w}\>>0$ defines a positive system $R(S,G)^+$ of roots in $R(S,G)$. We will write $\alpha>0$ when $\alpha \in R(S,G)$ belongs to that positive system. We will write $R(S,G/H)^+ = R(S,G/H) \cap R(S,G)^+$.

	We first investigate how both sides vary as functions of $\dot\delta_0$. Consider another strongly regular $\dot\delta_1$. Replacing $\dot\delta_0$ with $\dot\delta_1$ in the right hand side results in multiplication by $\prod_\alpha \tx{arg}(b_\alpha)$, where the product runs again over $R(S,G/H)^+$ and $b_\alpha=(\dot\delta_{1,\alpha}-\dot\delta_{1,-\alpha})/(\dot\delta_{0,\alpha}-\dot\delta_{0,-\alpha})$. By construction $\dot\delta_{0,-\alpha}=\sigma(\dot\delta_{0,\alpha})$, and the same holds for $\dot\delta_1$, from which follows $b_\alpha \in \R^\times$, and hence $\tx{arg}(b_\alpha)=\tx{sgn}(b_\alpha)$.

We now look at the left hand side. Replacing $\dot\delta_0$ by $\dot\delta_1$ multiplies $\tx{inv}(\dot\delta_0,\tx{pin})$ by $\prod_{\alpha>0,\sigma\alpha<0}\alpha^\vee(b_\alpha)$, and hence $\Delta_{I}$ by the Tate-Nakayama pairing of this 1-cocycle with the endoscopic element $s_\mf{w}$. Since the torus $S$ is elliptic, the conditions $\alpha>0$ and $\sigma\alpha<0$ are equivalent, and the value of the 1-cocycle at $\sigma$ equals $\prod_{\alpha>0}\alpha^\vee(b_\alpha)$. This is the product over $\alpha>0$ of the images of the 1-cocycles $b_\alpha \in Z^1(\Gamma,R_{\C/R}^1\mb{G}_m)$ under the homomorphisms $\alpha^\vee : R_{\C/R}^1\mb{G}_m \to S$, so the change in $\Delta_{I,\pm}$ is given by $\prod_{\alpha>0}\<b_\alpha,s_\alpha\>$, where $s_\alpha$ is the image of $s \in \hat S$ under $\hat\alpha : \hat S \to \C^\times$. This is a Galois-equivariant homomorphism, with $\sigma$ acting as inversion on $\C^\times$. Since $s$ is $\sigma$-fixed, so is $s_\alpha$, i.e. $s_\alpha \in \{\pm1\} \subset \C^\times$. By construction of the endoscopic group $H$, we have $s_\alpha=1$ precisely for $\alpha \in R(S,H)$. On the other hand, when $s_\alpha=-1$, then $\<b_\alpha,s_\alpha\>=\tx{sgn}(b_\alpha)$.

We have thus shown that both sides of the identity multiply by the same factor upon replacing $\dot\delta_0$ by a different element. To establish the identity we may thus evaluate at an arbitrary element $\dot\delta_0$. Our approach will be to choose this element so that both sides are equal to $1$.

To construct $\dot\delta_0$ 
we shall use the exponential map $\tx{Lie}(S_\mf{w})(\R) \to S_\mf{w}(\R)_{\pm\pm}$ discussed in \cite[\S3.7]{KalDC}. Write $\mf{s}=\tx{Lie}(S_\mf{w})$ and $\mf{s}'=\mf{s} \cap \mf{g'}$, where $\mf{g}'=[\g,\g]$.  Consider $d\tau_\mf{w} = d\tau_\mf{z}+d\tau' \in \mf{z}^* \oplus i\mf{s}'(\R)^*$ as in Remark \ref{rem:hcpar}. Let $H_{\mf{w}} \in i\mf{s}'(\R)$ be the element corresponding to $d\tau'$ via the Killing form on $\mf{g}'$, i.e. $\kappa(-,H_\mf{w})=d\tau'$. Then $-irH_\mf{w} \in \mf{s}'(\R)$ for any $r \in \R$. We will choose $r>0$ small enough, to be specified in a moment, and set $\ddot\delta_0=\exp(-irH_\mf{w}) \in S_\mf{w}(\R)_{\pm\pm}$. Let $\dot\delta_0$ be the image of $\ddot\delta_0$ in $S_\mf{w}(\R)_{\pm}$. We will now show that both sides take the value $1$ at this particular element.

We first evaluate the left hand side. By definition $\Delta_{I}(\dot\gamma_0,\dot\delta_0)$ is given by pairing the invariant $\tx{inv}(\ddot\delta_0,\mc{P})$ with the endoscopic element, see \cite[\S 4.1, \S4.3]{KalHDC}. It will be enough to show that $\tx{inv}(\ddot\delta_0,\mc{P})=0$. By \cite[Lemma 4.1.4]{KalHDC} we have $\tx{inv}(\ddot\delta_0,\mc{P})=\tx{inv}(-irH_\mf{w},\mc{P})$ provided $r$ is small enough.

Write $\mc{P}=(T_0,B_0,\{X_\alpha\})$, which, together with $\Lambda(x)=e^{ix}$, produces the chosen Whittaker datum $\mf{w}$. According to Lemma \ref{lem:w1}(3), if we define $X \in \mf{n}(\R)^*$ by $X(Y) = \sum_\alpha [Y_\alpha,X_{-\alpha}]/H_\alpha$, then $\mf{w}=\mf{w}_{iX}$. Proposition \ref{p:whittaker} implies that $d\tau'$ is $G(\R)$-conjugate to an element of $\K(iX)'$, and hence $H_\mf{w}$ is $G(\R)$-conjugate to an element of $\K(iX_\kappa)'$, where $X_\kappa \in \mf{g}'(\R)$ is the element such that $\kappa(X_\kappa,Y)=X(Y)$ for all $Y \in \mf{g}$. We can compute $X_\kappa$ as the element of $i\mf{\bar n}(\R)$ given by
\[ X_\kappa = \sum_\alpha \frac{\kappa(\alpha,\alpha)}{2} X_{-\alpha}. \]
Therefore, $-irH_\mf{w}$ is $G(\R)$-conjugate to the Kostant section $K(rX_\kappa)$, which means that $\tx{inv}(-irH_\mf{w},\mc{P}_\kappa)=0$, where $\mc{P}_\kappa$ is the pinning $(T_0,B_0,\{r_\alpha X_\alpha\})$ with $r_\alpha = \tfrac{r}{2}\kappa(\alpha,\alpha)$. It is well-known that $\kappa(\alpha,\alpha)$ is a positive rational number, see \cite[\S8.5]{Humphreys80}, and that the assignment $\alpha \to r_\alpha$ is invariant under complex conjugation. Therefore the unique element  $t \in T_{0,\tx{ad}}(\R)$ such that $\alpha(t)=r_\alpha$ for all $B_0$-simple roots $\alpha$ lies in the identity component $T_{0,\tx{ad}}(\R)^\circ$ of $T_{0,\tx{ad}}(\R)$. The isogeny $T_{0,\tx{sc}} \to T_{0,\tx{ad}}$ induces a surjective map $T_{0,\tx{sc}}(\R)^\circ \to T_{0,\tx{ad}}(\R)^\circ$, so the element $t$ lifts to $T_{0,\tx{sc}}(\R)$. Thus the pinnings $\mc{P}$ and $\mc{P}_\kappa$ are $G(\R)$-conjugate. According to \cite[Lemma 4.1.4]{KalHDC} we have $\tx{inv}(-irH_\mf{w},\mc{P})=\tx{inv}(-irH_\mf{w},\mc{P}_\kappa)=0$, and conclude that $\tx{inv}(\ddot\delta_0,\mc{P})$ as desired.

We now turn to the right hand side. We have for each $\alpha \in R(S,G)$
\[ \alpha^{1/2}(\dot\delta_0)=\dot\delta_{0,\alpha} = e^{d\alpha(-irH_\mf{w})/2}=e^{-ir\<\alpha^\vee,d\tau_\mf{w}\>/2},\]
hence
\[ \alpha^{1/2}(\dot\delta_0)-\alpha^{-1/2}(\dot\delta_0) = -2i\sin(r\<\alpha^\vee,d\tau_\mf{w}\>/2). \]
Choosing $r>0$ so that $0<r\<d\alpha,d\tau_\mf{w}\>/2 < \pi$ for all $\alpha>0$ we obtain
\[ \prod_{\alpha \in R(S,G/H)^+}\arg(\alpha^{1/2}(\dot\delta_0) - \alpha^{-1/2}(\dot\delta_0)) = (-i)^{\#R(S,G/H)^+}. \]
According to Lemma \ref{lem:epsilon} the right hand side equals $1$.
\end{proof}

\subsection{The right hand side: classical set-up}

In this subsection we will show that the right hand side of the identity of Theorem \ref{thm:main2}(2) is also equal to \eqref{eq:lhs}. The initial arguments of \S\ref{sub:rhs_cover}, which applied to the right hand side of the identity in Theorem \ref{thm:main2}(1), have direct analogs in the setting of Therem \ref{thm:main2}(2) and show that the right hand side of that identity equals the analog of the expression \eqref{eq:rhs1}, which is given by
\begin{equation} \label{eq:rhs2}
\sum_{\delta_0 \in S_\mf{w}(\R)} \Delta[\mf{w},\mf{e},z](\gamma_1,\delta_0)\<s_\mf{w},\tx{inv}(\delta_0,\delta)\> \frac{\tau'_{H_1}}{d'_{H_1}}(\gamma_1),	
\end{equation}
where we have used the parameter $\varphi_1$ of the z-extension $H_1$ to obtain the character $\tau_{H_1}'$ of the torus $S_{H_1}(\R)$.

It is the handling of the transfer factor that is slightly different. Indeed, in this setting without covers the transfer factor is given by
\[ \Delta = \epsilon\Delta_I^{-1}\Delta_{II}\Delta_{III_2}. \]
The construction of the pieces involves a choice of an admissible isomorphism $S^H \to S_\mf{w}$, which we take to be $j_\mf{w}\circ j_H^{-1}$, as we did in \S\ref{sub:rhs_cover}. It further involves choices of $\chi$-data and $a$-data for $S_\mf{w}$. We take $\rho_\mf{w}$-based $\chi$-data, and $(-\rho_\mf{w})$-based $a$-data, so that $\chi_\alpha(x)=\arg(x)$ when $\alpha>0$ and $a_\alpha=i$ when $\alpha<0$, and where $\alpha>0$ means $\<\alpha^\vee,d\tau_\mf{w}\>>0$.

We claim that
\[ \Delta_{III_2}(\gamma_1,\delta_0) = \frac{\tau'_\mf{w}(\delta_0)}{\tau'_{H_1}(\gamma_1)}.\]
This was the property \eqref{eq:tfd3} that was reviewed in an idealized situation. We say here a few more words about the actual situation, still assuming that the $z$-pair is trivial, i.e. there exists an $L$\-isomorphism $^L\eta : {^LH} \to \mc{H}$, the general case being entirely analogous by requiring more cumbersome notation. We then have the commutative diagram
\[ \xymatrix{
	&^LS_H\ar[dd]_a\ar[r]^{\rho^H}&^LH\ar[dd]^{^L\eta}\\
	W_\R\ar[ru]^{\varphi^H_S}\ar@/^4pc/[rru]_-{\varphi^H}\ar[rd]_{\varphi_S}\ar@/_4pc/[rrd]_-{\varphi}\\
	&^LS_\mf{w}\ar[r]^{\rho_\mf{w}}&^LG,
}
\]
where the horizontal arrows are the $L$\-embeddings obtained via $\rho_\mf{w}$-based and $\rho^H$-based $\chi$-data, respectively. We have $\Delta_{III_2}(\gamma_1,\delta_0)=\<a,\delta_0\>$, where $a \in Z^1(W_\R,\hat S)$ is the $1$-cocycle that makes the above diagram commute, the pairing is the Langlands pairing, and $\delta_0 \in S(\R)$ is the image of $\gamma_0 \in S^H(\R)$ to $S(\R)$ under the chosen fixed admissible isomorphism. The claim now follows from the above commutative diagram and the fact that $\tau$ and $\tau^H$ are the characters with parameters $\varphi_S$ and $\varphi_S^H$.

Next we consider
\[ \frac{\Delta_{II}(\gamma_0,\delta)}{d'_H(\gamma_0)}.\]
By definition,
\[ \Delta_{II}(\gamma_0,\delta) = \frac{\Delta_{II}^G(\gamma_0,\delta)}{\Delta_{II}^H(\gamma_0,\delta)}.\]
With the chosen $a$-data and $\chi$-data we have
\[ \Delta_{II}^H(\gamma_0)=\prod\limits_{\substack{\alpha \in R(S^H,H)\\ \<\alpha,\rho^H\>>0}}\tx{arg}\Big(\frac{\alpha(\gamma_0)-1}{-i}\Big) = i^{\#R(S^H,H)/2}\cdot\prod\limits_{\substack{\alpha \in R(S^H,H)\\ \<\alpha,\rho^H\>>0}}\tx{arg}(\alpha(\gamma_0)-1). \]
On the other hand,
\[ (\alpha(\gamma_0)-1)(1-\alpha(\gamma_0)^{-1})=\alpha(\gamma_0)+\alpha(\gamma_0)^{-1}-2=2(\tx{Re}(\alpha(\gamma_0)-1) < 0.\]
Hence
\[ \Delta_{II}^H(\gamma_0,\delta)d_H'(\gamma_0) = (-i)^{\#R(S^H,H)/2}. \]
In the same way one shows
\[ \Delta_{II}^G(\gamma_0,\delta)d_{G_0}'(\delta_0)) = (-i)^{\#R(S,G_0)/2}. \]
and we conclude
\[ \frac{\Delta_{II}(\gamma_0,\delta)}{d_H'(\gamma_0)} = \frac{i^{\#R(S^H,H)/2-\#R(S,G_0)/2}}{d_{G_0}'(\delta_0)}.\]
With this \eqref{eq:rhs2} becomes
\[ (-1)^{q(H)}i^{\#R(S^H,H)/2-\#R(S,G_0)/2}\epsilon \sum_{\delta_0 \in S_\mf{w}(\R)}\Delta_I(\gamma_0,\delta)^{-1} \<s_\mf{w},\tx{inv}(\delta_0,\delta)\> \cdot\frac{\tau_\mf{w}(\delta_0)}{d'_{G_0}(\delta_0)}. \]
The terms $\epsilon$ and $\Delta_I(\gamma_0,\delta)$ depend on the choice of a pinning $\mc{P}$ and unitary character $\Lambda : \R \to \C^\times$. These choices are to be made so that $\mf{w}=\mf{w}_{\mc{P},\Lambda}$. We are free to take $\Lambda(x)=e^{ix}$ and then $\mc{P}$ is determined by $\mf{w}$. 

We have
\[ \Delta_{I}(\gamma_0,\delta)^{-1} = \<s,\lambda\>^{-1} \]
where $\lambda \in H^1(\Gamma,S)$ is the splitting invariant of $S$ relative to the chosen $a$-data and $\mc{P}$. Using Lemma \ref{lem:epsilon} we simplify \eqref{eq:rhs2} to
\[(-1)^{q(G_0)}\<s,\lambda\>^{-1}\sum_{\delta_0}\<s,\tx{inv}(\delta_0,\delta)\>\cdot\frac{\tau_\mf{w}(\delta_0)}{d_{G_0}'(\delta_0)}. \]

The following lemma completes the proof.

\begin{lem} \label{lem:gen}
The splitting invariant $\lambda$ of $S$ computed in terms of $(-\rho_\mf{w})$-based $a$-data and the pinning $\mc{P}$ is trivial.
\end{lem}
\begin{proof}
Let $(S_\mf{w},\tau_\mf{w})$ be the Harish-Chandra parameter of the generic member of $\Pi_\varphi((G_0,1,1))$, as in Remark \ref{rem:hcpar}, and write $d\tau_\mf{w}=d\tau_\mf{z}+d\tau' \in \mf{z}^* \oplus i\mf{s}'(\R)^*$, where $\mf{s}=\tx{Lie}(S_\mf{w})$, $\mf{s}'=\mf{s} \cap [\mf{g},\mf{g}]$. Let $H_\mf{w} \in i\mf{s}'(\R)$ be the element determined by $\kappa(H_\mf{w},-)=\<d\tau_\mf{w},-\>$ with respect to the Killing form $\kappa$ on $\mf{g}'=[\g,\g]$. Write $\mc{P}=(T_0,B_0,\{X_\alpha\})$.

If we define $X \in \mf{n}(\R)^*$ by $X(Y) = \sum_\alpha [Y_\alpha,X_{-\alpha}]/H_\alpha$, then Lemma \ref{lem:w1}(3) shows that $\mf{w}=\mf{w}_{iX}$. Proposition \ref{p:whittaker} implies that $d\tau'$ is $G(\R)$-conjugate to an element of $\K(iX)'$, and hence $H_\mf{w}$ is $G(\R)$-conjugate to an element of $\K(iX_\kappa)'$, where $X_\kappa \in \mf{g}'(\R)$ is the element such that $\kappa(X_\kappa,Y)=X(Y)$ for all $Y \in \mf{g}$. We can compute $X_\kappa$ as the element of $i\mf{\bar n}(\R)$ given by
\[ X_\kappa = \sum_\alpha \frac{\kappa(\alpha,\alpha)}{2} X_{-\alpha}. \]
Therefore, $-iH_\mf{w}$ is $G(\R)$-conjugate to the Kostant section $K(X_\kappa)$. According to \cite[Theorem 5.1]{Kot99} the splitting invariant computed in terms of the $a$-data $d\alpha(-iH_\mf{w})$ and the pinning $\mc{P}_\kappa=(T_0,B_0,\{r_\alpha X_\alpha\})$, where $r_\alpha = \tfrac{r}{2}\kappa(\alpha,\alpha)$, is trivial.

For every $\alpha>0$ the complex number $d\alpha(-iH_\mf{w})$ is a positive real multiple of $-i$. At the same time it is well-known that $\kappa(\alpha,\alpha)$ is a positive rational number, see \cite[\S8.5]{Humphreys80}, and that the assignment $\alpha \to r_\alpha$ is invariant under complex conjugation. Applying \cite[Lemma 5.1]{KalGen} we conclude the triviality of the splitting invariant computed in terms of the pinning $\mc{P}$ and in terms of $a$-data obtained from the $-\rho_\mf{w}$-based $a$-data by rescaling each element by a \emph{positive} real number. Such a rescaling doesn't change the splitting invariant, see \cite[(2.3.2)]{LS87}. 
\end{proof}

\bibliographystyle{amsalpha}
%\bibliography{/Users/kaletha/Work/TexMain/bibliography.bib}
\bibliography{bibliography}

\end{document}